\documentclass[11pt]{amsart}

\RequirePackage[OT1]{fontenc}
\RequirePackage{amsthm,amsmath,mathrsfs}
\RequirePackage[numbers]{natbib}
\usepackage[colorlinks=true,linkcolor = blue, citecolor=blue, urlcolor=blue]{hyperref}
\usepackage[utf8]{inputenc}
\usepackage{nicefrac,dsfont,bbm,amstext,amsfonts,amssymb,verbatim}
\usepackage{floatrow,rotating}
\usepackage{multicol}
\usepackage{multirow,tabularx,lscape}
\usepackage{booktabs}
\usepackage{comment}
\usepackage{enumitem}

 \newtheorem{thm}{Theorem}[section]
 \newtheorem{cor}[thm]{Corollary}
 \newtheorem{lem}[thm]{Lemma}
 \newtheorem{prop}[thm]{Proposition}

 \theoremstyle{definition}
 
 \newtheorem{rem}[thm]{Remark}
 \newtheorem{ex}[thm]{Example}
 \newtheorem{ass}[thm]{Assumption}
 \newtheorem{conv}[thm]{Convention}

\numberwithin{equation}{section}

\DeclareFontEncoding{FMS}{}{}
\DeclareFontSubstitution{FMS}{futm}{m}{n}
\DeclareFontEncoding{FMX}{}{}
\DeclareFontSubstitution{FMX}{futm}{m}{n}
\DeclareSymbolFont{fouriersymbols}{FMS}{futm}{m}{n}
\DeclareSymbolFont{fourierlargesymbols}{FMX}{futm}{m}{n}
\DeclareMathDelimiter{\VERT}{\mathord}{fouriersymbols}{152}{fourierlargesymbols}{147}

\SetLabelAlign{parright}{\parbox[t]{\labelwidth}{\raggedleft#1}}
\setlist[description]{style=multiline,topsep=10pt,%
    align=parright}

 \newcommand{\ind}{\mathbbm{1}}

 \newcommand{\R}{\mathds{R}}
 \newcommand{\Z}{\mathds{Z}}

 \newcommand{\ic}{\mathrm{i}}

 \newcommand{\E}{\mathds{E}}

 \newcommand{\N}{\mathds{N}}
 
 \newcommand{\F}{\mathcal{F}}
 \newcommand{\OO}{\mathcal{O}}
 
 \newcommand{\oo}{\mbox{\scriptsize $\mathcal{O}$}}

 \newcommand{\MM}{\mathds{M}}

 \newcommand{\m}{{m}}
 \newcommand{\md}{\mathfrak{m}}
 
 \newcommand{\ad}{\mathfrak{a}}
 \newcommand{\bd}{\mathfrak{b}}
 \newcommand{\cd}{\mathfrak{c}}

 \newcommand{\sd}{\mathfrak{s}}

 \newcommand{\FF}{\mathcal{F}}
 \newcommand{\co}{\mathrm{c}}
 \newcommand{\Co}{\mathrm{C}}
 \newcommand{\Do}{\mathrm{D}}
 \newcommand{\ff}{{f}}
 \newcommand{\hh}{{h}}
 \newcommand{\Td}{\mathfrak{T}}
 \newcommand{\Cd}{\mathfrak{C}}
 \newcommand{\Id}{\mathfrak{I}}
 \newcommand{\Ad}{\mathfrak{A}}
 \newcommand{\kapp}{\kappa}

\renewcommand{\P}{\mathbb{P}}
\renewcommand{\ss}{\mathfrak{s}}

\newcommand{\Aone}{{\bf (A1)}}
\newcommand{\Atwo}{{\bf (A2)}}
\newcommand{\Athree}{{\bf (A3)}}

\newcommand{\Tone}{{\bf (T1)}}
\newcommand{\Ttwo}{{\bf (T2)}}
\newcommand{\Tthree}{{\bf (T3)}}

\newcommand{\Bone}{{\bf (B1)}}
\newcommand{\Btwo}{{\bf (B2)}}
\newcommand{\Bthree}{{\bf (B3)}}

\usepackage{epstopdf}
\begin{document}

\title[]{Sharp connections between Berry-Esseen characteristics and Edgeworth expansions for stationary processes}
\author{Moritz Jirak}
\email[Moritz Jirak]{moritz.jirak@univie.ac.at}
\author{Wei Biao Wu}
\email[Wei Biao Wu]{wbwu@galton.uchicago.edu}
\author{Ou Zhao}
\email[Ou Zhao]{ouzhao@stat.sc.edu}

\begin{abstract}
Given a weakly dependent stationary process, we describe the transition between a Berry-Esseen bound and a second order Edgeworth expansion in terms of the Berry-Esseen characteristic. This characteristic is sharp: We show that Edgeworth expansions are valid if and only if the Berry-Esseen characteristic is of a certain magnitude.
If this is not the case, we still get an optimal Berry-Esseen bound, thus describing the exact transition. We also obtain (fractional) expansions given $3 < p \leq 4$ moments, where a similar transition occurs. Corresponding results also hold for the Wasserstein metric $W_1$, where a related, integrated characteristic turns out to be optimal.
As an application, we establish novel weak Edgeworth expansion and CLTs in $L^p$ and $W_1$. As another application, we show that a large class of high dimensional linear statistics admit Edgeworth expansions without any smoothness constraints, that is, no non-lattice condition or related is necessary.  In all results, the necessary weak-dependence assumptions are very mild. In particular, we show that many prominent dynamical systems and models from time series analysis are within our framework, giving rise to many new results in these areas.
\end{abstract}

\maketitle








\tableofcontents

\section{Introduction}

Consider a strictly stationary sequence $(X_k)_{k \in \Z}$ of real-valued random variables with $\E X_k = 0$ and $\E X_k^2 < \infty$. If the sequence exhibits weak dependence in a certain sense, then the distribution \hypertarget{sn:eq2}{of}
\begin{align*}
n^{-1/2}S_n, \,\, \text{where $S_n = X_1 + X_2 + \ldots + X_n$},
\end{align*}
is asymptotically normal, see for instance ~\cite{maxwellwoodroofe} and the references therein. This fact has made the central limit theorem one of the most important tools in probability theory and statistics. On the other hand, it was already noticed by Chebyshev ~\cite{chebyshev1890} and Edgeworth ~\cite{Edgeworth1894} that normal approximations can be improved in terms \hypertarget{psin:eq2}{of} (Edgeworth) expansions $\Psi_n$ (see \eqref{defn_PSI_m} below), implying the approximation
\begin{align}\label{eq:edge}
\sup_{x \in \R}\big|\P\big(S_n \leq x\sqrt{n}\big) - \Psi_n(x)\big| = \oo\big(n^{-\frac{1}{2}}\big) \quad \text{(or even better).}
\end{align}

Motivated by applications in actuarial science, Cram\'{e}r gave rigorous proofs in ~\cite{cramer1928}, and ever since, Edgeworth expansions have been an indispensable tool in actuarial science and finance, see for instance ~\cite{alos_2006_fin_and_stoch}, ~\cite{fouque_2003} among others for applications regarding option and derivative pricing.
Edgeworth expansions and local limit theorems also play an important role in dynamical system theory, we refer to ~\cite{jacopoc_de_simoi_liverani_2018inv}, ~\cite{kasun:2018:edgeworth} and ~\cite{gouezel2009duke} for more contemporary results and accounts.
On the other hand, in a very influential work, Efron ~\cite{efron1979} broadened the view on resampling techniques (e.g. bootstrapping) and demonstrated their significant superior performance compared to normal approximations, see ~\cite{efron_1982}, ~\cite{hall_book_boot}, ~\cite{lahiri_book_2003} for an overview. Not surprisingly, the key tools for analysing, and, in particular, showing superiority of resampling methods, are again Edgeworth expansions.

As is well-known, approximation \eqref{eq:edge} is not for free and requires additional assumptions compared to a Berry-Esseen bound of order $n^{-1/2}$. If $(X_k)_{k \in \Z}$ are independent, a well developed theory is given in ~\cite{Bhattacharya_rao_1976_reprint_2010}, ~\cite{petrov_book_sums}. However, if (weak) dependence is present, the theory becomes much more complicated. In a Markovian (and dynamical systems) context, based on the spectral method initiated by ~\cite{nagaev:1959}, the validity of Edgeworth expansions has been shown by many authors subject to various additional assumptions, e.g. ~\cite{gouezel2009duke},~\cite{herve_pene_2010_bulletin},~\cite{jacopoc_de_simoi_liverani_2018inv} and ~\cite{meyn_kontoyiannis_2003aap}, ~\cite{kasun:2018:edgeworth}, ~\cite{kasun:pene:2020expansions} to name a few, but this list is by no means exhaustive. Results without relying on a Markovian structure are less common and were obtained, among others, in ~\cite{goetze_hipp_1983}, ~\cite{heinrich_1990}, ~\cite{lahiri_1993}. Particularly the work of ~\cite{goetze_hipp_1983} is considered as a breakthrough. A key assumption in all those results is a conditional Cram\'{e}r type condition, a version of which can be stated \hypertarget{gm:eq2}{as}
\begin{align}\label{cramer_condi}
\sup_{|u| \geq \mathrm{c}}\E \Big|\E\big[e^{\mathrm{i} u S_{2m}}\big| \mathcal{G}_m \big]\Big| \leq e^{-\mathrm{c}}, \quad \mathrm{c}^{-1} < m < n,
\end{align}
where $\mathrm{c} > 0$ is a constant and $\mathcal{G}_m$ are special $\sigma$-algebras. On one hand, \eqref{cramer_condi} allows for a modification of the Tikhomirov-Stein method, established in ~\cite{tikhomirov_1981}, to obtain Edgeworth expansions of any order $s \in \N$. On the other hand, \eqref{cramer_condi} is known to be suboptimal for truly dependent sequences (cf. ~\cite{goetze_Hipp_1989_filed_ptrf}). For i.i.d. sequences, \eqref{cramer_condi} is equivalent to the classical Cram\'{e}r condition, which is also suboptimal, see e.g. ~\cite{angst2017} for more recent advances.  Despite its significance, when and how Edgeworth expansions are valid for weakly dependent sequences has been an open problem for decades. The only exception appears to be ~\cite{goetze_Hipp_1989_filed_ptrf}, where a transition is described in the case of $m$-dependent ($m$ fixed) potential functions. Note that the method of proof in ~\cite{goetze_Hipp_1989_filed_ptrf} crucially hinges on the fact that $m$ is fixed and finite.

In this paper, we provide a neat solution to this open problem in the context of general, weakly dependent Bernoulli-shift sequences\footnote{The results can be extended to more general K-automorphisms subject to a related weak-dependence condition, see Section \ref{sec_K_aut} for details.}. In particular, we describe the exact transition between a Berry-Esseen bound and a second-order Edgeworth expansion. Our novel approach is fundamentally different from previous attempts (see Section \ref{sec_int_approach} for a brief overview), and works in a very general setup. Among other things, we show \hypertarget{psin:eq3}{that}
\begin{align*}
\sup_{x \in \R}\big|\P\big(S_n \leq x\sqrt{n}\big) - \Psi_n(x)\big| = \oo\big(n^{-\frac{1}{2}}\big) \quad  \text{iff $\Cd_{T_n} = \oo\big(n^{-\frac{1}{2}}\big)$},
\end{align*}
where $\Cd_{T_n}$ is the \textit{Berry-Esseen characteristic}, defined in \eqref{berry_esseen_char} below. Moreover, \hypertarget{littlelamkp:eq3}{we} only require the very mild algebraic decay assumption
\begin{align}\label{Bernoulli_condi_int}
\sum_{k = 1}^{\infty} k^2 \lambda_{k,p} < \infty, \quad p \geq 3,
\end{align}
where $\lambda_{k,p}$ are well-established dependence coefficients in the literature, see Section \ref{sec_main} for details. Typically, even if a conditional Cram\'{e}r condition is assumed to be valid, a geometric decay of the dependence coefficients is required in the literature (as is the case in the works mentioned above), ruling out many important processes, for instance linear processes with dependent innovations, functions of linear processes or Volterra processes. Algebraic results we are aware of, in this context, are ~\cite{jensen_1989_edge_markov}, ~\cite{lahiri_1996}, which, apart from a conditional Cram\'{e}r condition, require higher moment assumptions depending on the rate of decay, leading to rather strong conditions. A geometric decay typically allows for simpler arguments, as the additional loss is only of logarithmic order.  In contrast, algebraic bounds require more refined arguments, in particular, we find that relying only on \eqref{Bernoulli_condi_int} is a rather difficult problem. In the following, we describe additional features of our results.

\subsection{Fractional expansions}
Given $3 < p < 4$ non-integer moments, apart from the recent result in ~\cite{kasun:2018:edgeworth}\footnote{Requires an exponential type of mixing, and does not yield the sharp transition.}, our results appear to be the first fractional expansions for weakly-dependent sequences with almost optimal error terms, even in the case of $m$-dependent ($m$ fixed) Bernoulli-shift sequences. In addition, we describe the exact transition. Here, 'fraction' refers to the fact that $p$ is not an integer. Note that even for i.i.d. sequences, fractional expansions are not so easy to obtain, see ~\cite{Bobkov2017_survey},~\cite{petrov_book_sums}. The additional difficulty arises from the fact that the improvement can no longer solely be gained by 'matching moments', as is the case if $p$ is an integer.

\subsection{The Wasserstein metric $W_1$}
More recently, the Wasserstein distance $W_p$, $p \geq 1$ has attracted considerable attention, often in connection with Stein's method. Central limit theorems, alongside convergence rates, are given in ~\cite{chen_book_stein_2011}, ~\cite{dedecker_rio_mean_2008}, ~\cite{zhai_2017} to name a few. For independent sequences, Edgeworth expansions \hypertarget{scttn:eq4}{are} considered in ~\cite{Bobkov2017_wasser}. For the special case $W_1$, we are able to transfer our results and show that second-order Edgeworth expansions are valid if and only if the \textit{integrated characteristic} $\Id_{T_n}$, defined in \eqref{wasserstein_char}, satisfies $\Id_{T_n} = \oo(n^{-1/2})$.

{
\subsection{High dimensional linear statistics}\label{se_int_application_examples}

Given a separable Hilbert space $\mathds{H}$, many important statistics and tests in a (broad) high dimensional context are based on
\begin{align*}
\langle \theta, S_n \rangle, \quad \theta, S_n \in \mathds{H},
\end{align*}
see Section \ref{sec:applicationII} for some key examples. We show that under some natural conditions, statistics of this type admit Edgeworth expansions without any smoothness condition. In particular, no (conditional) Cram\'{e}r condition is necessary.
}

\subsection{Smooth functions and weak Edgeworth expansions}\label{se_int_weak_edge}

In a series of papers ~\cite{mykland1993aop},~\cite{mykland1992}, ~\cite{mykland1995}, Mykland (see also ~\cite{woodroofe1989}) derived Edgeworth expansions for $\E f(n^{-1/2}S_n)$ in a test function topology ($o_2$) for martingales subject to some regularity conditions, and gave many examples for their usefulness. For a function $f:\R \to \R$, this means, among other things, that the second derivative exists and is uniformly bounded. Related results have been established in ~\cite{goetze_hipp_1983}, ~\cite{lahiri_1993}, where also the second (or third) derivative is a minimum requirement, see also ~\cite{breuillard2005} and ~\cite{kasun:2018:edgeworth}, ~\cite{kasun:pene:2020expansions} for some recent results.

Being able to describe the exact transition allows us to considerably strengthen this result and recover the phenomenon from the i.i.d. case: If $f:\R \to \R$ possesses a continuous and uniformly bounded derivative, then $\E f(n^{-1/2}S_n)$ can be approximated by an Edgeworth-corrected expectation without any additional conditions, a so-called \textit{weak Edgeworth expansion} (cf. ~\cite{woodroofe1989}). In particular, no Cram\'{e}r condition is necessary, see for instance ~\cite{Bhattacharya_rao_1976_reprint_2010}. 
Based on the characteristic $\Cd_{T_n}$, we can show an analogous result for weakly dependent Bernoulli-shift sequences satisfying \eqref{Bernoulli_condi_int}. 

{
\subsection{Examples}\label{sec:int:examples}

We discuss a number of prominent examples from the literature in Section \ref{sec:examples}, where our results apply. Apart from the novel, sharp characterization of the transition, the CLTs in $L^p$ and $W_1$ with optimal rate appear to be new as well. As examples, we discuss well-known dynamical systems, random walks and time series models.
}

\subsection{The approach}\label{sec_int_approach}

For the proof, we develop recursive arguments, linking Edgeworth expansions either to Berry-Esseen type bounds ($p = 3$), or appropriately smoothed Edgeworth expansions ($3 < p \leq 4$), allowing for iterations. A number of delicate problems need, however, to be solved along this path, since very precise bounds need to be established subject to the mild weak dependence condition \eqref{Bernoulli_condi_int}. For the fractional case $3 < p < 4$, recall that an additional difficulty arises from the fact that 'matching moments only' is not sufficient for optimality, as is the case if $p$ is an integer. Key results in the context of our recursive arguments are, among others, Lemma \ref{lem_polylipschitz_distance}, Lemma \ref{lem_taylor_expansion_smooth_function}, 
Lemma \ref{lem_diamond_Sn_expansions_I}, 
Lemma \ref{lem_taylor_expansion_smooth_function_quad} 
and Lemma \ref{lem_diamond_Sn_expansions_quad}.

\subsection{Structure and outline}
This paper is structured as follows. In Section \ref{sec_main}, all results with respect to the Kolmogorov (uniform) metric are presented. In particular, the characteristic $\Cd_{T_n}$ is introduced. Section \ref{sec_wasserstein} is devoted to analogous results with respect to the Wasserstein metric $W_1$.
Applications and examples are given in Sections \ref{sec:applicationII}, \ref{sec:appl:I} and \ref{sec:examples}.
The proofs are structured into several sections. The core technical results and notations    are given in Section \ref{sec_main_tech}, whereas the proofs are relegated to Sections \ref{sec_proof_prop_m_dep}, \ref{sec_proof_dep_it} and \ref{sec_proof_prop_quad}. Some more technical key lemmas are stated and proven in Section \ref{sec_main_lemmas}. Finally, the proofs of the main results are given in Sections \ref{sec_proof_main}-\ref{sec:proof:applicationsII}. 

\subsection{Notation}
To improve the readability, all the relevant notations are gathered in  Table \hyperlink{table:fin}{11}.

\section{Main Results: Kolmogorov Metric}\label{sec_main}

 \hypertarget{normp:eq4}{For} a random variable $X$, we write $\E X $ for expectation, $\|X\|_p$ for $\big(\E |X|^p \big)^{1/p}$, $p \geq 1$ and sometimes $\E_{\mathcal{H}} X = \E [X | \mathcal{H}]$ for the conditional expectation. In addition, $\lesssim$, $\gtrsim$, ($\thicksim$) denote (two-sided) inequalities involving an absolute multiplicative constant. For $a,b \in \R$ we put $a \vee b = \max\{a,b\}$, $a \wedge b = \min\{a,b\}$. Finally, for two random variables $X,Y$ we write $X \stackrel{d}{=} Y$ for equality in distribution.

We assume that $(X_k)_{k \in \Z}$ is a real valued Bernoulli-shift process, that is, \hypertarget{SINF:eq5}{it} can be written \hypertarget{XSUBK:eq5}{as}
\begin{align}\label{defn_bernoulli}
X_k = f(\varepsilon_k,\varepsilon_{k-1},\ldots), \quad k \in \Z, \quad f:\mathbb{S}^{\N} \to \R,
\end{align}
where $\epsilon_k \in \mathbb{S}$ are i.i.d. random variables in some measurable space $\mathbb{S}$, and $f$ is measurable. We adopt the functional dependence measure (cf. ~\cite{wu_2005}), present in the literature also in many variants (e.g. ~\cite{aue_etal_2009}, ~\cite{billingsley_1968},~\cite{goetze_hipp_1995}, ~\cite{poetscher_1999}). To \hypertarget{bigXkstar:eq5}{this} end, let $(\varepsilon_k')_{k \in \Z}$ be an independent copy of $(\varepsilon_k)_{k \in \Z}$. We denote by
\begin{align}\label{defn_couple_star}
X_k^* = f\big(\varepsilon_k,\ldots, \varepsilon_{1}, \varepsilon_{0}', \varepsilon_{-1}', \ldots\big), \quad k \in \Z,
\end{align}
the coupled version, \hypertarget{littlelamkp:eq5}{and} measure the corresponding distance \hypertarget{lambdaqp:eq5}{with}
\begin{align*}
\lambda_{k,p}=\big\|X_k-X_k^*\big\|_p, \quad \Lambda_{q,p} = \big\|X_0\big\|_p +  \sum_{k = 1}^{\infty} k^q \lambda_{k,p}, \, q \geq 0.
\end{align*}

For an extension of our setup to more general K-automorphisms, we refer to Section \ref{sec_K_aut}.

We \hypertarget{BAD:eq5}{work} under \hypertarget{sss:eq5}{the} following conditions.

\begin{ass}\label{ass_dependence_main}
For $p \geq 3$, $\ad > 2$, $(X_k)_{k\in \Z}$ satisfies
\begin{enumerate}
\item[\Aone]\label{A1} $\E |X_k|^p\log(1 + |X_k|)^{\ad} < \infty$, $\E X_k = 0$,
\item[\Atwo]\label{A2} $\Lambda_{2,p} = \|X_0\|_p + \sum_{k = 0}^{\infty}k^{2}\|X_k-X_k^*\|_p < \infty$,
\item[\Athree]\label{A3} $\ss^2 > 0$, where $\ss^2 =  \sum_{k \in \Z}\E X_0X_k$.
\end{enumerate}
\end{ass}
Here and elsewhere $\ss$ is to be understood as $\ss = |\ss|$, that is, we only consider the positive root, and the same goes for the sample variance $s_n^2$ defined below. Existence of $\ss^2$ is guaranteed by \hyperref[A2]{\Atwo} (cf. Lemma \ref{lem_sig_expressions_relations}). A lot of dynamical systems and popular time series models are within our framework, 
see, for instance ~\cite{aue_etal_2009}, ~\cite{cuny_quickly_2020},~\cite{jirak_be_aop_2016}, ~\cite{korepanov2018}, ~\cite{poetscher_1999}, ~\cite{wu_2011_asymptotic_theory} and Section \ref{sec_K_aut} and Section \ref{sec:examples} for a discussion and examples. Define \hypertarget{snsquare:eq5}{the} second and third moment (or cumulant) \hypertarget{kappan3:eq5}{as}
\begin{align}\label{defn_moments}
s_n^2 = n^{-1}\E S_n^2, \quad \kapp_n^3 = n^{-\frac{3}{2}}\E S_n^3,
\end{align}
 \hypertarget{Phi:eq5}{and} the \hypertarget{littlephi:eq5}{formal} second-order Edgeworth expansion \hypertarget{psin:eq5}{as}
\begin{align}\label{defn_PSI_m}
\Psi_n\bigl(x\bigr) = \Phi\Bigl(\frac{x}{s_n}\Bigr) + \frac{1}{6} \Big(\frac{{\kapp}_n}{s_n}\Big)^3\Bigl(1 - \frac{x^2}{s_n^2} \Bigr) \phi\Bigl(\frac{x}{s_n}\Bigr), \quad x \in \R,
\end{align}
where $\Phi$ is the distribution function of a standard normal random variable and $\phi$ its density. For $0 \leq a \leq b$, \hypertarget{TTab:eq5}{define} the \textit{Berry-Esseen tail}
\begin{align}
\Td_{a}^{b}(x) = \int_{a \leq |\xi| \leq b} e^{-\ic \xi x}\E\bigl[e^{\ic \xi S_n/\sqrt{n}} \bigr] \Bigl(1 - \frac{|\xi|}{b}\Bigr) \frac{1}{\xi} d \, \xi,
\end{align}
which arises naturally in Berry's smoothing inequality. Similarly, for $a > 0$, we consider \hypertarget{ca:eq5}{the} \textit{Berry-Esseen characteristic}
\begin{align}\label{berry_esseen_char}
\Cd_{a} = \inf_{b \geq a}\Big(\sup_{x \in \R}\bigl|\Td_{a}^{b}(x) \bigr| + 1/b\Big).
\end{align}
Observe that by selecting $a = b$, we have the trivial upper bound
\begin{align*}
\Cd_{a} \leq a^{-1}.
\end{align*}
In order to define $\Cd_a$, we use the classical smoothing kernels, others could be taken here as well (cf. ~\cite{Bhattacharya_rao_1976_reprint_2010}). In addition, as will be apparent from the results, we only define $\Cd_{a}$ up to a multiplicative absolute constant for the sake of simplicity. In the sequel, \hypertarget{bigtn:eq6}{we} mainly consider $\Cd_{a}$ with $a = T_n$, \hypertarget{Deltasubn:eq6}{where}
\begin{align}\label{defn_Tn}
T_n = {\co_T} \sqrt{n}
\end{align}
with ${\co_T} > 0$ being a constant sufficiently small. The central object of our study is the difference
\begin{align*}
\Delta_n(x) = \P\bigl(S_n \leq x \sqrt{n}\bigr) -  \Psi_n\bigl(x\bigr),
\end{align*}
which we \hypertarget{BBD:eq6}{will} consider under \hypertarget{BAD:eq6}{both} the Kolmogorov and Wasserstein distances.

\begin{thm}\label{thm_edge}
Grant Assumption \ref{ass_dependence_main}. Then for $2\bd+2 < \ad$, we have
\begin{align*}
\sup_{x \in \R}\bigl|\Delta_n(x) \bigr| \lesssim \oo\big(n^{-\frac{1}{2}}\big) (\log n)^{-\bd} + \Cd_{T_n},
\end{align*}
where $T_n$ is as in \eqref{defn_Tn}. Moreover, \hypertarget{kappa3:eq6}{the} following statements are equivalent:
\begin{description}
\item[(i)] $\sup_{x \in \R}\bigl|\Delta_n(x) \bigr| = \oo\bigl(n^{-\frac{1}{2}}\bigr)$.
\item[(ii)] $\Cd_{T_n} = \oo\bigl(n^{-\frac{1}{2}}\bigr)$.
\end{description}
All involved constants only depend on $\ad$, $\ss$ and $\Lambda_{2,3}$.
\end{thm}

\begin{rem}\label{rem_replace_s_sqrt_n}
 Lemma \ref{lem_sig_expressions_relations} shows that one may replace $s_n$ with $\ss$ as normalization. Similarly, one can replace $\kapp_n^3$ with $n^{-\frac{1}{2}}\kapp^3$, where
 $$
 \kapp^3 = \sum_{i,j \in \Z}\E X_0 X_i X_j,
 $$
  see Lemma \ref{lem_S_n_third_expectation}. This is true for all results in Sections \ref{sec_main} and \ref{sec_wasserstein}.
\end{rem}

\begin{rem}\label{rem_local_limit_theorem}
Edgeworth expansions are also an important tool for deriving local limit theorems, see e.g. ~\cite{shepp1964}. We refrain from developing this any further, just mentioning that by adopting the proof of Theorem 2 in ~\cite{shepp1964}, Theorem \ref{thm_edge} yields a local limit theorem if $\Cd_{T_n} = \oo\bigl(n^{-\frac{1}{2}}\bigr)$.
\end{rem}

Due to the trivial bound $\Cd_{T_n} \lesssim n^{-1/2}$, Theorem \ref{thm_edge} immediately yields the following corollary.

\begin{cor}\label{cor_transition}
Subject to the assumptions in Theorem \ref{thm_edge}, $\Cd_{T_n}$ with $T_n$ as in \eqref{defn_Tn} describes the exact transition between the Berry-Esseen bound $n^{-1/2}$ and a second-order Edgeworth expansion with error rate $\oo(n^{-\frac{1}{2}})$.
\end{cor}

Apart from the special case of $m$-dependent Bernoulli-shift sequences handled in ~\cite{goetze_Hipp_1989_filed_ptrf}, this appears to be the first such result for general, weakly dependent processes. Moreover, the moment condition \hyperref[A1]{\Aone} with $\ad > 2$ is almost optimal. Previous results (cf. ~\cite{goetze_hipp_1983},~\cite{lahiri_1993}) require stronger moment conditions, even though they utilize a conditional Cram\'{e}r condition and assume a geometric decay of the dependence coefficients. It should be mentioned though that ~\cite{goetze_hipp_1983},~\cite{lahiri_1993} also give expansion of any order $s \in \N$. On the other hand, for most (statistical) applications such as resampling, second-order Edgeworth expansions with an error term of magnitude $\OO(n^{-1+\delta})$, $\delta > 0$ are the relevant ones (cf. ~\cite{hall_horowitz_1996},~\cite{lahiri_book_2003} and the references mentioned in the introduction), particularly in the presence of weak dependence. In Theorems \ref{thm_edge_frac} and \ref{thm_edge_four} below, we provide such expansions.

Although the characteristic $\Cd_{T_n}$ appears quite naturally, we do not find it in the literature. At least for the i.i.d. case, one reason might be that more explicit results are desirable, and in some sense also available. Indeed, a somewhat related result is Theorem 19.1 in ~\cite{Bhattacharya_rao_1976_reprint_2010}. In ~\cite{angst2017}, the notion of a \textit{weak Cram\'{e}r condition} is introduced leading to $\Cd_{T_n} =  o(n^{-1/2})$ (and stronger), see also ~\cite{kasun:dolgo:2020expansions}. 

Verifying the conditional Cram\'{e}r condition is not an easy task, see e.g. ~\cite{hall_horowitz_1996} for a comment on this matter. It has been achieved though for special classes of (weakly-dependent) Bernoulli processes, e.g. ~\cite{goetze_hipp_1995}. In this context, it is worth mentioning that the conditional Cram\'{e}r condition (basically) implies $\Cd_{T_n} \lesssim  n^{-q}$ for any $q > 0$. The following \hypertarget{bigHk:eq7}{example}, essentially taken from ~\cite{goetze_Hipp_1989_filed_ptrf}, shows that it can fail even in \hypertarget{USUBK:eq7}{very} simple cases where $\Cd_{T_n} = 0$.
\begin{ex}\label{ex_1}
Let $(U_k)_{k \in \Z}$ be i.i.d. with $\P(U_0 = -1) = \P(U_0 = 1) = 1/2$, and $(H_k)_{k \in \Z}$ be i.i.d., independent of $(U_k)_{k \in \Z}$ with $\|H_k\|_4 < \infty$
 \hypertarget{phiHt:eq7}{and} characteristic function $\varphi_H(t)$ satisfying
\begin{align*}
\varphi_H(t) = 0 \quad \text{for $|t| > \co_{T}$},
\end{align*}
where $\co_{T}$ is the constant in \eqref{defn_Tn}. An example of such a distribution is given in Section \ref{sec_proof_dep_it}. Set $X_k = U_{k-1} + H_{k} - H_{k-1}$. In ~\cite{goetze_Hipp_1989_filed_ptrf}, failure of the conditional Cram\'{e}r condition is demonstrated. On the other hand, since $S_n = U_{0} + \ldots + U_{n-1} + H_n - H_0$, we have, with $T_n$ as in \eqref{defn_Tn}, 
\begin{align}\label{eq_example_1}
\sup_{x \in \R}\bigl|\Td_{T_n}^{N}(x) \bigr| = 0
\end{align}
for any $N \geq T_n$. We conclude $\Cd_{T_n} = 0$, and thus Theorem \ref{thm_edge} implies a second order Edgeworth expansion.
\end{ex}

Our next result provides fractional expansions for $3 < p < 4$.

\begin{thm}\label{thm_edge_frac}
Grant Assumption \ref{ass_dependence_main} with $3 < p < 4$. Then, for $\delta > 0$ arbitrarily small, we have
\begin{align*}
\sup_{x \in \R}\bigl|{\Delta}_n(x) \bigr| \lesssim n^{1-\frac{p}{2}+\delta} + \Cd_{T_n},
\end{align*}
where $T_n$ is as in \eqref{defn_Tn}. Moreover, the following statements are equivalent:
\begin{description}
\item[(i)] $\sup_{x \in \R}\bigl|\Delta_n(x) \bigr| \lesssim n^{1-\frac{p}{2}+\delta}$.
\item[(ii)] $\Cd_{T_n} \lesssim n^{1-\frac{p}{2}+\delta}$.
\end{description}
All involved constants only depend on $\delta > 0$, $p$, $\ad$, $\ss$ and $\Lambda_{2,p}$.
\end{thm}

Unlike Theorem \ref{thm_edge}, we require the additional loss $n^{\delta}$ in the rate. The proof is delicate, and requires, among other things, a precise control of the truncated partial sums. If $p = 4$, we again only obtain a logarithmic loss.

\begin{thm}\label{thm_edge_four}
Grant Assumption \ref{ass_dependence_main} with $p \geq 4$. Then
\begin{align*}
\sup_{x \in \R}\bigl|{\Delta}_n(x) \bigr| \lesssim n^{-1}(\log n)^5 + \Cd_{T_n},
\end{align*}
where $T_n$ is as in \eqref{defn_Tn}. Moreover, the following statements are equivalent:
\begin{description}
\item[(i)] $\sup_{x \in \R}\bigl|\Delta_n(x) \bigr| = \oo\big(n^{-1}\big)(\log n)^5$.
\item[(ii)] $\Cd_{T_n} = \oo\big(n^{-1}\big)(\log n)^5$.
\end{description}
All involved constants only depend on $\ad$, $\ss$ and $\Lambda_{2,4}$.
\end{thm}

\section{Main Results: Wasserstein Metric}\label{sec_wasserstein}

For two probability measures $\P_1,\P_2$, let $\mathcal{L}(\P_1,\P_2)$ be the set of probability measures on $\R^2$ with marginals $\P_1, \P_2$. The Wasserstein metric (of order one) is defined as the minimal coupling $\mathds{L}^1$-distance, \hypertarget{Wone:eq10}{that} is,
\begin{align}\label{wasserstein_1}
W_1(\P_1,\P_2) = \inf\Big\{ \int_{\R} |x-y| \P(dx, dy): \, \P \in \mathcal{L}(\P_1, \P_2)\Big\}.
\end{align}
By the Kantorovich-Rubinstein Theorem, this is equivalent \hypertarget{VVN:eq11}{to}
\begin{align}\label{wasserstein_2}
W_1(\P_1,\P_2) = \sup\Big\{\Bigl|\int_{\R} f(x)(\P_1 - \P_2)(dx)\Bigr|\,:\, f \in \mathcal{H}_1^{1} \Bigr\}.
\end{align}

Let $\mathds{V}_n$ be the (signed) measure induced by $\Psi_n$. Then a priori, the distance $W_1(\P_1, \mathds{V}_n)$ is not defined in general. In ~\cite{Bobkov2017_wasser}, generalized transport distances are introduced that also allow for signed measures. In order to maintain the original definition in terms of couplings, we replace $\Psi_n$ with a probability measure that is induced by a sequence of i.i.d. random variables. There \hypertarget{ALPHA:eq11}{are} many possible choices, \hypertarget{LSUBN:eq11}{and} a simple
 \hypertarget{BETA:eq11}{one} is the following. Let $Z$ be a \hypertarget{kappan3:eq10}{zero} mean Gaussian random variable $\mathcal{N}(0,\sigma^2)$ \hypertarget{TAUSUBN:eq11}{with} variance $\sigma^2 = s_n^2$, and $G$ follow a Gamma distribution $\Gamma(\alpha, \beta)$ with shape parameter $\alpha = s_n^2 \beta$ and rate $\beta = 2 s_n^{2} (\sqrt{n} \kapp_n)^{-3}$, independent of $Z$. 
 \hypertarget{bigMk:eq10}{Then} for $1\leq k \leq n$, \hypertarget{MMn:eq10}{let}
\begin{align}\label{defn_gauss_plus_gamma}
M_k \stackrel{d}{=}  \left\{
\begin{array}{ll}
\big(Z + G-\E G\big)/\sqrt{2} &\text{if $\kapp_n^3 > 0$},\\
\big(Z - G+\E G\big)/\sqrt{2} &\text{if $\kapp_n^3 < 0$},\\
Z &\text{if $\kapp_n^3 = 0$}
\end{array}
\right.
\end{align}
be i.i.d., and denote by $\MM_n$ the probability measure induced by $L_n = n^{-1/2}\sum_{k = 1}^n M_k$. Observe that $\E L_n^2 = s_n^2$ and $\E L_n^3 = \kapp_n^3$. Moreover, $(M_k)_{k \geq 1}$ is well defined for any $n \in \N$: Due to Lemma \ref{lem_S_n_third_expectation} {\bf (ii)} we have $|\kapp_n^3 - n^{-1/2} \kapp^3|\lesssim n^{-3/2}$, and Lemma \ref{lem_S_n_third_expectation} {\bf (iv)} implies that $\kapp^3$ exists, that is, $|\kapp^3| < \infty$. Note that $\sqrt{n}L_n \stackrel{d}{=} Z_0 \pm  (G_0 - \E G_0)$, with $Z_0,G_0$ independent and $Z_0 \stackrel{d}{=}\mathcal{N}(0,s_n^2)$, $G_0 \stackrel{d}{=}\Gamma(n \alpha, \beta/\sqrt{n})$ with $\alpha$, $\beta$ as above.

Next, we introduce the Wasserstein counterpart of the Berry-Esseen characteristic. Denote by $\tau_n = \co_{\tau} \sqrt{\log n}$ for $\co_{\tau} > 0$ large enough (corresponding to Theorem \ref{thm_wasserstein} below). We \hypertarget{scttn:eq10}{then} define \hypertarget{TTa:eq10}{the} \textit{integrated characteristic}

\begin{align}\label{wasserstein_char}
\Id_{a} = \inf_{b \geq a}\int_{-\tau_n}^{\tau_n}\Big(\big| \Td_{a}^b(x) \big| + \frac{1}{b}\Big) dx.
\end{align}
Note that we have the trivial bounds $\Id_a \leq 2\tau_n \Cd_a \leq 2 \tau_n a^{-1}$.


\begin{thm}\label{thm_wasserstein}
Grant Assumption \ref{ass_dependence_main} with $\ad > 2$. Then, for $2\bd + 2 < \ad$, we have
\begin{align*}
W_1\big(\P_{S_n/\sqrt{ n}}, \MM_n\big) \lesssim \oo\big(n^{-\frac{1}{2}}\big)(\log n)^{\frac{1}{2} - \bd} + n^{-\frac{1}{2}} \wedge \Id_{T_n},
\end{align*}
where $T_n$ is as in \eqref{defn_Tn}. If $\ad > 3$, the following statements are equivalent:
\begin{description}
\item[(i)] $W_1\big(\P_{S_n/\sqrt{ n}}, \MM_n\big) = \oo\bigl(n^{-\frac{1}{2}}\bigr)$.
\item[(ii)] $\Id_{T_n} = \oo\bigl(n^{-\frac{1}{2}}\bigr)$.
\end{description}
All involved constants only depend on $\ad$, $\ss$ and $\Lambda_{2,p}$.
\end{thm}

\begin{rem}
The proof of Theorem \ref{thm_wasserstein} reveals that $\tau_n$ can be replaced by any larger value in the definition of $\Id_{a}$ in \eqref{wasserstein_char}. Whether $\tau_n$ is of the minimal order remains an open question.
\end{rem}

We are not aware of an analogous result for i.i.d. sequences in the literature. However, we conjecture that Theorem \ref{thm_wasserstein} can be extended to $W_p$, $p > 1$, at least for i.i.d. sequences.

\begin{cor}\label{cor_clt_wasser}
Grant Assumption \ref{ass_dependence_main} with $\ad > 3$. \hypertarget{GGn:eq11}{Then}
\begin{align*}
W_1\big(\P_{S_n/\sqrt{n}}, \mathds{G}_n\big) \lesssim n^{-\frac{1}{2}},
\end{align*}
where $\mathds{G}_n$ denotes the Gaussian measure induced by $\Phi(\cdot/s_n)$.
\end{cor}

\begin{rem}
The argument, used in the proof of Theorem \ref{thm_wasserstein}, yields $\Id_{T_n} \lesssim n^{-\frac{1}{2}}$ in conjunction with Corollary \ref{cor_clt_wasser}.
\end{rem}

Corresponding results to Corollary \ref{cor_clt_wasser} have been derived in ~\cite{dedecker_rio_mean_2008}, ~\cite{pene_2005} using different dependence measures and higher moment conditions.

\begin{cor}\label{cor_clt_Lp}
Grant Assumption \ref{ass_dependence_main} with $\ad > 3$. Then for any $q \geq 1$
\begin{align*}
\int_{\R}\big|\P(S_n \leq x \sqrt{ n}) - \Phi(x/s_n)\big|^q d x \lesssim n^{-\frac{q}{2}}.
\end{align*}
\end{cor}

A related result was established in ~\cite{jirak_be_aop_2016} under the additional assumption that $(X_k)_{k \in \Z}$ are martingale differences.

Finally, we have the analogous result if $3 < p \leq 4$ moments exist.

\begin{thm}\label{thm_wasserstein_frac}
Grant Assumption \ref{ass_dependence_main} with $3 < p < 4$. Then for any $\delta > 0$
\begin{align*}
W_1\big(\P_{S_n/\sqrt{n}}, \MM_n\big) \lesssim n^{1-\frac{p}{2}+\delta} + n^{-\frac{1}{2}} \wedge \Id_{T_n},
\end{align*}
where $T_n$ is as in \eqref{defn_Tn}. In addition, the following statements are equivalent:
\begin{description}
\item[(i)] $W_1\big(\P_{S_n/\sqrt{ n}}, \MM_n\big) \lesssim n^{1-\frac{p}{2}+\delta}$.
\item[(ii)] $\Id_{T_n} \lesssim n^{1-\frac{p}{2}+\delta}$.
\end{description}
All involved constants only depend on $\delta > 0$, $p$, $\ad$, $\ss$ and $\Lambda_{2,p}$. If $p = 4$, then we may replace $n^{\delta}$ with $(\log n)^5$ in all statements.
\end{thm}

\section{More general K-automorphisms}\label{sec_K_aut}

We demanded the Bernoulli-shift representation \eqref{defn_bernoulli} and stationarity largely for notational reasons. In fact, what we really need is that
\begin{align}\label{defn_K_aut}
X_k = f_k(\varepsilon_k,\varepsilon_{k-1},\ldots), \quad k \in \Z, \quad f_k:\mathbb{S}^{\N} \to \R,
\end{align}
where $\epsilon_k \in \mathbb{S}$ are i.i.d. random variables in some measurable space $\mathbb{S}$, $f_k$ is measurable for all $k \in \Z$, and a non-degeneracy condition (essentially) ensuring $\liminf_{n\to \infty} s_n^2 >0$. The crucial difference here is the dependence of the function $f_k$ on $k$. While it is well-known in the literature that representation \eqref{defn_bernoulli} does not hold in general for any K-automorphism (e.g. \cite{ornstein_1973_example}), representation \eqref{defn_K_aut} is always valid, a consequence of Vershik's famous \textit{theorem on lacunary isomorphism}, see for instance ~\cite{emery_schachermayer_2001},~\cite{vershik_doc_transl}. A concrete example marking the difference is given in ~\cite{feldman_rudolph_1998}. While \eqref{defn_bernoulli} and \eqref{defn_K_aut} are very different from a dynamical system point of view, \hypertarget{xiklstar:eq13}{it} makes little difference for our approach. In fact, defining $\xi_{k}^{(l,*)}$ as
\begin{align}\label{defn_strich_depe_0}
\xi_{k}^{(l,*)} = \bigl(\varepsilon_{k}, \varepsilon_{k - 1},\ldots,\varepsilon_{k - l}',\varepsilon_{k - l - 1}',\varepsilon_{k-l-2}',\ldots\bigr)
\end{align}
for an independent copy $(\varepsilon_k')_{k \in \Z}$ of $(\varepsilon_k)_{k \in \Z}$ and putting $X_k^{(l,\ast)} = f_k(\xi_{k}^{(l,*)})$, we can modify \hypertarget{lambdaLPP:eq13}{the} weak dependence measure by setting
\begin{align*}
\lambda_{l,p}= \sup_{k \in \Z}\big\|X_k-X_k^{(l,\ast)}\big\|_p.
\end{align*}
It is then an easy (but tedious) task to see that all of our results are equally valid. Since the proofs already require a notation based on multiple indices, we want to spare the reader any further complication. Similarly, a \textit{quenched} version is possible. That is, instead of $X_k$ we have
\begin{align*}
X_{kx}=f_k\big(\epsilon_k, \epsilon_{k-1},\ldots, \epsilon_0,x\big),
\end{align*}
where $x$ is some initial value. Again, the modifications are straightforward. Finally, let us mention that in the case of more concrete models, it appears that the conditions may be further relaxed, see for instance ~\cite{CUNY20181347spa} for special Markov processes. 


\section{Application I: High dimensional linear statistics}\label{sec:applicationII}

Throughout this section, we are given a separable Hilbert space $\mathds{H}$ with scalar product $\langle \cdot, \cdot \rangle$ and induced norm $\| \cdot \|_{\mathds{H}}$. For an orthonormal basis $u_1, u_2, \ldots$ and $x \in \mathds{H}$, we write $x_i$ for the coordinates $x_i = \langle x, u_i \rangle u_i$. For example, $X_{ki}$ for $X_k \in \mathds{H}$ and likewise $S_{ni}$ for $S_n \in \mathds{H}$. Except for constants, all random variables, sets and so on may depend on $n$, which is particularly important in a high dimensional context, where the dimension (if not already infinite) is typically allowed to grow in $n$. To keep the notation simple, we do not express this in any particular way.

Many important statistics and tests appearing in a high dimensional context\footnote{More precisely, this refers to the information theoretic complexity. Random variables may take values in a finite dimensional space, but the statistical complexity can nonetheless be high dimensional.} are based on the linear map
\begin{align}
\theta \mapsto \langle \theta, S_n \rangle, \quad \theta, S_n \in \mathds{H}.
\end{align}

Let us list a few prominent examples:

\begin{itemize}
\item[{\bf (i)}] Aggregation of estimators, weighted estimators, distributed estimation and parallel computation,  e.g. ~\cite{bunea2007}, ~\cite{ingster_book_2003}.
\item[{\bf (ii)}] Aggregation of processes, super linear processes, e.g. ~\cite{granger1980}, ~\cite{jirak_limit_aggregation}, ~\cite{volny_2010_superlinear_clt}.
\item[{\bf (iii)}] Chi-square statistics and tests in high dimension, e.g. ~\cite{ingster_book_2003}.
\item[{\bf (iv)}] (Random) Projections in high dimension, e.g. ~\cite{Li:2006:VSR:1150402.1150436}.
\end{itemize}

The above examples typically have the following features: In {\bf (i)}, {\bf (ii)}, $(X_{ki})_{k \in \Z}$, $(X_{kj})_{k \in \Z}$ are assumed to be independent for $i \neq j$ (subject to a specific basis). If $\theta$ is random in {\bf (i)}, {\bf (ii)}, {\bf (iv)}, then it is assumed to be independent of $(X_k)_{k \in \Z}$. Moreover, $\theta$ satisfies a tail condition of the type (possibly after reordering $\theta_i$)
\begin{align}\label{feature:decay}
\sum_{i \geq m} \theta_i^2 \to 0, \quad  \text{as $m \to \infty$}.
\end{align}
Below, we discuss {\bf (iii)} a little more detailed as it is not immediately obvious how this fits into the linear framework.

In light of these features, we make the following assumptions.

\begin{ass}[Global Assumptions]\label{ass:global}
Given $\theta \in \mathds{H}$, the process $(\langle \theta, X_k \rangle)_{k \in \Z}$ satisfies Assumption \ref{ass_dependence_main} for $p \geq 4$.
\end{ass}

This assumption is quite general and can easily be verified in many cases.

\begin{ass}[Tail Assumptions]\label{ass:tail}
Subject to a given orthonormal basis $u_1, u_2, \ldots$, there exists an index set $\mathcal{I}\subseteq \N$ (with complement $\mathcal{I}^c$) such that
\begin{itemize}
   \item[\Tone]\label{T1}$(X_{ki})_{k \in \Z, i \in \mathcal{I}}$ and  $(X_{ki})_{k \in \Z, i \in \mathcal{I}^c}$ are independent, and $(X_{ki})_{i \in \mathcal{I}}$ form a martingale difference sequence (with respect to some filtration) for any $k \in \Z$. 
  \item[\Ttwo]\label{T2} $(X_{ki})_{k \in \Z, i \in \mathcal{I}}$ uniformly satisfies Assumption \ref{ass_dependence_main}.
  \item[\Tthree]\label{T3} Given $\theta \in \mathds{H}$, there exist $\alpha > 0, \beta > 1$ and constants $c,C > 0$ such that
  \begin{align*}
  c n^{-1} \log^{\beta} n \leq   \sum_{i \in \mathcal{I}} \theta_i^2 \leq C n^{-\alpha}.
  \end{align*}
\end{itemize}
\end{ass}

\begin{rem}
In \hyperref[T2]{\Ttwo}, uniformly means that there exist absolute constants $c,C > 0$ such that
\begin{align*}
&\sup_{i \in \mathcal{I}}\big\{\sum_{k \in \Z} \E |X_{ki}|^p\log(1 + |X_{ki}|)^{\ad}, \, \|X_{0i}\|_p + \sum_{k = 0}^{\infty}k^{2}\|X_{ki}-X_{ki}^*\|_p\big\} \leq C, \\
&\inf_{i \in \mathcal{I}} \sum_{k \in \Z}\E X_{0i}X_{ki} \geq c > 0.
\end{align*}
\end{rem}

In view of the above examples and their features, the Tail Assumptions \ref{ass:tail} are quite natural and fairly general. Let us specifically mention here that the set $\mathcal{I}$ in \hyperref[T3]{\Tthree} may depend on $n$. Subject to these assumptions, the following result shows that high dimensional linear statistics always admit Edgeworth expansions.

\begin{thm}\label{thm:high:linear:stat}
Given $\theta \in \mathds{H}$, assume that Assumptions \ref{ass:global} and \ref{ass:tail} hold. Then
\begin{align*}
\sup_{x \in \R}\big|\P\big( \langle S_n, \theta \rangle \leq x \sqrt{n}\big) - \Psi_n(x)\big| \lesssim n^{-1}\log^5 n + n^{-(1 + \alpha)/2}.
\end{align*}
An analogous result holds for the metric $W_1$.
\end{thm}

\begin{rem}
If $\theta$ is random and independent of $(X_k)_{k \in \Z}$, observe that Theorem \ref{thm:high:linear:stat} also applies by conditioning with respect to $\theta$. In particular, if $\theta \in \Theta$ and the constants in Assumptions \ref{ass:global} and \ref{ass:tail} remain bounded from above (and below) uniformly for $\theta \in \Theta$, the bound in Theorem \ref{thm:high:linear:stat} holds uniformly over $\Theta$.
\end{rem}

Note that \textit{no additional smoothness} like a non-lattice or a (conditional) Cram\'{e}r condition or even a density is necessary for this result. As the proof shows, the Tail Assumptions \ref{ass:tail} are enough to control the characteristic $\Cd_{T_n}$.

Let us now briefly discuss {\bf (iii)} (Chi Square statistics and tests), where we recall that we assume $\mathds{H}$ to be separable. For $X \in \mathds{H}$ with $\E X_k = 0$ and $\E \|X\|_{\mathds{H}}^2 < \infty$, we have the Karhunen-Lo\`{e}ve expansion
\begin{align*}
X_k = \sum_{i \geq 1} \sqrt{\lambda_i} u_i \eta_{ki},
\end{align*}
where $\lambda_i$, $u_i$ denote the eigenvalues and eigenfunctions of the covariance operator $\Sigma = \E X \otimes X$, and $\eta_{ki}$ are the (normalised) scores. The scores are uncorrelated and have unit variance. In this context, it is often assumed in the literature that the sequences $(\eta_{ki})_{k \in \Z}$ are independent for different $i \in \N$, we refer to ~\cite{JW18} for a discussion. Since
\begin{align*}
\big\|X_k\big\|_{\mathds{H}}^2 - \E \big\|X_k\big\|_{\mathds{H}}^2 = \sum_{i \geq 1} \lambda_i \big(\eta_{ki}^2 - 1\big),
\end{align*}
where the latter can obviously be written as a scalar product with $\theta_i = \lambda_i$, we see how this fits into the above framework.

\section{Application II: Smooth functions and weak Edgeworth expansions}\label{sec:appl:I}

Consider a function $f: \R \to \R$, whose \hypertarget{bigdf:eq9}{first} derivative $f^{(1)}$ satisfies

\begin{align}\label{eq_derivative_condition_1}
\text{$f^{(1)}$ is continuous and $\Do_f = \sup_{x \in \R}\bigl|f^{(1)}(x)\big| < \infty$.}
\end{align}

\begin{thm}\label{thm_smooth}
Grant Assumption \ref{ass_dependence_main} for $\ad > 3$. Then for any function $f$ satisfying \eqref{eq_derivative_condition_1}, \hypertarget{psin:eq9}{we} have
\begin{align*}
\Big|\int_{\R}f(x)d\bigl(\P_{S_n/\sqrt{ n}}(x) - \Psi_n(x)\bigr)\Big| = \oo\big(n^{-\frac{1}{2}}\big).
\end{align*}
\end{thm}

Such expansions are sometimes referred to as weak Edgeworth expansions, see e.g. ~\cite{woodroofe1989}.

Theorem \ref{thm_smooth} requires \hypertarget{smid:eq10}{the} same minimal smoothness and moment conditions (except for $\ad > 3$) as in the i.i.d. case, see ~\cite{Bhattacharya_rao_1976_reprint_2010}. For weakly dependent sequences, stronger assumptions are needed in the literature, see for instance ~\cite{goetze_hipp_1983} ~\cite{lahiri_1993} and in particular ~\cite{mykland1992},~\cite{mykland1993aop}, ~\cite{mykland1995}. For results subject to spectral conditions, see e.g. \cite{kasun:2018:edgeworth}.

Let $s > 0$, then we \hypertarget{schls:eq9}{can} represent $s$ as $s = \m + \alpha$, where $[s] = \m$ denotes the integer part, and $0 < \alpha \leq 1$. Denote by $\mathcal{H}_{\mathrm{L}}^{s}$ the H\"{o}lder-class, that is, all real-valued functions $f$ such that the $\m$-th derivative exists and satisfies
\begin{align*}
\bigl|f^{(\m)}(x) - f^{(\m)}(y)\bigr| \leq \mathrm{L} \bigl|x - y\bigr|^{\alpha}, \quad \mathrm{L} > 0.
\end{align*}

\begin{thm}\label{thm_hoelder}
Grant Assumption \ref{ass_dependence_main} with $3 < p < 4$. Then for any fixed $f \in \mathcal{H}_{\mathrm{L}}^{s}$ with $s \in (1,2)$, we have
\begin{align*}
\Big|\int_{\R}f(x)d\bigl(\P_{S_n/\sqrt{ n}}(x) - \Psi_n(x)\bigr)\Big| \lesssim \big(\mathrm{L} + |f^{(1)}(0)|\big)n^{-\frac{p}{2}+1+\delta} + n^{-\frac{s}{2}},
\end{align*}
where $\delta > 0$ is arbitrarily  \hypertarget{SCL:eq10}{small}. If $p = 4$, then we may replace $n^{\delta}$ with $(\log n)^{6\frac{1}{2}}$.
\end{thm}

\section{Examples}\label{sec:examples}

In this section, we discuss how some prominent examples from the literature fit into our framework. Let us point out that in all these examples, not only the characterization by Theorems \ref{thm_edge} and \ref{thm_wasserstein} (and higher moment versions) and the weak expansions in Theorems \ref{thm_smooth}, \ref{thm_hoelder} are new, but also the CLTs in $L^p$ and $W_1$. Many more examples satisfying our conditions can be found, for instance, in \cite{aue_etal_2009},\cite{wu_2005} and the references therein.

\subsection{Functions of Banach space valued linear processes}\label{ex:banach:linear}

Suppose that $\mathbb{S} = \mathds{B}$ is a Banach space with norm $\|\cdot\|_{\mathds{B}}$. Let $(A_i)_{i \in \N}$ be a sequence of linear operators $A_i : \mathds{B} \to \mathds{B}$, and denote with
$\Vert A_i \Vert_{op}$ the corresponding operator norm. For an i.i.d. sequence $(\varepsilon_k)_{k \in \Z} \in \mathds{B}$, consider the linear process
\begin{align*}
L_k = \sum_{i = 0}^{\infty} A_i \varepsilon_{k-i}, \quad k \in \Z,
\end{align*}
which exists if $\E \|\varepsilon_0\|_{\mathds{B}} < \infty$ and $\sum_{i \in \N} \|A_i\|_{op} < \infty$, which we assume from now on. Note that autoregressive processes (even of infinite order) can be expressed as linear processes, but also the famous dynamical system 2x mod 1 (Bernoulli convolution, doubling map). For the latter, we refer to Example 3.2 in ~\cite{jirak_be_aop_2016} and the references therein for more details. For any finite $d$, let $\mathds{B}^d = \mathds{B} \oplus \ldots \oplus \mathds{B}$ be the direct sum, and we equip $\mathds{B}^d$ with any of the equivalent norms, which we denote with $\Vert \cdot \Vert_{\mathds{B}^d}$. Let $g:\mathds{B}^d \to \R$ be a function satisfying
\begin{align*}
\big|g(x) - g(y)\big| \leq C \big \Vert x-y\big \Vert_{\mathds{B}^d}^{\alpha} \big(1 + \Vert x \Vert_{\mathds{B}^d} + \Vert y \Vert_{\mathds{B}^d} \big)^{\beta}, \quad C,\alpha > 0, \beta \geq 0,
\end{align*}
and define $X_k$ by
\begin{align}\label{ex:linear:banach:defn}
X_k = g\big(L_k,L_{k-1}, \ldots, L_{k-d+1}\big) - \E g\big(L_k,L_{k-1}, \ldots, L_{k-d+1}\big).
\end{align}
Note that for $\mathds{B} = \R$, this setup includes empirical autocovariance functions and other important statistics. Suppose that
\begin{align}\label{ex:linear:banach:condi}
\E \|\varepsilon\|_{\mathds{B}}^q < \infty, \quad \sum_{k = 0}^{\infty} k^2  \Big(\sum_{i \geq k} \big\|A_i\big\|_{op} \Big)^{\alpha} < \infty,
\end{align}
for $q > \beta p (\alpha/\beta + 1)$ if $\beta > 0$, and $q > p \alpha$ otherwise. The triangle inequality, H\"{o}lder's inequality and some computations then yield
\begin{align*}
\big\|X_k - X_k^{\ast}\big\|_p \lesssim  \Big(\sum_{i \geq k} \big\|A_i\big\|_{op} \Big)^{\alpha}.
\end{align*}
If $\mathds{B}$ admits a Burkholder inequality, this bound can be further improved, see ~\cite{wu_2005}. By the above, we can now state the following result.

\begin{prop}\label{prop:linear:banach}
Given the above conditions, if \eqref{ex:linear:banach:condi} holds and $\sd^2 > 0$, then the process $X_k$ defined in \eqref{ex:linear:banach:defn} satisfies Assumption \ref{ass_dependence_main}.
\end{prop}

Extensions to more general processes (bilinear or even Volterra processes) are possible.


\subsection{Left random walk on $GL_d(\R)$ and cocycles}\label{ex:randomwalk}

Cocycles, in particular the random walk on $GL_d(\R)$, have been heavily investigated in the literature, see e.g. ~\cite{bougerol_book_1985} and ~\cite{benoist2016}, ~\cite{cuny2017}, ~\cite{CUNY20181347spa} for some more recent results. We will particularly exploit ideas of ~\cite{cuny2017}, ~\cite{CUNY20181347spa}.\\ 
Let $(\varepsilon_k)_{k \in \N}$ be independent random matrices taking values in $G = GL_d(\R)$, with common distribution $\mu$. Let $A_0 = \mathrm{Id}$, and for every $n \in \N$, $A_n = \prod_{i = 1}^n \varepsilon_i$. Denote with $\|\cdot\|$ the Euclidean norm on $\R^d$. We adopt the usual convention that $\mu$ has a moment of order $q$, if
\begin{align*}
\int_G \big(\log N(g) \big)^q \mu(d g) < \infty, \quad N(g) = \max\big\{\|g\|, \|g^{-1}\|\big\}.
\end{align*}
Let $\mathds{X} = P_{d-1}(\R^d)$ be the projective space of $\R^d\setminus\{0\}$, and write $\overline{x}$ for the projection from $\R^d\setminus\{0\}$ to $\mathds{X}$. We assume that $\mu$ is strongly irreducible and proximal, see ~\cite{cuny2017} for details. The left random walk of law $\mu$ started at $\overline{x} \in \mathds{X}$ is the Markov chain given by $Y_{0\overline{x}} = \overline{x}$, $Y_{k\overline{x}} = \varepsilon_k Y_{k-1\overline{x}}$ for $k \in \N$. Following the usual setup, we consider the associated random variables $(X_{k\overline{x}})_{k \in \N}$, given by
\begin{align}\label{ex:randomwalk:defn}
X_{k\overline{x}} = h\big(\varepsilon_k, Y_{k-1\overline{x}} \big), \quad h\big(g, \overline{z}\big) = \log \frac{\| g z\|}{\|z\|},
\end{align}
for $g \in G$ and $z \in \R^d\setminus\{0\}$. It follows that, for any $x\in \mathbf{S}^{d-1}$, we have
\begin{align*}
S_{n\overline{x}} = \sum_{k = 1}^n \big(X_{k\overline{x}} - \E X_{k\overline{x}}\big) = \log \big\|A_n \overline{x}\big\| - \E \log \big\|A_n \overline{x}\big\|.
\end{align*}
Following ~\cite{CUNY20181347spa}, Proposition 3 in ~\cite{cuny2017} implies that, if $q > 4p$, then
\begin{align*}
\sum_{k = 1}^{\infty} k^2 \sup_{\overline{x},\overline{y} \in \mathds{X}}\big\|X_{k\overline{x}} - X_{k\overline{y}} \big\|_p < \infty.
\end{align*}
In particular, it holds that
\begin{align*}
\lim_{n \to \infty} n^{-1} \E S_{n\overline{x}}^2 = \sd^2,
\end{align*}
where the latter does not depend on $\overline{x} \in \mathds{X}$. We are now exactly in the \textit{quenched} setup briefly mentioned in Section \ref{sec_K_aut}, and can state the following result.

\begin{prop}\label{prop:randomwalk}
If $q > 4p$ and $\sd^2 > 0$, then for $\overline{x} \in \mathds{X}$, the process $X_{k\overline{x}}$ defined in \eqref{ex:randomwalk:defn} satisfies Assumption \ref{ass_dependence_main} (quenched) with $p$.
\end{prop}

As an immediate consequence, we obtain a Berry-Essen bound with optimal rate (Kolmogorov, $W_1$ and $L^p$), weak Edgeworth expansions and the sharp characterisation of the transition subject to mild moment conditions. To the best of our knowledge, all of these results are new. Typically, the literature requires the existence of all moments for similar results. In case of $W_1$ and $L^p$, we are not even aware of any other result where the optimal rate has been obtained.

\subsection{Non-uniform expanding maps}\label{ex:nonuniform:maps}

In this section, we do not discuss concrete examples, but list many well-known non-uniformly expanding (and non-uniformly
hyperbolic) transformations, where (two-sided) variants of Assumption \ref{ass_dependence_main} (or its quenched version, see Section \ref{sec_K_aut} for a brief comment) are satisfied. This is due to the recent advancements made in \cite{korepanov2018},\cite{cuny_quickly_2020},\cite{cuny_dedecker_korepanov_merlevede}, where it is shown that (two-sided) variants of Assumption \ref{ass_dependence_main} hold for H\"{o}lder continuous observables for many (non)-uniformly expanding maps with exponential tails (cf. \cite{korepanov2018}), sub exponential tails (cf. \cite{cuny_quickly_2020}), and also slowly mixing systems (cf. \cite{cuny_dedecker_korepanov_merlevede}). Among others, this includes:
\begin{itemize}
  \item Anosov and Axiom A diffeomorphisms,
  \item Gibbs-Markov maps with big images,
  \item dispersing billiards,
  \item classes of logistic and H\'{e}non maps,
  \item Lozi maps,
  \item uniformly expanding maps such as the doubling map, Gauss continued fraction map,
$\beta$-shifts.
\end{itemize}

With some (partially laborious) adaptations, results can be transferred.

\subsection{Functions of iterated random maps}\label{ex:iterated}

Let $(\mathcal{Y},\rho)$ be a complete, separable metric space. An iterated random
function system on the state space $\mathcal{X}$ is defined as
\begin{align}
Y_k = F_{\epsilon_k}\big(Y_{k-1} \big), \quad k \in \N,
\end{align}
where $\epsilon_k \in \mathbb{S}$ are i.i.d. with $\epsilon \stackrel{d}{=} \epsilon_k$. Here, $F_{\epsilon}(\cdot) = F(\cdot, \epsilon)$ is the $\epsilon$-section of a jointly
measurable function $F : \mathcal{Y} \times \mathbb{S} \to \mathcal{Y}$. Many dynamical systems, Markov processes and non-linear time series are within this framework, see for instance ~\cite{diaconis_freedman_1999}. For $y \in \mathcal{Y}$, let
$Y_k(y) = F_{\epsilon_k} \circ F_{\epsilon_{k-1}} \circ \ldots \circ F_{\epsilon}(y)$, and, given $y,y' \in \mathcal{Y}$ and $\gamma > 0$, we say that the system is $\gamma$-\textit{moment contracting} if
\begin{align}\label{ex:iterated:contraction}
\sup_{y,y'}\E \rho^{\gamma}\big(Y_k(y),Y_k(y')\big) \leq C \theta^k, \quad \theta \in (0,1).
\end{align}
We note that slight variations exist in the literature. A key quantity for verifying the moment contraction \eqref{ex:iterated:contraction} is
\begin{align*}
L_{\epsilon} = \sup_{y \neq y'}\frac{\rho\big(F_{\epsilon}(y), F_{\epsilon}(y')\big)}{\rho(y,y') }.
\end{align*}
Essentially (subject to some mild regularity conditions), \eqref{ex:iterated:contraction} holds if $\E L_{\epsilon}^{\gamma} < \infty$ and $\E \log L_{\epsilon} < 0$, see ~\cite{diaconis_freedman_1999}, ~\cite{wu_shao_iterated_2004}. Let $g:\mathcal{Y}^d \to \R$ be a function satisfying ($\rho_{\mathcal{Y}^d}$ is any product metric on $\mathcal{Y}^d$, $0\in \mathcal{Y}^d$ is some reference point)
\begin{align*}
\big|g(x) - g(y)\big| \leq C \rho_{\mathcal{Y}^d}^{\alpha}\big(x,y\big) \big(1 + \rho_{\mathcal{Y}^d}(x,0) + \rho_{\mathcal{Y}^d}(y,0) \big)^{\beta}, \quad C,\alpha > 0, \beta \geq 0,
\end{align*}
and define $X_k(y)$ by (we use the abbreviation $Y_k = Y_k(y)$ below)
\begin{align}\label{ex:iterate:defn}
X_k(y) = g\big(Y_k,Y_{k-1}, \ldots, Y_{k-d+1}\big) - \E g\big(Y_k,Y_{k-1}, \ldots, Y_{k-d+1}\big).
\end{align}
For $\gamma \geq q$, where $q > \beta p (\alpha/\beta + 1)$ if $\beta > 0$, and $q > p \alpha$ otherwise, the triangle inequality, H\"{o}lder's inequality and some computations then yield
\begin{align*}
\sup_{y,y' \in \mathcal{Y}}\big\|X_k(y) - X_k^{\ast}(y')\big\|_p \lesssim \theta^{k}, \quad \text{where $0 < \theta < 1$,}
\end{align*}
provided that $\E \rho^q(Y_k(y),0)<\infty$. As in Example \ref{ex:randomwalk}, we are now in the \textit{quenched} setup briefly mentioned in Section \ref{sec_K_aut}, and can state the following result.
\begin{prop}\label{prop:iterated}
Given the above conditions, if $\sd^2 > 0$, then the process $X_k$ defined in \eqref{ex:iterate:defn} satisfies Assumption \ref{ass_dependence_main}.
\end{prop}

As before in Example \ref{ex:randomwalk}, in conjunction with our theory, this provides many new quantitative limit theorems in various probability metrics and (weak) Edgeworth expansions.

\subsection{Functions of GARCH$(\mathfrak{p},\mathfrak{q})$ processes}\label{sec:ex:garch}

Let $\mathbb{S} = \R$. A very prominent stochastic recursion is the GARCH($\mathfrak{p},\mathfrak{q})$ sequence, given through the relations
\begin{align*}
&Y_k = \varepsilon_k L_{k} \quad \text{where $(\varepsilon_{k})_{k \in \Z}$ is a zero mean i.i.d. sequence and}\\
&L_k^2 = \mu + \alpha_1 L_{k - 1}^2 + \ldots + \alpha_\mathfrak{p} L_{k - \mathfrak{p}}^2 + \beta_1 Y_{k - 1}^2 + \ldots + \beta_{\mathfrak{q}} Y_{k - \mathfrak{q}}^2,
\end{align*}
with $\mu > 0$, $\alpha_1,\ldots,\alpha_\mathfrak{p}, \beta_1,\ldots,\beta_{\mathfrak{q}} \geq 0$. Assume first $\|\varepsilon_k\|_q < \infty$ for some $q \geq 2$. An important quantity is
\begin{align*}
\gamma_C = \sum_{i = 1}^{r} \bigl\|\alpha_i + \beta_i \varepsilon_i^2\bigr\|_{q/2}, \quad \text{with $r = \max\{\mathfrak{p},\mathfrak{q}\}$},
\end{align*}
where we replace possible undefined $\alpha_i, \beta_i$ with zero. If $\gamma_C < 1$, then $(Y_k)_{k \in \Z}$ is stationary. In particular, one can show the representation
\begin{align*}
Y_k = \sqrt{\mu}\varepsilon_k\biggl(1 + \sum_{n = 1}^{\infty} \sum_{1 \leq l_1,\ldots,l_n\leq r} \prod_{i = 1}^n\bigl(\alpha_{l_i} + \beta_{l_i}\varepsilon_{k - l_1 - \ldots - l_i}^2\bigr) \biggr)^{1/2},
\end{align*}
we refer to ~\cite{aue_etal_2009} for comments and references. Using this representation and the fact that $|x-y|^q \leq |x^2 - y^2|^{q/2}$ for $x,y \geq 0$, one can follow the proof of Theorem 4.2 in ~\cite{aue_etal_2009} to show that
\begin{align*}
\bigl\|Y_k - Y_k^{\ast}\bigr\|_q \leq C \theta^{k}, \quad \text{where $0 < \theta < 1$.}
\end{align*}
Let $g:\R^d \to \R$ be a function satisfying ($\|\cdot\|_{\R^d}$ denotes the Euclidean norm)
\begin{align*}
\big|g(x) - g(y)\big| \leq C \big \Vert x-y\big \Vert_{\R^d}^{\alpha} \big(1 + \Vert x \Vert_{\R^d} + \Vert y \Vert_{\R^d} \big)^{\beta}, \quad C,\alpha > 0, \beta \geq 0,
\end{align*}
and define $X_k$ by
\begin{align}\label{ex:garch:defn}
X_k = g\big(Y_k,Y_{k-1}, \ldots, Y_{k-d+1}\big) - \E g\big(Y_k,Y_{k-1}, \ldots, Y_{k-d+1}\big).
\end{align}
Note that this setting includes empirical autocovariance functions and other important statistics. Given $q > \beta p (\alpha/\beta + 1)$ if $\beta > 0$, and $q > p \alpha$ otherwise, the triangle inequality, H\"{o}lder's inequality and some computations then yield
\begin{align*}
\big\|X_k - X_k^{\ast}\big\|_p \lesssim \theta^{k}, \quad \text{where $0 < \theta < 1$.}
\end{align*}
By the above, we can now state the following result.

\begin{prop}\label{prop:garch}
Given the above conditions, if $\sd^2 > 0$, then the process $X_k$ defined in \eqref{ex:garch:defn} satisfies Assumption \ref{ass_dependence_main}.
\end{prop}

We note that analogous results can be shown for many more processes of this kind, such as augmented GARCH$(\mathfrak{p},\mathfrak{q})$ processes.



\section{Main Technical Results and Notation}\label{sec_main_tech}

A key tool is the $m$-dependence approximation. To this end, we require some additional notations and definitions. Throughout the remainder of this section, let
\begin{align*}
X_k = f_m(\varepsilon_k,\ldots,\varepsilon_{k-m+1}) \quad \text{for $m \in \N$, $k \in \Z$,}
\end{align*}
where $f_m$ are measurable functions. Unless otherwise stated, $m$ can be any value, even $m = \infty$ is allowed. Strictly speaking, we work with $(m-1)$-dependent sequences (instead of $m$) for notational reasons. However, the difference is irrelevant in the sequel since we are particularly interested in the case where $m = m_n \to \infty$ as $n$ increases, and $(m-1)$-dependent sequences are $m$-dependent any way. For the sake of reference, we now restate Assumption \ref{ass_dependence_main} in \hypertarget{ssm2:eq12}{this} context.
\begin{ass}\label{ass_dependence}
For $p \geq 3$, $\ad > 2$, \hypertarget{lamdaqp:eq12}{the} $(m-1)$-dependent sequence $(X_k)_{k\in \Z}$ satisfies
\begin{enumerate}
\item[\Bone]\label{B1} $\E |X_k|^p\log(1 + |X_k|)^{\ad} < \infty$, $\E X_k = 0$,
\item[\Btwo]\label{B2} $\Lambda_{2,p} = \sum_{k = 1}^{\infty}k^{2}\|X_k-X_k^*\|_p < \infty$,
\item[\Bthree]\label{B3} $\ss_m^2 > 0$, where $\ss_m^2 =  \sum_{k \in \Z}\E X_0X_k = \sum_{k = -m}^{m}\E X_0X_k$.
\end{enumerate}
\end{ass}

Assumption \ref{ass_dependence} always indicates that we work with $m$-dependent sequences. In fact, except for Sections \ref{sec_proof_main} and \ref{sec_proof_wasser}, we \textit{always} work exclusively with $m$-dependent sequences.

As starting point, we use a similar conditioning scheme as in ~\cite{jirak_be_aop_2016}. Key tools employed there are Ideal (Zolotarev) metrics and martingale approximations. Unfortunately though, these key methods and results break down and are no longer available in the present context, forcing us to find a new argument. Fortunately, a recursive argument, briefly mentioned in Section \ref{sec_int_approach}, turns out to work, \hypertarget{littleek:eq12}{but} is not easy to implement.

Let $n = 2Nm + m'$ with $0 \leq m' \leq m$, and $e_1 = \{\varepsilon_{-m+1},\ldots,\varepsilon_0\}$, $e_2 = \{\varepsilon_{m+1},\ldots,\varepsilon_{2m}\}$ and so on for $e_j$, $j \leq N+1$. Define \hypertarget{FFm:eq12}{the} following $\sigma$-algebra
\begin{align}\label{defn_sigma_algebra}
\FF_m = \sigma\bigl(e_1 \cup e_2 \cup \ldots \cup e_{N+1}\bigr).
\end{align}
In the sequel, we can assume without loss of generality that $m' = 0$, hence $n = 2Nm$. We can do so since our viewpoint is  asymptotic, that is, whether $m = a_n$ or $m = 2a_n$ for $a_n \to \infty$ will not matter. If, indeed $m' > 0$, we can simply redefine $e_{N+1}$ as $e_{N+1} = e_{N+1} \cup \{\varepsilon_{2Nm+1}, \ldots, \varepsilon_{2Nm + m'}\}$, and all arguments remain the same. However, assuming $m' = 0$ is notationally much \hypertarget{EEH:eq12}{more} convenient.

We write $\P_{\FF_m}$ for the conditional law \hypertarget{barSNM:eq12}{and}, sometimes use the abbreviation $\E_{\FF_m} \cdot = \E[\cdot | \FF_m]$ (or $\E_{\mathcal{H}} \cdot $ for some other $\sigma$-algebra $\mathcal{H}$).  \hypertarget{tilSNM:eq12}{Introduce}
\begin{align*}
\overline{S}_{n|m} = \sum_{k = 1}^n  \big(X_k - \E_{\FF_m} X_k\big) \quad \text{and} \quad \widetilde{S}_{n|m} = \sum_{k = 1}^n \E_{\FF_m} X_k,
\end{align*}
hence
\begin{align*}
S_n = \sum_{k = 1}^n X_k  = \overline{S}_{n|m} + \widetilde{S}_{n|m}.
\end{align*}

For $1 \leq j \leq N$, we construct \hypertarget{biguj:eq13}{the} block random  \hypertarget{bigrj:eq13}{variables}
\begin{align}\label{defn_U_R}
{U}_j = \sum_{k = (2j-2)m + 1}^{(2j - 1)m} \big(X_k - \E_{\FF_m} X_k\bigr), \quad  R_j =  \sum_{k = (2j-1)m + 1}^{2jm} \big(X_k - \E_{\FF_m} X_k \bigr),
\end{align}
and put $\sqrt{2m}\,\overline{V}_j = U_j + R_j$, hence $\overline{S}_{n|m} = \sqrt{2m}\sum_{j = 1}^{N} \overline{V}_j$. Note that by construction, the blocks $\overline{V}_j$, $j = 1,\ldots,N$ are i.i.d. random variables under the conditional probability measure $\P_{\FF_m}$.
We also put
\begin{align*}
&\sqrt{2m}\widetilde{V}_0 = \sum_{k = 1}^{m} \E_{\FF_m}X_k, \quad \sqrt{2m}\widetilde{V}_N = \sum_{k = (2N-1)m + 1}^{2Nm} \E_{\FF_m}X_k,\\ &\sqrt{2m}\widetilde{V}_j = \sum_{k = (2j-1)m + 1}^{(2j+1)m} \E_{\FF_m} X_k, \quad \text{for $j = 1,\ldots,N-1$.}
\end{align*}
Note that $\widetilde{V}_j$, $j = 0,\ldots,N$ are a sequence of independent random variables with respect to $\P$, and are i.i.d. for $1 \leq j \leq N-1$.

The \hypertarget{barsigmajj:eq13}{following} partial \hypertarget{kappaJM:eq13}{and} conditional \hypertarget{sigmaJCM:eq13}{variances} \hypertarget{nuJM:eq13}{are} relevant \hypertarget{barNuj:eq13}{for} the \hypertarget{tilsigmajj:eq13}{proofs}:
\begin{align}
\overline{\sigma}_{j|m}^2 = \E_{\FF_m} \overline{V}_j^2, \quad \overline{\sigma}_j^2 = \E \overline{V}_j^2, \quad \widetilde{\sigma}_{j}^2 =  \E \widetilde{V}_j^2.
\end{align}
 \hypertarget{tildeNuj:eq13}{Similarly}, \hypertarget{tilkappaj:eq13}{put}

\begin{align*}
&\overline{\kapp}_{j|m}^3 = \E_{\FF_m} \overline{V}_j^3, \quad \overline{\kapp}_j^3 = \E \overline{V}_j^3,\quad \widetilde{\kapp}_{j}^3 = \E \widetilde{V}_j^3,
\end{align*}
and if $p \geq 4$
\begin{align*}
&\overline{\nu}_{j|m}^4 = \E_{\FF_m} \overline{V}_j^4, \quad \overline{\nu}_j^4 = \E \overline{V}_j^4,\quad \widetilde{\nu}_{j}^4 = \E \widetilde{V}_j^4.
\end{align*}
We remark that Lemma \ref{lem_sig_expansion} reveals that $2 \overline{\sigma}_j^2 = \ss_m^2 + \OO(m^{-1})$, and analogous results (see Section \ref{sec_lem_moments}) hold for $\overline{\kapp}_j^3$, $\overline{\nu}_j^4$ and $\widetilde{\sigma}_j^2$, $\widetilde{\kapp}_j^3$ and $\widetilde{\nu}_j^4$. Let ${F}$ (formally\footnote{A convolution of the normal and (centered) gamma distribution works, see Section \ref{sec_wasserstein} and Lemma \ref{lem_edge_for_FF}. }) be any continuous distribution function 
and $(\overline{Z}_{j})_{1 \leq j \leq N}$ be i.i.d. and distributed according to ${F}_{}$ such that
\begin{align}\nonumber \label{moment_conditions_global_FF}
&\E\overline{Z}_{j} = 0, \quad \E\overline{Z}_{j}^2 = \overline{\sigma}_j^2 + \OO\big(m^{-1}\big), \\&\E\overline{Z}_{j}^3 = \overline{\kappa}_j^3 + \OO\big(m^{-1}\big), \quad \E\overline{Z}_{j}^4 = \overline{\nu}_j^4 + \OO\big(m^{-1}\big) \quad \text{if $p \geq 4$}.
\end{align}
Let $(\widetilde{Z}_j)_{1 \leq j \leq N}$ be an independent \hypertarget{barZ:eq14}{copy} of $(\overline{Z}_j)_{1 \leq j \leq N}$, \hypertarget{tildeZ:eq14}{and} put
\begin{align}\nonumber
\overline{Z} &= N^{-\frac{1}{2}} \sum_{j = 1}^N \overline{Z}_j, \quad \widetilde{Z} = N^{-\frac{1}{2}} \sum_{j = 1}^{N} \widetilde{Z}_j.
\end{align}

Instead of using $(\overline{Z}_j)_{1 \leq j \leq N}$ and $(\widetilde{Z}_j)_{1 \leq j \leq N}$, we could also work with the usual signed measures $\mathds{V}_n$, induced by $\Psi_n$. However, the latter are notationally (much) less convenient for our proofs. 
Next, we define \hypertarget{UUTn:eq14}{the} counter part of $\Cd_{T_n}$, \hypertarget{darkm:eq15}{namely}

\begin{align*}
\Ad_{T_n} = \int_{-T_n}^{T_n}\Bigl|\E e^{\ic \xi S_n/\sqrt{n}}  - \E e^{\ic \xi (\overline{Z}+ \widetilde{Z})} \Bigr| \frac{1}{|\xi|}\, d\xi.
\end{align*}

\begin{prop}\label{prop_m_dependence}
Assume that Assumption \ref{ass_dependence} holds for $m = n^{\md}$ with $0 \leq \md < 1$. Then for $2\bd +2 < \ad$, we have
\begin{align*}
\Ad_{T_n} = \oo\big(n^{-\frac{1}{2}}\log(n)^{-\bd}\big).
\end{align*}
All involved constants only depend on $\ad$, $\md$, $\ss_m$ and $\Lambda_{2,3}$.
\end{prop}

\begin{rem}
Note that without loss of generality, $\md$ can be chosen to be arbitrarily close to one, a fact we shall use repeatedly in the proof.
\end{rem}

\begin{prop}\label{prop_m_dep_it}
Assume that Assumption \ref{ass_dependence} holds for $m = n (\log n)^{-\md}$ with $\md > 2$. If $3 < p < 4$, then
\begin{align*}
\Ad_{T_n} \lesssim n^{1-\frac{p}{2}+\delta},
\end{align*}
where $\delta > 0$ can be arbitrarily small. All involved constants only depend on $p,\ad,\md$, $\ss_m$ and $\Lambda_{2,p}$.
\end{prop}

\begin{prop}\label{prop_m_dep_it_quad}
Assume that Assumption \ref{ass_dependence} holds for $m = n (\log n)^{-\md}$ with $\md \geq 5$. If $p \geq 4$, then
\begin{align*}
\Ad_{T_n} \lesssim n^{-1} (\log n)^5.
\end{align*}
All involved constants only depend on $p,\ad,\md$, $\ss_m$ and $\Lambda_{2,p}$.
\end{prop}

Proving Propositions \ref{prop_m_dependence}, \ref{prop_m_dep_it} and \ref{prop_m_dep_it_quad} is the most involved part. For the proof of Proposition \ref{prop_m_dependence}, we link Edgeworth expansions to Berry-Esseen type bounds, see Lemmas \ref{lem_taylor_expansion_smooth_function}, \ref{lem_taylor_expansion_smooth_function_II} and \ref{lem_taylor_expansion_smooth_function_I_BE}. For the proof of Propositions \ref{prop_m_dep_it} and \ref{prop_m_dep_it_quad}, we use an iterative argument. Throughout all proofs, we make the following conventions.
\begin{conv}\label{conv}
\hspace{10cm}
\begin{description}
\item[(i)] We \hypertarget{scAc:eq14}{assume} without loss of generality $m' = 0$. 
\item[(ii)] The abbreviations $I$, $II$, $III$, $\ldots$, for expressions (possible with some additional indices) vary from proof to proof.
\item[(iii)] Given a set $\mathcal{A}$, we denote by $\mathcal{A}^c$ its complement.
\item[(iv)] We write $\stackrel{def}{=}$ if we make definitions on the fly.
\item[(v)] We denote by $\co, \co_0, \ldots $ generic, absolute constants that may vary from line to line.
\end{description}
\end{conv}

\section{Main Lemmas}\label{sec_main_lemmas}

Throughout this section, we work under Assumption \ref{ass_dependence}. For the proof of \hypertarget{xiklstar:eq15}{the} main lemmas, \hypertarget{EEk:eq15}{we} require some more notation connected to the weak dependence coefficients $\lambda_{k,p}$. For $k \in \Z$, consider the $\sigma$-algebra $\mathcal{E}_k = \sigma\bigl(\varepsilon_j, \, j \leq k\bigr)$, and recall \hypertarget{xkell:eq15}{that} the filters $\xi_{k}^{(l,*)}$ are defined in \eqref{defn_strich_depe_0}  \hypertarget{bigXkstar:eq15}{as}
\begin{align}\label{defn_strich_depe_2}
\xi_{k}^{(l,*)} = \bigl(\varepsilon_{k}, \varepsilon_{k - 1},\ldots,\varepsilon_{k - l}',\varepsilon_{k - l - 1}',\varepsilon_{k-l-2}',\ldots\bigr).
\end{align}
Write ${X}_{k}^{(l,*)} = f_m(\xi_{k}^{(l,*)})$ and note that ${X}_{k}^{*} = {X}_{k}^{(k,*)}$, where ${X}_{k}^{*}$ is given in \eqref{defn_couple_star}. In an analogous manner, let $(\varepsilon_k'')_{k \in \Z}$ be another independent copy of $(\varepsilon_k)_{k \in \Z}$. For $l \leq k$, we then introduce the quantities $X_k^{(l,**)}, X_k^{**}$ in analogy to $X_k^{(l,*)}, X_k^{*}$. This means \hypertarget{scEkstar:eq15}{that} we replace every $\varepsilon_k'$ with $\varepsilon_k''$ at all corresponding places. For $k \geq 0$, we \hypertarget{EEKK:eq15}{also} introduce the $\sigma$-algebras
\begin{align*}
\mathcal{E}_k^* = \sigma\bigl(\varepsilon_j, \, 1 \leq j \leq k \, \text{and} \,  \varepsilon_i', \, i \leq 0\bigr),
\end{align*}
and $\mathcal{E}_k^{**}$ in an analogous manner.

\subsection{Moments and Conditional Moments}\label{sec_lem_moments}

Recall Convention \ref{conv} {\bf (i)}, that we heavily use in this section to simplify notation. This has no impact on \hypertarget{snsquare:eq15}{the} actual results.

\begin{lem}\label{lem_S_n_third_expectation}
Grant Assumption \ref{ass_dependence}. \hypertarget{sn:eq15}{Then}
\begin{description}
  \item[(i)] $\big|\E S_n^3\big| \lesssim n$,
  \item[(ii)] $\big|\E S_n^3  - n \sum_{i,j \in \Z} \E X_0 X_i X_j  \big| = O\big(1\big)$,
  \item[(iii)]  $\big|\E S_n^4  - 3 n^2 s_n^4\big| \lesssim n$,
  \item[(iv)] $\big|\sum_{i,j \in \Z} \E X_0 X_i X_j  \big| < \infty$.
\end{description}
\end{lem}

The above results are based on computations connected to cumulants and the higher-order spectral density. We refer to ~\cite{wu_zhang_IEEE_2017} for more details on this subject.

\begin{proof}[Proof of Lemma \ref{lem_S_n_third_expectation}]
{\bf (i)}: For $i \leq j \leq k$, we have
\begin{align*}
\E X_j X_k  = \E\big[(X_j X_k)^{(k-i,*)}\big|\mathcal{E}_i\big] = \E \big[X_j^{(j-i,*)} X_k^{(k-i,*)}\big|\mathcal{E}_i\big].
\end{align*}
Then, since $\E X_i = 0$, the triangle and H\"{o}lder's inequality yield
\begin{align*}
\big|\E X_i X_j X_k\big| &= \big|\E \big[ X_i \E\big[X_j X_k\big|\mathcal{E}_i\big]\big]\big| = \big|\E \big[ X_i \E\big[X_j X_k - (X_j X_k)^{(k-i,*)}\big|\mathcal{E}_i\big]\big]\big| \\&\leq \|X_i\|_3 \big(\|X_j\|_3 \|X_k -  X_k^{(k-i,*)}\|_3 + \|X_k\|_3 \|X_j -  X_j^{(j-i,*)}\|_3\big)\\&=\|X_0\|_3^2 \big(\lambda_{k-i,3} + \lambda_{j-i,3}\big).
\end{align*}
Similarly, one derives that
$$
\big|\E X_i X_j X_k \big| \leq \|X_i\|_3 \|X_j\|_3 \|X_k - X_k^{(k-j,*)}\|_3  \leq \|X_0\|_3^2 \lambda_{k-j,3}.
$$
Then by the above
\begin{align*}
&\big|\E S_n^3\big| \lesssim \sum_{1 \leq i \leq j \leq k \leq n}\big|\E X_i X_j X_k\big| \\ &\lesssim \sum_{i = 1}^n \sum_{j = i}^n\Big(\sum_{k = j}^{2j-i-1} \lambda_{j-i,3} + \sum_{k = 2j-i}^n \lambda_{k-j,3}  \Big) \\&\lesssim \sum_{i = 1}^n \sum_{j = 1}^{\infty} j \lambda_{j,3} + \sum_{i = 1}^n \sum_{j = i}^{\infty} \sum_{k = 1}^{\infty} \lambda_{k+j-i,3} \lesssim n.
\end{align*}

{\bf(ii)}: Using \hypertarget{littlelamkp:eq16}{the} results of {\bf (i)}, we obtain \hypertarget{sn:eq16}{by} stationarity
\begin{align*}
&\big|\E S_n^3 - n \sum_{i,j \in \Z} \E X_0 X_i X_j \big| \lesssim \sum_{i,j \in \N} \big(n \wedge (i \vee j )\big)\big|\E X_0 X_i X_j\big| \\&\lesssim \sum_{1 \leq i \leq j} \big(n \wedge j\big) \big|\E X_0 X_i X_j \big| \lesssim \sum_{j = 1}^n j \sum_{i = 1}^{j/2} \lambda_{j-i,3} + n \sum_{j > n} \sum_{i = 1}^{j/2} \lambda_{j-i,3} \\&+ \sum_{j = 1}^n j \sum_{i = j/2}^j \lambda_{i,3} + n \sum_{j > n} \sum_{i = j/2}^j \lambda_{i,3} \lesssim \sum_{i = 1}^{\infty} i^2 \lambda_{i,3} < \infty.
\end{align*}

{\bf (iii)}: Expanding $\E S_n^4$, we obtain by stationarity
\begin{align*}
\E S_n^4 &= \sum_{i = 1}^n \E X_i^4 + 4 \sum_{1 \leq i  < j \leq n} \E X_i^3 X_j + 4 \sum_{1 \leq i  < j \leq n} \E X_i X_j^3 \\&+ 12 \sum_{1 \leq i  < j < k \leq n} \E X_i^2 X_j X_k +12 \sum_{1 \leq i  < j < k \leq n} \E X_i X_j^2 X_k \\&+ 12 \sum_{1 \leq i  < j < k \leq n} \E X_i X_j X_k^2 +6\sum_{1 \leq i  < j \leq n} \E X_i^2 X_j^2 \\&+24\sum_{1 \leq i < j < k < l \leq n} \E X_i X_j X_k X_l\\&\stackrel{def}{=} I_n + II_n + III_n + IV_n + V_n + VI_n + VII_n + VIII_n.
\end{align*}
The most difficult term to deal with is $VIII_n$, which we handle first. For $i \leq j < k \leq l$ we have
\begin{align}
\E X_i X_j (X_k X_l)^{(l-j,*)} = \E X_i X_j  \E X_k X_l.
\end{align}
Arguing as before in {\bf (i)}, it follows that
\begin{align}
\big\|X_i X_j X_k X_l - X_i X_j (X_k X_l)^{(l-j,*)} \big\|_1 \lesssim \min\big\{\lambda_{j-i,4}, \lambda_{l-j,4} + \lambda_{k-j,4}, \lambda_{l-k,4} \big\}.
\end{align}
Next, for $1 \leq i < j < k < l \leq n$, define \hypertarget{scSkj:eq17}{the} three classes of sets
\begin{align}\nonumber
\mathcal{S}_{lk} &= \big\{i,j\, : \, \,(l-k) \geq (k-j)\vee (j-i)\big\},\\ \nonumber
\mathcal{S}_{kj} &= \big\{i,l\, : \, \,(k-j) \geq (l-k)\vee (j-i)\big\},\\
\mathcal{S}_{ji} &= \big\{k,l\, : \, \,(j-i) \geq (k-j)\vee (l-k)\big\}.
\end{align}
Note that the cardinalities $|\cdot|$ are bounded by
\begin{align}
\big|\mathcal{S}_{lk}\big| \lesssim (l-k)^2, \quad  \big|\mathcal{S}_{kj}\big| \lesssim (k-j)^2, \quad \big|\mathcal{S}_{ji}\big| \lesssim (j-i)^2.
\end{align}
With $l_k > k$\footnote{$\lambda_{k,p}$ is not necessarily monotone decreasing in $k$, which is not a problem though.} for $1 \leq k \leq n$, it follows \hypertarget{VN:eq18}{that}
\begin{align*}
&\Big|\sum_{1 \leq i < j < k < l \leq n} \big(\E X_i X_j X_k X_l - \E X_i X_j (X_k X_l)^{(l-j,*)}\big) \Big|\\&\lesssim \sum_{j = 1}^n \sum_{i = 1}^j (j-i)^2 \lambda_{j-i,4} + \sum_{k = 1}^n \sum_{j = 1}^k (k-j)^2(\lambda_{k-j,4} + \lambda_{l_k - j,4}) \\&+ \sum_{l = 1}^n \sum_{k = 1}^l (l-k)^2 \lambda_{l-k,4} \lesssim n + \sum_{k = 1}^n \sum_{j = 1}^{l_k} (l_k - j)^2\lambda_{l_k - j,4} \lesssim n.
\end{align*}
Next, let $v_n = \sum_{k = 1}^n \E X_k X_0$. Then since $\big|\sum_{m >j}\E X_m X_0\big| \lesssim j^{-2}$, we obtain
\begin{align*}
&24\sum_{1 \leq i < j < k < l \leq n} \E X_i X_j \E X_k X_l = 24\sum_{j = 2}^{n-2} \sum_{j<k<l\leq n}\E X_k X_l\big(v_n + \OO(j^{-2})\big) \\&= 24\sum_{j = 2}^{n-2} \sum_{j<k\leq n}\big(v_n + \OO((n-k)^{-2})\big)\big(v_n + \OO(j^{-2})\big) \\&=24 \frac{n^2}{2} v_n^2 + \OO\big(n\big) = 12 n^2 v_n^2 + \OO\big(n\big).
\end{align*}
Similarly, one obtains
\begin{align*}
I_n, II_n,III_n,V_n \lesssim n,
\end{align*}
and
\begin{align*}
IV_n + VI_n = 12 n^2\E X_0^2 v_n + \OO\big(n\big), \quad VII_n = 3 n^2 \E\big[X_0^2\big]^2+ \OO\big(n\big).
\end{align*}
Piecing everything together, the claim follows. Finally, {\bf (iv)} follows analogous to {\bf (i)}.
\end{proof}

For the next results, we require some \hypertarget{BVJS:eq17}{more} notation. For $1 \leq j \leq N$, denote by
\begin{align}\label{defn_Vj_star}
&\sqrt{2m}\,\overline{V}_j^* = \sum_{k = (2j-2)m + 1}^{(2j - 1)m} X_k^{(k - (2j-2)m,*)}.
\end{align}
Observe that $\overline{V}_j^* \stackrel{d}{\neq}\overline{V}_j$, however, this difference is negligible in the sequel. What is more important is that $\overline{V}_j^*$ is independent of $\FF_m$ for all $1 \leq j \leq N$ by construction. Similarly, for $1 \leq j \leq N$ ($\widetilde{V}_0$ is degenerate) we \hypertarget{VVj:eq18}{put}
\begin{align}\label{defn_Vj_star_II}
&\sqrt{2m}\widetilde{V}_j^* = \sum_{k = (2j-1)m + 1}^{2jm} X_k^{(k - (2j-1)m,*)}.
\end{align}
Note that $S_m/\sqrt{2m} \stackrel{d}{=} \overline{V}_j^* \stackrel{d}{=} \widetilde{V}_j^*$.

\begin{lem}\label{lem_polylipschitz_distance}
Let $g(\cdot)$ be a three times differentiable function with $|g^{(s)}| \leq \Co_g$ for $s \in \{0,1,2,3\}$. Assume that Assumption \ref{ass_dependence} holds. Then
\begin{align*}
\big\|\E_{\FF_m}\big[g\big(\overline{V}_j\big)-g\big(\overline{V}_j^*\big)\big]\big\|_1\lesssim \Co_g m^{-1},
\end{align*}
where $\overline{V}_j^{*}$ is as in \eqref{defn_Vj_star}.
\end{lem}

\begin{proof}[Proof of Lemma \ref{lem_polylipschitz_distance}]
For notational convenience, assume that $j = 1$. All \hypertarget{RR1:eq18}{other} cases follow exactly \hypertarget{barV1:eq18}{in} the same manner \hypertarget{barvv1:eq18}{by} stationarity. Define
\begin{align}\label{eq_lem_polylipschitz_distance_1}
&\overline{V}_1^{(*,>l)} = \frac{1}{\sqrt{2m}}\sum_{k = m - l}^{m} X_k^{*}, \quad  \overline{V}_1^{(\leq l,*)} = \overline{V}_1^{*} - \overline{V}_1^{(>l,*)},\\ \label{eq_lem_polylipschitz_distance_1.5}
&\overline{V}_1^{(**,>l)} = \frac{1}{\sqrt{2m}}\sum_{k = m - l}^{m} X_k^{(k-m+l,**)}, \quad  \overline{V}_1^{(\leq l,*)} = \overline{V}_1^{*} - \overline{V}_1^{(>l,*)},\\ \label{eq_lem_polylipschitz_distance_2.5}
&R_1^{(l,**)} = \sum_{k = m + 1}^{2m} \big(X_k - \E_{\FF_m} X_k\big)^{(k-m + l,**)}.
\end{align}
Note that $\overline{V}_1^{(*,>l)}$ and $\overline{V}_1^{(**,>l)}$ are defined differently. By the triangle inequality

\begin{align}\label{eq_lem_polylipschitz_distance_3}
&\sqrt{2m}\big\|\overline{V}_1^{(>l,*)} - \overline{V}_1^{(>l,**)} \big\|_p \leq \sum_{k = 1}^{\infty} \big\|X_k - X_k^*\big\|_p < \infty.
\end{align}
For $m < k \leq 2m$ and $l > 0$ we have due to $X_k = f_m(\varepsilon_k,\ldots,\varepsilon_{k-m})$

\begin{align}\label{eq_lem_polylipschitz_distance_3.5}
\big(\E_{\FF_m}X_k\big)^{(k-m + l,**)} = \E_{\FF_m}X_k  = \E_{\FF_m}X_k^{(k-m+l,**)}.
\end{align}
To see this, note the existence of a measurable function $H(\varepsilon_k, \ldots,\varepsilon_{m+1})$ such that
$\E_{F_m} X_k = H(\varepsilon_k, \ldots, \varepsilon_{m+1})$. Since $l>0$ and $m< k \le 2 m$, $H(\varepsilon_k, \ldots, \varepsilon_{m+1})^{(k-m+l, **)} = H(\varepsilon_k, \ldots, \varepsilon_{m+1})$.
This explains the LHS identity of \eqref{eq_lem_polylipschitz_distance_3.5}. On the other hand, since $k-m+l \ge k-m+1$, we have $\E_{F_m} X_k = \E_{F_m} X_k ^{(k-m+l, **)}$. This explains the RHS identity of \eqref{eq_lem_polylipschitz_distance_3.5}.

Using \eqref{eq_lem_polylipschitz_distance_3.5} and the triangle inequality gives
\begin{align}\label{eq_lem_polylipschitz_distance_4}
&\big\|R_1 - R_1^{(l,**)}\big\|_p \leq l^{-2}\sum_{k = l}^{\infty}k^2\big\|X_k - X_k^*\big\|_p \lesssim l^{-2}.
\end{align}
We are now ready to proceed to the actual proof. By Taylor expansion and Lemma \ref{lem_bound_R1}
\begin{align}\label{eq_lem_polylipschitz_distance_5} \nonumber
&\big\|\E_{\FF_m}\big[g\big(\overline{V}_1\big)-g\big(\overline{V}_1^*\big)\big] - \E_{\FF_m}\big[g^{(1)}\big(\overline{V}_1^*\big)(\overline{V}_1 - \overline{V}_1^*)\big] \big\|_1 \\&\lesssim \Co_g\big\|\overline{V}_1 - \overline{V}_1^* \big\|_2^2 \lesssim \Co_g m^{-1}.
\end{align}

Next, we further investigate
\begin{align*}
g^{(1)}\big(\overline{V}_1^*\big)(\overline{V}_1 - \overline{V}_1^*) = g^{(1)}\big(\overline{V}_1^{*} \big)(2m)^{-\frac{1}{2}}R_1  + g^{(1)}\big(\overline{V}_1^*\big)(\overline{V}_1 - (2m)^{-\frac{1}{2}}R_1 - \overline{V}_1^*).
\end{align*}

{\bf Case} $g^{(1)}\big(\overline{V}_1^{*} \big)R_1$:
By another Taylor expansion, H\"{o}lder's inequality and Lemma \ref{lem_bound_R1}
\begin{align}\label{eq_lem_polylipschitz_distance_6} \nonumber
&\big\|\big(g^{(1)}\big(\overline{V}_1^*\big) - g^{(1)}\big(\overline{V}_1^{(\leq l,*)} \big) - g^{(2)}\big(\overline{V}_1^{(\leq l,*)} \big) \overline{V}_1^{(> l,*)} \big)R_1 \big\|_1 \\&\lesssim \Co_g \big\|R_1\big\|_3 \big\|\overline{V}_1^{(>l,*)}\big\|_3^2 \lesssim \Co_g m^{-1} l,
\end{align}
where $\big\|\overline{V}_j^{(>l,*)}\big\|_3^2 \lesssim l$ follows from Lemma \ref{lem_wu_original}.

{\bf Subcase} $g^{(1)}\big(\overline{V}_1^{(\leq l,*)} \big)R_1$: We first note that by \eqref{eq_lem_polylipschitz_distance_4}

\begin{align*}
\big\|g^{(1)}\big(\overline{V}_1^{(\leq l,*)} \big)(R_1 - R_1^{(l,**)})\big\|_1 & \leq \Co_g \big\|R_1 - R_1^{(l,**)}\big\|_1 \lesssim \Co_g l^{-2}.
\end{align*}

Since $\overline{V}_1^{(\leq l,*)}$ and $R_1^{(l,**)}$ are independent with respect to $\P_{\FF_m}$, we have
\begin{align*}
\E_{\FF_m}g^{(1)}\big(\overline{V}_1^{(\leq l,*)} \big)R_1^{(l,**)} = \E_{\FF_m}g^{(1)}\big(\overline{V}_1^{(\leq l,*)} \big) \E_{\FF_m}R_1^{(l,**)} = 0.
\end{align*}

Hence by the above
\begin{align}\label{eq_lem_polylipschitz_distance_7}
\big\|\E_{\FF_m}g^{(1)}\big(\overline{V}_1^{(\leq l,*)} \big)R_1\big\|_1 \lesssim \Co_g l^{-2}.
\end{align}

{\bf Subcase} $g^{(2)}\big(\overline{V}_j^{(\leq l,*)} \big)R_j \overline{V}_j^{(> l,*)}$: By Cauchy-Schwarz, \eqref{eq_lem_polylipschitz_distance_3}, \eqref{eq_lem_polylipschitz_distance_4}, Lemma \ref{lem_wu_original} and Lemma \ref{lem_bound_R1}
\begin{align*}
&\big\|g^{(2)}\big(\overline{V}_1^{(\leq l,*)} \big)R_1 \overline{V}_j^{(> l,*)} - g^{(2)}\big(\overline{V}_1^{(\leq l,*)} \big)R_1^{(l,**)} \overline{V}_1^{(> l,**)} \big\|_1 \\&\lesssim \Co_g \big\|R_1 -  R_1^{(l,**)}\big\|_2 \big\|\overline{V}_1^{(> l,*)}\big\|_2 + \Co_g \big\|R_1^{(l,**)}\big\|_2 \big\|\overline{V}_1^{(> l,*)} - \overline{V}_1^{(> l,**)}\big\|_2 \\&\lesssim \Co_g l^{-2} + \Co_g^{(2)} m^{-\frac{1}{2}}.
\end{align*}

Since $\overline{V}_1^{(\leq l,*)}$ and $R_1^{(l,**)}\overline{V}_1^{(> l,**)}$ are $\P_{\FF_m}$-independent, we have
\begin{align*}
\E_{\FF_m}g^{(2)}\big(\overline{V}_1^{(\leq l,*)} \big)R_1^{(l,**)} \overline{V}_1^{(> l,**)} = \E_{\FF_m}g^{(2)}\big(\overline{V}_1^{(\leq l,*)} \big) \E_{\FF_m}R_1^{(l,**)} \overline{V}_1^{(> l,**)}.
\end{align*}
Consequently, we deduce
\begin{align*}
\big\|\E_{\FF_m}g^{(2)}\big(\overline{V}_1^{(\leq l,*)} \big)R_1^{(l,**)} \overline{V}_1^{(> l,**)}\big\|_1 \lesssim \Co_g \big\|\E_{\FF_m}R_1^{(l,**)} \overline{V}_1^{(> l,**)}\big\|_{1}.
\end{align*}
Moreover, by independence of $X_j^{(j-m+l,**)}$ from $\FF_m$ for $m-l \leq j \leq m$, we have
$$
\E_{\FF_m}X_j^{(k-m+l,**)} = \E X_j = 0,
$$
and hence
\begin{align*}
\sqrt{2m}\big\|\E_{\FF_m}R_j^{(l,**)} \overline{V}_j^{(> l,**)} \big\|_{1} = \Big\|\sum_{j = m-l}^m \sum_{k = m+1}^{2m} \E_{\FF_m}X_j^{(k-m+l,**)} X_k^{(k-m+l,**)} \Big\|_{1}.
\end{align*}
By independence of $X_j^{(j-m+l,**)}$ and $X_k^{(k-j,*)}$ for $m-l \leq j \leq m$, $m+1\leq k \leq 2m$
\begin{align*}
&\sum_{j = m-l}^m \sum_{k = m+1}^{2m} \E_{\FF_m}X_j^{(j-m+l,**)} X_k^{(k-j,*)} \\&= \sum_{j = m-l}^m \sum_{k = m+1}^{2m} \E_{\FF_m}X_j^{(j-m+l,**)}  \E_{\FF_m}X_k^{(k-j,*)} = 0,
\end{align*}
where we used $\E_{\FF_m} X_j^{(j-m+l,**)} = \E X_j = 0$. By Cauchy-Schwarz inequality and the above
\begin{align*}
&\sqrt{2m}\Big\|\E_{\FF_m}R_j^{(l,**)} \overline{V}_j^{(> l,**)} \Big\|_{1} \\&= \Big\|\sum_{j = m-l}^m \sum_{k = m+1}^{2m} \E_{\FF_m}\big(X_k^{(k-m+l,**)} - X_k^{(k-j-1,*)}\big) X_j^{(k-m+l,**)} \Big\|_{1} \\&\leq \sum_{j = m-l}^m \sum_{k = m+1}^{2m} \big\|X_k^{(k-m+l,**)} - X_k^{(k-j,*)}\big\|_2 \big\|X_j\big\|_2 \\&= \sum_{j = 0}^l \sum_{k = 1}^{m} \big\|X_k^{} - X_k^{(k+j,*)}\big\|_2 \big\|X_j\big\|_2 \lesssim \sum_{k = 1}^{\infty} k \big\|X_k - X_k^*\big\|_2.
\end{align*}
Piecing everything together, we finally obtain
\begin{align*}
\big\|\E_{\FF_m}g^{(2)}\big(\overline{V}_1^{(\leq l,*)} \big)R_1 \overline{V}_1^{(> l,*)} \big\|_1 \lesssim \Co_g l^{-2} + \Co_g m^{-\frac{1}{2}}.
\end{align*}
Selecting $l = m^{\frac{1}{3}}$ then yields
\begin{align}\label{eq_lem_polylipschitz_distance_8}
\big\|\E_{\FF_m}g^{(1)}\big(\overline{V}_1^{*} \big)R_1 \big\|_1 \lesssim \Co_g m^{-\frac{1}{2}}.
\end{align}

{\bf Case} $g^{(1)}\big(\overline{V}_1^*\big)(\overline{V}_1 - (2m)^{-\frac{1}{2}}R_1 - \overline{V}_1^*)$: We have
\begin{align*}
\sqrt{2m}g^{(1)}\big(\overline{V}_1^*\big)(\overline{V}_1 - (2m)^{-\frac{1}{2}}R_1 - \overline{V}_1^*) =  \sum_{k = 1}^m g^{(1)}\big(\overline{V}_1^*\big) \big(X_k - X_k^* - \E_{\FF_m}X_k\big).
\end{align*}
Using $\|E_{\FF_m}X_k\|_p \leq \|X_k - X_k^*\|_p$ for $1 \leq k \leq m$, the Lipschitz-continuity of $g^{(1)}$, Cauchy-Schwarz and Lemma \ref{lem_wu_original}, we get
\begin{align*}
&\Big\|\sum_{k = 1}^m g^{(1)}\big(\overline{V}_1^*\big) \big(X_k - X_k^* - \E_{\FF_m} X_k\big) \\& \quad \quad-\sum_{k = 1}^m g^{(1)}\big(\overline{V}_1^{(> m-k-1,*)}\big) \big(X_k - X_k^* - \E_{\FF_m} X_k \big)\Big\|_1\\& \lesssim \Co_g \sum_{k = 1}^m \big\|\overline{V}_1^{(\leq m-k-1,*)} \big\|_2 \big\|X_k - X_k^* - \E_{\FF_m} X_k\big\|_2 \\&\lesssim \Co_g m^{-\frac{1}{2}} \sum_{k = 1}^m k^{\frac{1}{2}}\big\|X_k - X_k^*\big\|_2 \lesssim \Co_g m^{-\frac{1}{2}},
\end{align*}
where $\overline{V}_1^{(> -1)} = 0$. Similarly, we obtain from \eqref{eq_lem_polylipschitz_distance_3}
\begin{align*}
&\Big\|\sum_{k = 1}^m g^{(1)}\big(\overline{V}_1^{(> m-k-1,*)}\big) \big(X_k - X_k^* - \E_{\FF_m} X_k\big)\\& \quad \quad -\sum_{k = 1}^m g^{(1)}\big(\overline{V}_1^{(> m-k-1,**)}\big) \big(X_k - X_k^* - \E_{\FF_m}X_k\big)\Big\|_1 \\&\lesssim \Co_g \sum_{k = 1}^m \big\|\overline{V}_1^{(> m-k-1,*)}-\overline{V}_1^{(> m-k-1,**)}\big\|_2 \big\|X_k - X_k^* - \E_{\FF_m} X_k\big\|_2 \\&\lesssim \Co_g m^{-\frac{1}{2}} \sum_{k = 1}^{\infty}\big\|X_k - X_k^*\big\|_2 \lesssim \Co_g m^{-\frac{1}{2}}.
\end{align*}
Note that $\overline{V}_1^{(> m-k-1,**)}$ and $X_k - X_k^* - \E_{\FF_m} X_k$ are $\P_{\FF_m}$-independent. Since
\begin{align*}
\E_{\FF_m}\big[(X_k - X_k^* - \E_{\FF_m}X_k)\big] = 0 - \E_{\FF_m}X_k^* = 0
\end{align*}
for $1 \leq k \leq m$, it follows that
\begin{align*}
&\sum_{k = 1}^m\E_{\FF_m}\big[g^{(1)}(\overline{V}_1^{(> m-k-1,**)})(X_k - X_k^* - \E_{\FF_m}[X_k])\big] \\&= \sum_{k = 1}^m\E_{\FF_m}\big[g^{(1)}(\overline{V}_1^{(> m-k-1,**)})\big] \E_{\FF_m}\big[(X_k - X_k^* - \E_{\FF_m}[X_k])\big] = 0.
\end{align*}
Piecing everything together, we finally obtain
\begin{align}\label{eq_lem_polylipschitz_distance_9}
\big\|g^{(1)}\big(\overline{V}_1^*\big)(\overline{V}_1 - (2m)^{-1/2}R_1 - \overline{V}_1^*) \big\|_1 \lesssim \Co_g m^{-1}+\Co_g m^{-1}.
\end{align}
Combining \eqref{eq_lem_polylipschitz_distance_5}, \eqref{eq_lem_polylipschitz_distance_8} and \eqref{eq_lem_polylipschitz_distance_9} yields
\begin{align*}
\big\|\E_{\FF_m}g\big(\overline{V}_1\big)-g\big(\overline{V}_1^*\big)\big]\big\|_1\lesssim \Co_g m^{-1}.
\end{align*}
\end{proof}

\begin{lem}\label{lem_polylipschitz_distance_moments}
Let $g(\cdot)$ be a third-degree polynomial \hypertarget{BVJS:eq22}{with} coefficients bounded by $\Co_g$. Assume that Assumption \ref{ass_dependence} holds. Then
\begin{align*}
\big\|\E_{\FF_m}\big[g\big(\overline{V}_j\big)-g\big(\overline{V}_j^*\big)\big]\big\|_1\lesssim \Co_g m^{-1},
\end{align*}
where $\overline{V}_j^{*}$ is as in \eqref{defn_Vj_star}. This bound is also valid if $g(\cdot)$ is a fourth-degree polynomial and Assumption \ref{ass_dependence} holds for $p \geq 4$.
\end{lem}

\begin{proof}[Proof of Lemma \ref{lem_polylipschitz_distance_moments}]
The proof is almost identical to the one of Lemma \ref{lem_polylipschitz_distance}. The only difference is that we use H\"{o}lder's inequality and corresponding moment bounds instead of bounding derivatives with $\Co_g$. Moreover, due to the polynomial structure, no Taylor expansions are necessary.
\end{proof}

\begin{lem}\label{lem_polylipschitz_distance_tilde}
Let $g(\cdot)$ be a three times differentiable function with $|g^{(s)}| \leq \Co_g$ for $s \in \{0,1,2,3\}$. Assume that Assumption \ref{ass_dependence} holds. \hypertarget{scFmstar:eq22}{Then}
\begin{align*}
\big\|g\big(\widetilde{V}_j\big)-g\big(\widetilde{V}_j^*\big)\big\|_1 \lesssim \Co_g m^{-1},
\end{align*}
where $\widetilde{V}_j^{*}$ is defined in \eqref{defn_Vj_star_II}.
\end{lem}

\begin{proof}[Proof of Lemma \ref{lem_polylipschitz_distance_tilde}]
For simplicity, we set $j = 1$. Let $\FF_m^{*} = \sigma\big(\FF_m, \varepsilon_1', \ldots, \varepsilon_m'\big)$. Then
\begin{align*}
\sum_{k = m+1}^{2m} \E_{\FF_m}X_k = \sum_{k = m+1}^{2m} \E_{\FF_m^{*}}\big[X_k-X_k^{(k-m,*)}\big] + X_k^{(k-m,*)}.
\end{align*}
We may now proceed exactly as in the proof of Lemma \ref{lem_polylipschitz_distance}. This includes $j = N$, requiring only a slight adaptation.
\end{proof}

\begin{lem}\label{lem_polylipschitz_distance_moments_tilde}
Let $g(\cdot)$ be a third-degree polynomial with coefficients bounded by $\Co_g$. Assume that Assumption \ref{ass_dependence} holds. Then
\begin{align*}
\big\|g\big(\widetilde{V}_j\big)-g\big(\widetilde{V}_j^*\big)\big\|_1\lesssim  m^{-1},
\end{align*}
where $\widetilde{V}_j^{*}$ is as in \eqref{defn_Vj_star_II}. This bound is also valid if $g(\cdot)$ is a fourth-degree polynomial and Assumption \ref{ass_dependence} holds for $p \geq 4$.
\end{lem}

\begin{proof}[Proof of Lemma \ref{lem_polylipschitz_distance_moments_tilde}]
The proof is almost identical to the one of Lemma \ref{lem_polylipschitz_distance_tilde}. The only difference is that we use H\"{o}lder's inequality and corresponding moment bounds, instead of bounding derivatives with $\Co_g$.
\end{proof}

\begin{lem}\label{lem_sig_expansion}
Grant Assumption \ref{ass_dependence}. Then
\begin{description}
\item[(i)]$\bigl\|\overline{\sigma}_{j|m}^2 - \overline{\sigma}_j^2\bigr\|_{\frac{p}{2}} \lesssim \bigl\|\overline{\sigma}_{j|m}^2 - \frac{1}{2}{s}_m^2\bigr\|_{\frac{p}{2}} + m^{-1} \lesssim m^{-1}$ for $1 \leq j \leq N$,
\item[(ii)] $2m\overline{\sigma}_j^2 = {s}_m^2 + \OO(1)$ for $1 \leq j \leq N$.
\item[(iii)] $\big\|2\overline{S}_{n|m}^2 - s_n^2\big\|_{p/2} \lesssim  m^{-1} N^{\frac{2}{p\wedge 4}}$.
\end{description}
\end{lem}

\begin{proof}[Proof of Lemma \ref{lem_sig_expansion}]
See ~\cite{jirak_be_aop_2016}. Alternatively, one may also use Lemma \ref{lem_polylipschitz_distance_moments} for {\bf (i)}, which implies  {\bf (ii)}.
\end{proof}

\begin{lem}\label{lem_cum_expansion}
Grant Assumption \ref{ass_dependence}. Then
\begin{description}
  \item[(i)] 
$\bigl\|\overline{\kapp}_{j|m}^3 - (2m)^{-\frac{3}{2}}\E S_m^3\bigr\|_{1}, \bigl\|\overline{\kapp}_{j|m}^3 - \overline{\kapp}_j^3\bigr\|_{1} \lesssim m^{-1}$ for $1 \leq j \leq N$,
  \item[(ii)] $\E_{}{S}_{n}^3 = \E_{}\overline{S}_{n|m}^3 + \E_{}\widetilde{S}_{n|m}^3 + \OO\big(n^{\frac{1}{2}} N^{\frac{2}{p\wedge 4}}\big) \\ \hspace{6cm} =(2m)^{\frac{3}{2}}\sum_{j = 1}^N \big(\overline{\kapp}_j^3 + \widetilde{\kapp}_j^3\big) +\OO\big(n^{\frac{1}{2}} N^{\frac{2}{p\wedge 4}}\big)$.
\end{description}
\end{lem}

\begin{rem}
Lengthy calculations even reveal the bound $m^{-\frac{3}{2}}$ in {\bf (i)}.
\end{rem}

\begin{proof}[Proof of Lemma \ref{lem_cum_expansion}]
Regarding {\bf (i)}, the first claim follows immediately from Lemma \ref{lem_polylipschitz_distance_moments}. For the second claim, conditioning, the triangle inequality and the first result give
\begin{align*}
\bigl\|\overline{\kapp}^3_{j|m} - \overline{\kapp}^3_{j} \big\|_1 &\lesssim m^{-1} + \bigl\|\overline{\kapp}^3_{j} -  (2m)^{-\frac{3}{2}}\E S_m^3\big\|_1 \\&\lesssim \bigl\|\overline{\kapp}^3_{j|m} - (2m)^{-\frac{3}{2}}\E S_m^3 \big\|_1 + m^{-1} \lesssim m^{-1}.
\end{align*}
For {\bf (ii)}, since $\E_{\FF_m}\overline{S}_{n|m}=0$, conditioning gives
\begin{align*}
&\E_{}{S}_{n}^3 = \E\overline{S}_{n|m}^3 + 3\E\big[\widetilde{S}_{n|m} \E_{\FF_m}[\overline{S}_{n|m}^2]\big] + \E\widetilde{S}_{n|m}^3.
\end{align*}
Hence it suffices to show
\begin{align}
\big|\E\big[\widetilde{S}_{n|m}\E_{\FF_m}[\overline{S}_{n|m}^2] \big]\big| \lesssim n^{\frac{1}{2}} N^{\frac{2}{p \wedge 4}}.
\end{align}
Since $\E \widetilde{S}_{n|m} = 0$, we have
\begin{align*}
\big|\E\big[\widetilde{S}_{n|m}\E_{\FF_m}[\overline{S}_{n|m}^2] \big]\big| = \big|\E\big[\widetilde{S}_{n|m}\big(\E_{\FF_m}[\overline{S}_{n|m}^2] - \E_{}[\overline{S}_{n|m}^2]\big)\big]\big|.
\end{align*}
From Lemmas \ref{lem_sig_expansion}, \ref{lem_bound_R1} and H\"{o}lder's inequality, we get the first equality in {\bf (ii)}. Since $(\overline{V}_j)_{1 \leq j \leq N}$ is $\P_{\FF_m}$-independent and $(\widetilde{V}_j)_{1 \leq j \leq N}$ is $\P$-independent, we get the second equality.
\end{proof}

\begin{lem}\label{lem_quad_expansion}
Grant Assumption \ref{ass_dependence}. Then for $1 \leq j \leq N$
\begin{align*}
\bigl\|\overline{\nu}_{j|m}^4 - (2m)^{-2}\E S_m^4\bigr\|_{1}, \bigl\|\overline{\nu}_{j|m}^4 - \overline{\nu}_j^4\bigr\|_{1} \lesssim m^{-1}.
\end{align*}
\end{lem}

\begin{proof}[Proof of Lemma \ref{lem_quad_expansion}]
We may argue \hypertarget{barsigmajj:eq23}{as} in the proof \hypertarget{tilsigmajj:eq23}{of} Lemma \ref{lem_cum_expansion}.
\end{proof}

\begin{lem}\label{lem_widetilde_moments}
 \hypertarget{barNuj:eq23}{Grant} Assumption \ref{ass_dependence}. \hypertarget{tilkappaj:eq23}{Then} \hypertarget{tildeNuj:eq23}{for} $1 \leq j \leq N$ 
\begin{description}
  \item[(i)] $\widetilde{\sigma}_j^2 = \overline{\sigma}_j^2 + \OO(m^{-1})$,
  \item[(ii)] $|\widetilde{\kappa}_j^3 - \overline{\kappa}_j^3|\lesssim m^{-1}$,
  \item[(iii)] $|\widetilde{\nu}_j^4 - \overline{\nu}_j^4|\lesssim m^{-1}$ if $p \geq 4$.
\end{description}
\end{lem}

\begin{proof}[Proof of Lemma \ref{lem_widetilde_moments}]
All three results follow from Lemma \ref{lem_polylipschitz_distance_moments_tilde} and Lemmas \ref{lem_sig_expansion}, \ref{lem_cum_expansion} and \ref{lem_quad_expansion}. Alternatively, one may also prove it directly.

\end{proof}

\subsection{Conditional Approximations and Distributions}

We \hypertarget{BVJS:eq23}{require} some additional notation. Recall the definitions of $\overline{V}_j^{\ast}$ and $\widetilde{V}_j^{\ast}$ in \eqref{defn_Vj_star} and \eqref{defn_Vj_star_II}. Since
$\overline{V}_j^{\ast} \stackrel{d}{=} \widetilde{V}_j^{\ast}$ and
\begin{align}\label{Z_equality}
\overline{Z}_j \stackrel{d}{=}\widetilde{Z}_j \stackrel{d}{=} Z,\quad 1 \leq j \leq N
\end{align}
for some random variable $Z$, we \hypertarget{DeltaJMX:eq24}{put}
\begin{align}\label{defn_Delta_star_diamond}
{\Delta}_{j,m}(x) = \overline{\Delta}_{j,m}(x) &= \P_{}\bigl(\overline{V}_j^* \geq x \bigr) - \P_{}\bigl(\overline{Z}_j \geq x \bigr), \quad 1 \leq j \leq N,
\end{align}
$\widetilde{\Delta}_{j,m} = {\Delta}_{j,m}$ for $1 \leq j \leq N$. Observe that ${\Delta}_{j,m}$ actually do not depend on $j$, but we stick to this notation to distinguish them from $\Delta_m$ and $\Delta_m^{\diamond}$. Recall that $\xi_N = \xi/\sqrt{N}$.




\begin{lem}\label{lem_taylor_expansion_smooth_function}
Grant Assumption \ref{ass_dependence}, and let $f$ be a smooth function such that $\sup_{x \in \R}|f^{(s)}(x)| \leq 1$ for $s = 0,\ldots,7$. Then for $\tau_n \geq {\co_{\tau}} \sqrt{\log n}$, ${\co_{\tau}} > 0$ sufficiently large, $2\bd + 2 < \ad$ and $p \geq 3$
\begin{align*}
&\, \max_{1 \leq j \leq N}\bigl\|\E_{\FF_m}\ff\bigl(\xi_N \overline{V}_{j}\bigr) - \E_{}\ff\bigl(\xi_N \overline{Z}_{j}\bigr)\bigr\|_1\\& \quad \lesssim |\xi_N|^3 \tau_n^{3-p} \log(n)^{-\bd} \oo\big(m^{1-\frac{p}{2}}\big) + \tau_n^4(|\xi_N|^4 + |\xi_N|^7)m^{-1} \\&\quad +\sup_{x \in \R}\big|{\Delta}_{j,m}(x)\big|\bigl(\tau_n^4|\xi_N|^4 + |\tau_n|^5 |\xi_N|^5\big) + |\xi_N|^3 m^{-1}+ |\xi_N|^2 m^{-1}.
\end{align*}
\end{lem}

\begin{proof}[Proof of Lemma \ref{lem_taylor_expansion_smooth_function}]
The overall proof is lengthy and consists of several parts. Some key technical results are deferred to subsequent lemmas.\\
{\bf Step 1}: \textit{Taylor expansion}.
Recall the following Taylor expansion
\begin{align}\label{taylor_expansion}\nonumber
f\bigl(x + h\bigr) - f\bigl(x\bigr) &= \sum_{j = 1}^{s}\frac{f^{(j)}\bigl(x\bigr) h^j}{j!} + \frac{f^{(s+1)}\bigl(x\bigr) h^{s+1}}{(s+1)!} \\&+ \frac{h^{s+1}}{s!}\int_0^1 (1-t)^s\bigl(f^{(s+1)}(t h + x) - f^{(s+1)}(x)\bigr)\,dt,
\end{align}
that we use in the sequel. Expanding at $x = 0$ we have
\begin{align}\nonumber \label{eq_lem_taylor_exp_5}
&\E_{\FF_m}\ff\bigl(\xi_N \overline{V}_{j}\bigr) - \E_{}\ff\bigl(\xi_N \overline{Z}_{j}\bigr) = \frac{\xi_N^2}{2}(\overline{\sigma}_{j|m}^2 - \overline{\sigma}_j^2)\ff^{(2)}(0) + \frac{\xi_N^3}{6}(\overline{\kapp}_{j|m}^3 - \overline{\kapp}_j^3)\ff^{(3)}(0)\\ \nonumber &+\frac{1}{2}\int_0^1 (1-t)^2 \E_{|\FF_m}\Bigl[(\xi_N\overline{V}_j)^{3} \big(\ff^{(3)}(t \xi_N \overline{V}_j) - \ff^{(3)}(0) \bigr)\Bigr]\,dt  \\&-\frac{1}{2}\int_0^1 (1-t)^2 \E_{}\Bigl[(\xi_N\overline{Z}_{j})^{3} \big(\ff^{(3)}(t \xi_N \overline{Z}_{j}) - \ff^{(3)}(0)\big)\Bigr]\,dt.
\end{align}
By Lemma \ref{lem_sig_expansion} and Lemma \ref{lem_cum_expansion}, the first part is bounded by
\begin{align}\label{eq_lem_taylor_exp_5.5}
\big\|\frac{\xi_N^2}{2}(\overline{\sigma}_{j|m}^2 - \overline{\sigma}_j^2)\ff^{(2)}(0) + \frac{\xi_N^3}{6}(\overline{\kapp}_{j|m}^3 - \overline{\kapp}_j^3)\ff^{(3)}(0) \big\|_1 \lesssim \xi_N^2 m^{-1} + |\xi_N|^3 m^{-1}.
\end{align}
In the next steps, we deal with the second residual term.\\

{\bf Step 2}: \textit{Truncation of residual term}. Let $\tau_n \geq {\co_{\tau}} \sqrt{\log n}$ with ${\co_{\tau}} > 0$ to be specified, and $\hh_n(x)$ be \hypertarget{hnx:eq24}{a} three times continuously differentiable function such that
\begin{align}\label{defn_hn}
\hh_n(x)  = \left\{\begin{array}{cl} 1, & \mbox{if $|x| \leq \tau_n/2$,}\\ 0, & \mbox{if $|x| \geq \tau_n$} \end{array}\right.
\end{align}
and $|\hh_n^{(s)}(x)| \leq \co$ for $s \in \{0,1,2,3\}$. For a random variable $X$ and $q\geq 1$ we have
\begin{align}\label{eq_lem_taylor_exp_6}
\E\bigl[|X|^q \ind(|X| \geq \tau_n)\bigr] \leq q\tau_n^q \P\bigl(|X| \geq \tau_n \bigr) +  q \int_{\tau_n}^{\infty} x^{q-1} \P\bigl(|X| \geq x \bigr)d\,x.
\end{align}
Lemma \ref{th:Moritz} then yields that for all $1 \leq j \leq N$
\begin{align}\nonumber \label{eq_lem_taylor_exp_7}
&\E\bigl[|\overline{V}_j|^3 \big(1 - \hh_n(\overline{V}_j)\big) \bigr] \leq \E |\overline{V}_j|^3 \ind(|\overline{V}_j| \geq \tau_n/2)\\& \nonumber \lesssim  m^{1-p/2} \tau_n^{3-p} \oo\big(\log(\tau_n)^{-\bd}\big) + \oo\big(m^{1-p/2}\big) \int_{\tau_n/2}^{\infty} x^{2-p} \log(n x)^{-\ad/2} d\,x\\& \lesssim m^{1-p/2} \tau_n^{3-p} \oo\big(\log(\tau_n)^{-\bd}\big),
\end{align}
for sufficiently large $\co_{\tau} > 0$ ($\tau_n \geq {\co_{\tau}} \sqrt{\log n}$). Since $|\ff^{(3)}| \leq 1$, we thus obtain by Jensen's inequality
\begin{align}\nonumber \label{eq_lem_taylor_exp_8}
&\Bigl\|\E_{\FF_m}\Bigl[(\xi_N \overline{V}_j)^{3} \big(\ff^{(3)}(t \xi_N \overline{V}_j) - \ff^{(3)}(0)\bigr)\big(1-\hh_n(\overline{V}_j)\big)\Bigr]\Bigr\|_1 \\&\lesssim |\xi_N|^3 \Bigl\||\overline{V}_j|^{3} \big(1 - \hh_n(\overline{V}_j)\big)\Bigr\|_1 \lesssim |\xi_N|^3 \tau_n^{3-p}\log(\tau_n)^{-\bd} \oo\big(m^{1-p/2}\big).
\end{align}

{\bf Step 3}: \textit{Decomposition and approximation of residual term one}.
Using another Taylor expansion we have
\begin{align}\nonumber \label{eq_lem_taylor_exp_9}
\nonumber
&\E_{\FF_m}\Bigl[(\xi_N \overline{V}_j)^{3} \big(\ff^{(3)}(t \xi_N \overline{V}_j) - \ff^{(3)}(0)\bigr)\hh_n(\overline{V}_j)\Bigr]\\& = t\xi_N^4 \int_0^1(1-s)\E_{\FF_m}\bigl[\overline{V}_j^{4}\ff^{(4)}(s t \xi_N \overline{V}_j)\hh_n(\overline{V}_j)\bigr]d\,s.
\end{align}
\hypertarget{GX:eq26}{Let}
\begin{align}\label{eq_lem_taylor_exp_g}
g(x) = x^4 \ff^{(4)}(st \xi_N x )\hh_n(x),
\end{align}
and recall that the derivatives of $\ff$ are uniformly bounded. Then $|g^{(s)}| \lesssim \tau_n^4(1 + |\xi_N|^3)$ for $s \in \{0,1,2,3\}$, and Lemma \ref{lem_polylipschitz_distance} yields for $1 \leq j \leq N$
\begin{align}\label{eq_lem_taylor_exp_10} \nonumber
&\Bigl\|\E_{\FF_m}\Bigl[t(\xi_N \overline{V}_j)^{4}\ff^{(4)}\big(st \xi_N \overline{V}_j\big) \hh_n(\overline{V}_j)\Bigr] \\&- \E_{\FF_m}\Bigl[t(\xi_N \overline{V}_j^*)^{4} \ff^{(4)}\big(st \xi_N \overline{V}_j^*\bigr)\hh_n(\overline{V}_j^*)\Bigr]\Bigr\|_1 \lesssim m^{-1} \tau_n^4(|\xi_N|^4 + |\xi_N|^7).
\end{align}
Note that in this step, the smoothness and boundedness of $\hh_n(\cdot)$ are essential.

{\bf Step 4}: \textit{Decomposition and approximation of residual term two}.
Observe $g(0) = 0$, and recall that for any random variable $Y$ and differentiable function $f$ we have
\begin{align}\label{eq_thm_smooth_4_second}
\E\bigl[f(Y) - f(0)\bigr] = \int_0^{\infty}f^{(1)}(y) \P\bigl(Y \geq y \bigr)dy - \int_{-\infty}^{0}f^{(1)}(y)\P\bigl(Y \leq y \bigr)dy.
\end{align}
Since $\hh_n(x)$, $\hh_n^{(1)}(x)$ and $g^{(1)}(x)$ vanish for $|x| > \tau_n$, we obtain from the above
\begin{align}\label{eq_lem_taylor_exp_11}\nonumber
&\E_{\FF_m}(\overline{V}_j^*)^{4}\ff^{(4)}\bigl(t \xi_N \overline{V}_j^*\bigr)\hh_n(\overline{V}_j^*)\\ \nonumber &= \E_{\FF_m}\Bigl[\int_{0}^{\tau_n} g_{}^{(1)}(x) \P_{\FF_m}\bigl(\overline{V}_j^* \geq x \bigr) d\,x - \int_{-\tau_n}^{0} g_{}^{(1)}(x) \P_{\FF_m}\bigl( \overline{V}_j^* \leq x \bigr) d\,x \Bigr] \\&= \int_{0}^{\tau_n} g_{}^{(1)}(x)\P_{}\bigl( \overline{V}_j^* \geq x \bigr) d\,x - \int_{-\tau_n}^{0} g_{}^{(1)}(x) \P_{}\bigl( \overline{V}_j^* \leq x \bigr) d\,x,
\end{align}
where we used the fact that $\overline{V}_j^*$ is independent of $\FF_m$. In \hypertarget{DeltaJMX:eq26}{addition}
\begin{align}\label{eq_lem_taylor_exp_11.6}
\Bigl\|\int_{0}^{\tau_n}g_{}^{(1)}(x)\Bigl( \P_{}\bigl( \overline{V}_j^* \geq x \bigr) - \P_{}\bigl( \overline{Z}_j \geq x \bigr)\Bigr)d\,x \Bigr\|_1 \lesssim \sup_{x \in \R}\big|{\Delta}_{j,m}(x)\big|\bigl(\tau_n^4 + |\tau_N|^5 |\xi_N|),
\end{align}
and the same bound applies to the second expression in \eqref{eq_lem_taylor_exp_11}. Hence we deduce from \eqref{eq_lem_taylor_exp_10}, \eqref{eq_lem_taylor_exp_11} and \eqref{eq_lem_taylor_exp_11.6} the bound
\begin{align}\label{eq_lem_taylor_exp_11.7} \nonumber
\Bigl\|\E_{\FF_m} t \xi_N^4  g\big(\overline{V}_j\big)  - \E_{} t \xi_N^4  g\big(\overline{Z}_j\big) \Big\|_1  \lesssim m^{-1} \tau_n^4(|\xi_N|^4 + |\xi_N|^7)
\\ \indent + \sup_{x \in \R}\big|{\Delta}_{j,m}(x)\big|\bigl(\tau_n^4|\xi_N|^4 + |\tau_N|^5 |\xi_N|^5).
\end{align}

{\bf Step 5}: \textit{Final estimate of residual term}.
Combining all bounds and equations \eqref{eq_lem_taylor_exp_8}, \eqref{eq_lem_taylor_exp_9} and \eqref{eq_lem_taylor_exp_11.7} 
we arrive at 
\begin{align}\label{eq_lem_taylor_exp_12} \nonumber
&\Bigl\|\E_{\FF_m}(\xi_N\overline{V}_j)^{3} \ff^{(3)}\bigl(t \xi_N \overline{V}_j \bigr)  - \E_{\FF_m}(\xi_N\overline{V}_j)^{3}\ff^{(3)}\bigl(0\bigr) \\ \nonumber & \quad \quad -\int_0^1 \E_{}\Bigl[(\xi_N \overline{Z}_j)^{4} \ff^{(4)}\big(st \xi_N \overline{Z}_j\big) \hh_n(\overline{Z}_j)\Bigr] d\,s\Bigr\|_1 \\ \nonumber & \lesssim  |\xi_N|^3 \tau_n^{3-p}\log(n)^{-\bd} \oo\big(m^{1-p/2}\big) + \tau_n^4(|\xi_N|^4 + |\xi_N|^7)m^{-1} \\ & \quad \quad  +\sup_{x \in \R}\big|{\Delta}_{j,m}(x)\big|\bigl(\tau_n^4|\xi_N|^4 + |\tau_N|^5 |\xi_N|^5).
\end{align}

{\bf Step 6}: \textit{Estimate of residual counter part}.
We now consider the counter part. Using similar arguments as before, together with Lemma \ref{lem_Z_approx_and_BE}, one derives the estimate
\begin{align}\nonumber \label{eq_lem_taylor_exp_13}
&\Bigl|\E_{}(\xi_N\overline{Z}_{j})^{3} \ff^{(3)}\bigl(t \xi_N \overline{Z}_{j}\bigr)  - \E_{}(\xi_N\overline{Z}_{j})^{3} \ff^{(3)}\bigl(0\bigr) \\ &-\int_0^1 \E_{}\Bigl[(\xi_N \overline{Z}_{j})^{4} \ff^{(4)}\big(st \xi_N \overline{Z}_{j}\big)\hh_n(\overline{Z}_{j})\Bigr]d\,s\Bigr| \lesssim |\xi_N|^3 m^{-2},
\end{align}
uniformly for $1 \leq j \leq N$.\\

{\bf Step 7}: \textit{Final overall bound for residual terms}.
In turn, employing the estimates \eqref{eq_lem_taylor_exp_5.5}, \eqref{eq_lem_taylor_exp_12} and \eqref{eq_lem_taylor_exp_13}, we finally obtain from equation (\ref{eq_lem_taylor_exp_5}) 
\begin{align}\nonumber \label{eq_lem_taylor_exp_17}
&\bigl\|\E_{\FF_m}\ff\bigl(\xi_N \overline{V}_{j}\bigr) - \E_{}\ff\bigl(\xi_N \overline{Z}_{j}\bigr)\bigr\|_1 \\& \nonumber \lesssim |\xi_N|^3 \tau_n^{3-p} \log(n)^{-\bd} \oo\big(m^{1-p/2}\big) + \tau_n^4(|\xi_N|^4 + |\xi_N|^7)m^{-1} \\&+\sup_{x \in \R}\big|{\Delta}_{j,m}(x)\big|\bigl(\tau_n^4|\xi_N|^4 + |\tau_N|^5 |\xi_N|^5) + |\xi_N|^3 m^{-1} + \xi_N^2 m^{-1}.
\end{align}
This completes the proof.
\end{proof}

\begin{lem}\label{lem_taylor_expansion_smooth_function_II}
Assume that Assumption \ref{ass_dependence} holds, and let $f$ be a smooth function such that $\sup_{x \in \R}|f^{(s)}(x)| \leq 1$ for $s = 0,\ldots,7$. Then for $\tau_n \geq {\co_{\tau}} \sqrt{\log n}$, ${\co_{\tau}} > 0$ sufficiently large, $2\bd +2 < \ad$ and $p \geq 3$
\begin{align*}
&\text{{\bf (i)}} \,\, \max_{1 \leq j \leq N}\bigl\|\E_{} \ff\bigl(\xi_N \widetilde{V}_{j}\bigr)  - \E_{}\ff\bigl(\xi_N \widetilde{Z}_{j}\bigr)\bigr\|_1 \\&\quad \lesssim |\xi_N|^3 \tau_n^{3-p} \log(n)^{-\bd} \oo\big(m^{1-\frac{p}{2}}\big) + \tau_n^4(|\xi_N|^4 + |\xi_N|^7)m^{-1}\\& \quad +\sup_{x \in \R}\big|{\Delta}_{j,m}(x)\big|\bigl(\tau_n^4|\xi_N|^4 + |\tau_n|^5 |\xi_N|^5\big) + \big(\xi_N^2 + |\xi_N|^3\big) m^{-1}.\\
&\text{{\bf (ii)}} \,\, \bigl\|\E_{}\ff\bigl(\xi_N \widetilde{V}_0\bigr)  - \E_{}\ff\bigl(0\bigr)\bigr\|_1 \lesssim \xi_N^2 m^{-1}.
\end{align*}
\end{lem}

\begin{proof}[Proof of Lemma \ref{lem_taylor_expansion_smooth_function_II}]
For {\bf (i)}, we may argue as in the proof of Lemma \ref{lem_taylor_expansion_smooth_function}. 
For {\bf (ii)}, it suffices to note that $\|\widetilde{V}_{0}\|_p \lesssim m^{-\frac{1}{2}}$, which follows from Lemma \ref{lem_bound_R1} {\bf (i)} together with the triangle inequality. Hence $\widetilde{\sigma}_1^2 \lesssim m^{-1}$, 
and the claim follows \hypertarget{DeltaJMX:eq27}{from} a second \hypertarget{OWJ:eq28}{order} Taylor expansion at $x = 0$.
\end{proof}

\begin{lem}\label{lem_taylor_expansion_smooth_function_I_BE}
The quantity $\sup_{x \in \R}\big|{\Delta}_{j,m}(x)\big|$, in the bounds of Lemmas \ref{lem_taylor_expansion_smooth_function} and \ref{lem_taylor_expansion_smooth_function_II}, can be replaced with $m^{-\frac{1}{2}}$.
\end{lem}

\begin{proof}[Proof of Lemma \ref{lem_taylor_expansion_smooth_function_I_BE}]
For $1 \leq j \leq N$, let $\overline{W}_j$ be a zero mean Gaussian random variable with variance $\sigma_{j}^2 = \|\overline{V}_j^*\|_2^2$. Since $\overline{V}_j^{\ast} \stackrel{d}{=} S_m/\sqrt{2m}$, Lemma \ref{lem_sig_expansion} implies $\sigma_j^2 = \overline{\sigma}_j^2 + \OO(m^{-1})$, and together with Lemma \ref{lem_sig_expressions_relations} and \hyperref[B3]{\Bthree} that $\sigma_j^2$ is bounded away from zero for $m \geq m_0$ large enough, uniformly for $1 \leq j \leq N$. Invoking Theorem 2.2 in ~\cite{jirak_be_aop_2016}, it follows that
\begin{align*}
\sup_{x \in \R}\bigl|\P\bigl(\overline{V}_j^* \leq x  \bigr) - \P\bigl(\overline{W}_j \leq x \bigr)\bigr| \lesssim m^{-\frac{1}{2}}.
\end{align*}
Hence by the above and Lemma \ref{lem_psi_compare}, the triangle inequality yields
\begin{align*}
\sup_{x \in \R}\big|\P\bigl(\overline{V}_j^* \leq x  \bigr) - \P\bigl(\overline{Z}_j \leq x \bigr)\big| &\leq \sup_{x \in \R}\big|\P\bigl(\overline{V}_j^* \leq x  \bigr) - \P\bigl(\overline{W}_j \leq x \bigr)\big| \\&+ \sup_{x \in \R}\big|\P\bigl(\overline{W}_j \leq x  \bigr) - \P\bigl(\overline{Z}_j \leq x \bigr)\big| \lesssim m^{-\frac{1}{2}}.
\end{align*}
Since $\overline{V}_j^* \stackrel{d}{=} \widetilde{V}_j^*$ for $1 \leq j \leq N$, the above argument remains valid for $(\widetilde{V}_j^*)_{1 \leq j \leq N}$.
\end{proof}

For $a_1 > 0$, $a_2, x \in \R$, let
\begin{align*}
G(a_1,a_2,x) = \Phi\bigl(\frac{x}{a_1}\bigr) + \frac{a_2}{6a_1^{3/2}}\bigl(1 - \frac{x^2}{a_1} \bigr) \phi\bigl(\frac{x}{a_1}\bigr).
\end{align*}

\begin{lem}\label{lem_psi_compare}
Suppose that $\max_{1 \leq i \leq 2}|a_i - b_i| \leq y$, $a_1 \geq c > 0$. Then
\begin{align*}
\sup_{x \in\R}(x^2+1)\big|G(a_1,a_2,x) - G(b_1,b_2,x)\big| \lesssim y
\end{align*}
for $y$ small enough.
\end{lem}

\begin{proof}[Proof of Lemma \ref{lem_psi_compare}]
A Taylor expansion yields
\begin{align*}
\sup_{x \in \R}(x^2+1)\big|G(a_1,a_2,x) - G(b_1,b_2,x)\big| \lesssim \big|a_1 - b_1\big| + \big|a_2 - b_2\big|,
\end{align*}
and hence \hypertarget{FXX:eq29}{the} claim.
\end{proof}

\begin{lem}\label{lem_edge_for_FF}
Grant Assumption \ref{ass_dependence}. Then $F$ (formally introduced above \eqref{moment_conditions_global_FF}) exists and can be chosen such that
\begin{align*}
&\sup_{x \in \R}\big|{F}(x) - {\Psi}_{m}(\sqrt{2}x)\big| = \sup_{x \in \R}\big|\P(\overline{Z}_j \leq x) - {\Psi}_{m}(\sqrt{2}x)\big| \lesssim m^{-1},
\end{align*}
where we recall that $\overline{Z}_j \stackrel{d}{=} \widetilde{Z}_j$. Moreover, we have
\begin{align*}
\sup_{x \in \R}(x^2+1)\big|\P\big(\overline{Z} + \widetilde{Z} \leq x\big) - \Psi_n(x) \big| \lesssim m^{-1}.
\end{align*}
An analogous result holds for the corresponding ${(\cdot)}^{\diamond}$ quantities.
\end{lem}

\begin{proof}[Proof of Lemma \ref{lem_edge_for_FF}]
Let $(A_k)_{1 \leq k \leq m}$ be i.i.d. \hypertarget{bigAk:eq28}{random} variables with continuous distribution function $F$ such that
\begin{align*}
&\E A_k = 0, \quad \E A_k^2 = \overline{\sigma}_j^2,\\
&\E A_k^3 = \overline{\kapp}_j^3, \quad \E|A_k|^8 < \infty.
\end{align*}
It is not hard to find such an $F$, for instance the linear combination of a normal and independent (centered) gamma random variable, see Section \ref{sec_wasserstein}. It follows that the first three moments of $m^{-\frac{1}{2}}\sum_{k = 1}^m A_k$ coincide with $0$, $\overline{\sigma}_j^2$ and $\overline{\kapp}_j^3$. In addition, we \hypertarget{OZJ:eq29}{have}
\begin{align*}
m^{-2}\big\|\sum_{k = 1}^m A_k \big\|_4^4 = \frac{3}{4} s_m^4 + \OO\big(m^{-1}\big) = \overline{\nu}_j^4 + \OO\big(m^{-1}\big)
\end{align*}
by Lemmas \ref{lem_S_n_third_expectation}, \ref{lem_quad_expansion}. Setting $\overline{Z}_j \stackrel{d}{=} m^{-\frac{1}{2}}\sum_{k = 1}^m A_k$, we conclude that \eqref{moment_conditions_global_FF} holds. Due to Lemma \ref{lem_widetilde_moments}, we may construct an analogous sequence $(B_k)_{1 \leq k \leq m}$, and set $\widetilde{Z}_j \stackrel{d}{=} m^{-\frac{1}{2}}\sum_{k = 1}^m B_k$. Since the first four moments (only three are necessary here) of $\overline{Z}_j$ (resp. $\widetilde{Z}_j$) now match those of $m^{-\frac{1}{2}}S_m$ up to an error term of $\OO(m^{-1})$ by Lemmas \ref{lem_sig_expansion}, \ref{lem_cum_expansion}, \ref{lem_quad_expansion} and \ref{lem_widetilde_moments}, the claim follows from classic Edgeworth expansions\footnote{Existence of $\E[|A_k|^8]$ is needed in Lemma \ref{lem_Z_approx_and_BE}}, e.g. ~\cite{Bhattacharya_rao_1976_reprint_2010}, and Lemma \ref{lem_psi_compare}. Likewise, the first four moments of $\overline{Z} + \widetilde{Z}$ match those of $n^{-\frac{1}{2}}S_n$ up to an error of $\OO(m^{-1})$, and the claim follows again from classic Edgeworth expansions and Lemma \ref{lem_psi_compare}.
\end{proof}

\begin{lem}\label{lem_Z_approx_and_BE}
Grant Assumption \ref{ass_dependence}. Then by construction in Lemma \ref{lem_edge_for_FF}
\begin{description}
  \item[(i)] $\sup_{x \in \R}\big|{F}_{}(x) - \Phi(x/\overline{\sigma}_{j})\big| \lesssim m^{-\frac{1}{2}}$.
  \item[(ii)] $\E|\overline{Z}_{j}|^3 \ind(|\overline{Z}_{j}| \geq \tau_n) \lesssim m^{-2}$.
\end{description}
An analogous result holds for $\widetilde{Z}_j$.
\end{lem}

\begin{proof}[Proof of Lemma \ref{lem_Z_approx_and_BE}]
{\bf (i)} follows from Lemma \ref{lem_edge_for_FF} and Lemma \ref{lem_sig_expansion}. For {\bf (ii)} we proceed as in step two in the proof of Lemma \ref{lem_taylor_expansion_smooth_function}, using the classical Fuk-Nagaev inequality.
\end{proof}

\begin{lem}\label{lem_bound_prod_cond_char}
Assume that Assumption \ref{ass_dependence} holds and recall $T_n = {\co_T} \sqrt{n}$. Then there exists an absolute constant $\co_{\varphi} > 0$ such that for sufficiently small ${\co_T} > 0$
\begin{align*}
\Bigl\|\prod_{j = N/2}^{N} \varphi_{j|m}\bigl(\xi_N \bigr)\Bigr\|_1 \lesssim e^{-\co_{\varphi} \xi^2} + e^{-\sqrt{N/32}\log 8/7}, \quad \xi^2 \leq {\co_T} n.
\end{align*}
\end{lem}

\begin{proof}[Proof of Lemma \ref{lem_bound_prod_cond_char}]
From Equations 4.14 and 4.20 in ~\cite{jirak_be_aop_2016}, there exist absolute, positive constants $\co_{\varphi,1}, \co_{\varphi,2}$ and $\co_{\varphi,3}$ such that uniformly for $\co_{\varphi,3} \leq l \leq m$
\begin{align*}
\Bigl\|\prod_{j = N/2}^{N} \varphi_{j|m}\bigl(\xi_N \bigr)\Bigr\|_1 \lesssim e^{-\frac{\co_{\varphi,1} \xi^2 (m-l)}{32m}} + e^{-\sqrt{\frac{N}{32}}\log \frac{8}{7}}, \quad \text{for $\xi^2 < \frac{\co_{\varphi,2} n}{m-l}$.}
\end{align*}
For employing this bound, we need to appropriately select $l= l(\xi)$. Choosing
\begin{align*}
l(\xi) = \ind\bigl(\xi^2 < N \co_{\varphi,2} \bigr) + \Big(m - \frac{n}{\xi^2}\vee \co_{\varphi,3}\Big)\ind\bigl(\xi^2 \geq N \co_{\varphi,2} \bigr)
\end{align*}
and ${\co_T}^2 < \co_{\varphi,2}/\co_{\varphi,3}$, it follows that
\begin{align*}
e^{-\co_{\varphi,1} \xi^2 (m-l)/32m} \lesssim e^{-\co_{\varphi} \xi^2} + e^{-\sqrt{N/32}\log 8/7}
\end{align*}
for some absolute constant $\co_{\varphi}> 0$, which completes the proof.
\end{proof}

\subsection{A generalized Nagaev-type inequalitiy}\label{sec_nagaev}

\hypertarget{xikell:eq29}{Similarly} \hypertarget{XKLP:eq30}{to} \eqref{defn_strich_depe_2}, denote \hypertarget{Xkprime:eq29}{by}
\begin{align}\label{defn_strich_depe}
\xi_{k}^{(l,')} = (\varepsilon_{k}, \varepsilon_{k - 1},\ldots,\varepsilon_{k - l}',\varepsilon_{k - l - 1},\ldots),
\end{align}
and ${X}_{k}^{(l,')} = f_m(\xi_{k}^{(l,')})$, in particular, we set ${X}_{k}' ={X}_{k}^{(k,')}$. Related to $\lambda_{k,p}$, we \hypertarget{thetaKPP:eq29}{then} define  \hypertarget{sceij:eq29}{the} coupling \hypertarget{Thetajp:eq29}{distance}
\begin{align}\label{defn_theta_coupling}
\theta_{k,p} = \|X_k-X_k'\|_p, \quad \Theta_{j,p} = \sum_{k = j}^{\infty} \theta_{k,p}.
\end{align}
  Note that $\theta_{k,p} \leq 2 \lambda_{k,p}$. For $i \geq j$, let $\mathcal{E}_{i,j} = \sigma\big(\varepsilon_i, \varepsilon_{i-1}, \ldots, \varepsilon_j \big)$, \hypertarget{bigXkm:eq29}{and}
\begin{align}\label{defn_X_km}
X_{k,m} = \E\big[X_k\big|\mathcal{E}_{k,k-m}\big], \quad k \in \Z.
\end{align}

\begin{lem}\label{th:Moritz}
Assume $\E[|X_i|^p g(X_i)|] < \infty$, where $3 \le p \le 4$, $g(x) = 1 + (\log (1+|x|))^{\ad}$, $\ad > 1$, and
\begin{align}\label{eq:23}
\sum_{j=1}^\infty j^2 \theta_{j, p} < \infty.
 \end{align}
Then for all $x > \Co (n \log n)^{\frac{1}{2}}$, where $\Co$ is a sufficiently large constant, we have
\begin{align}\label{eq:Moritz1}
    \P\big(S_n \ge x \big) = \oo\big(n x^{-p} (\log x)^{-\frac{\ad}{2}}\big).
\end{align}
\end{lem}

\begin{rem}
Since $\sum_{j=1}^\infty j^2 \theta_{j, p} \lesssim \Lambda_{2,p}$, Lemma \ref{th:Moritz} applies if Assumption \ref{ass_dependence_main} holds.
\end{rem}

\begin{proof}[Proof of Lemma \ref{th:Moritz}]
We first require some additional notation and couplings for $X_k$. In the proof we shall denote by ${\co}$ a constant \hypertarget{dotXi:eq30}{that} is independent of $n$ and $x$ and its value may change from place to place, and by $\co_q$ a constant only depending on $q$. Let $K = x (\log x)^{-\ad/q}$, $2p < q < 2(3p+1)/3$, and the truncated process $\dot{X_i} = \max(-K, \min(X_i, K))$. Then the functional dependence measure $\dot{\theta}_{k, p}$ for the process $(\dot{X_i})$ satisfies 
\begin{equation}
\dot{\theta}_{k, q} \le K^{1-p/q} \theta_{k, p}^{p/q}.
\end{equation}
By (\ref{eq:23}), \hypertarget{dotSn:eq30}{there} exists a constant ${\co} > 0$ such  \hypertarget{dotSnm:eq30}{that}
\begin{equation} \label{eq:Moritz4}
\sum_{n=k+1}^{2 k} \theta_{n, p}^{p/q} \le k^{1-p/q} (\sum_{n=k+1}^{2 k} \theta_{n, p})^{p/q}
\le k^{1-p/q} ({\co} k^{-2})^{p/q} = {\co} k^{1-3p/q}.
\end{equation}
Let $\dot{S_n} = \sum_{i=1}^n \dot{X_i}$ and $\dot{S}_{n,m} = \sum_{i=1}^n \E[\dot{X_i} | \varepsilon_{i-m}, \ldots, \varepsilon_i]$.
   We shall use the chaining argument in the proof of Theorem 2 in ~\cite{Wu_fuk_nagaev}. Let $L = \lfloor \log n / \log 2 \rfloor + 1$, $\alpha_0= 0$, $\alpha_l = 2^{l-1}$, $1 \le l \le L-1$ and $\alpha_L = n$. \hypertarget{dotMIL:eq30}{Write}
\begin{eqnarray}
\dot{S}_{i,n} - \dot{S}_{i,0} = \sum_{l=1}^L \dot{M}_{i, l}, \, \mbox{ where } \,
\dot{M}_{i, l} = \sum_{k=1}^i (\dot{X}_{k, \alpha_l} - \dot{X}_{k,\alpha_{l-1}})
\end{eqnarray}
and ${X}_{k,m}$ is defined in \eqref{defn_X_km}. \hypertarget{Thetanq:eq30}{By} (2.15) in ~\cite{Wu_fuk_nagaev} and (\ref{eq:Moritz4}),
\begin{equation}\label{eq:e54}
\P\big(\max_{i \le n} |\dot{S}_i - \dot{S}_{i, n}| \ge x\big) \le { {\|\max_{i \le n} |\dot{S}_i - \dot{S}_{i, n}|\|_q^q } \over x^q} \lesssim { {(\sqrt n \dot{\Theta}_{n, q})^q} \over x^q} \lesssim { {n^{3q/2-3p} K^{q-p} } \over x^q},
\end{equation}
   where $\dot{\Theta}_{n, q} = \sum_{k = n}^{\infty} \dot{\theta}_{k,q}$. By Burkholder's inequality, there exists a constant ${\co}_q > 0$ such that
$$
\Big\| \sum_{k=1}^i (\dot{X}_{k,a} - \dot{X}_{k,a-1}) \Big\|_q \le {\co}_q \dot{\theta}_{a, q}.
$$
\hypertarget{breMNL:eq31}{Then}
$$
{ {\| \dot{M}_{i, l} \|_q} \over \sqrt i} \le {\co}_q \sum_{a=1+\alpha_{l-1}}^{\alpha_l}
\dot{\theta}_{a, q} \stackrel{def}{=} {\co}_q \breve{\theta}_{l,q} \lesssim K^{1-p/q} \alpha_l^{1-3p/q},
$$
 \hypertarget{brevenul:eq31}{and}
$$
{{\| \dot{M}_{i, l} \|_2} \over \sqrt i}
\le \sum_{a=1+\alpha_{l-1}}^{\alpha_l} \dot{\theta}_{a, 2} \stackrel{def}{=} \breve{\theta}_{l,2}.
$$
Let $\breve M_{n, l} = \max_{i \le n} |\dot{M}_{i, l}|$ and $\breve \upsilon_1, \ldots \breve \upsilon_L$ be a positive sequence for which  $\sum_{l=1}^L \breve \upsilon_l \le 1$.
 \hypertarget{bremuL:eq31}{Then}
\begin{eqnarray}\label{eq:e57}
\P\big(\breve M_{n, l} \ge 3 \breve \upsilon_l x\big) \le {\co}_q {n \over {x^q}}
{{\alpha_l^{q/2-1} \breve{\theta}_{l,q}^q} \over {\breve \upsilon_l^q}}
 + 2 \exp\Big(-{\co}_q { {(\breve \upsilon_l x)^2}
 \over { n \breve{\theta}_{l, 2}^2}} \Big).
\end{eqnarray}
Let $\breve \mu_l = (\alpha_l^{q/2-1} \breve{\theta}_{l,q}^q)^{1/(q+1)}$, $\overline{\mu}_L = \sum_{l=1}^L \breve \mu_l$ and $\breve \upsilon_l = \breve \mu_l / \overline{\mu}_L$. Since $2p < q < 2(3p+1)/3$,
\begin{eqnarray*}
 \sum_{l=1}^L { {\alpha_l^{q/2-1} \breve{\theta}_{l,q}^q} \over {\breve \upsilon}_l^q}
   = \overline{\mu}_L^{q+1}
   \le  \Big(\sum_{l=1}^L (\alpha_l^{q/2-1} {\co}^q K^{q-p} \alpha_l^{q-3p})^{1/(q+1)} \Big)^{q+1} \lesssim K^{q-p}.
\end{eqnarray*}
By (\ref{eq:e54}) and (\ref{eq:e57}) and the classical Nagaev inequality for independent random variables
\begin{align}\label{eq:S30847} \nonumber
&\P\big( \max_{i\le n} |\dot{S}_i-\E[\dot{S}_i]| \ge 5 x\big)\leq
\sum_{l=1}^L \P\big(\breve M_{n, l} \ge 3 \breve{\upsilon}_l x\big)\\ \nonumber
&+\P\Big( \max_{1\le i\le n} |\dot{S}_i - \dot{S}_{i,n}|  \geq x\Big) +\P\Big( \max_{1\le i\le n} |\dot{S}_{i,0} - \E \dot{S}_{i,0} | \ge x\Big)\\ \nonumber
&\lesssim {n \over {x^q}} K^{q-p} + \sum_{l=1}^L 2 \exp\Big(-{\co}_q { {(\breve \upsilon_l x)^2}
\over{ n \breve{\theta}_{l, 2}^2}} \Big) \\  & + { {n^{3q/2-3p} K^{q-p} } \over x^q}
 + {\co}_q { {n \|\dot{X}_0\|_q^q} \over {x^q}} + 2 \exp\Big( - { {{\co}_q x^2} \over {n \|\dot{X}_0\|^2} } \Big).
\end{align}
Note that ${\|\dot{X}_0\|_q^q} \le K^{q-p}  \|X_0\|_p^p$. Since $\breve \upsilon_l / \breve{\theta}_{l, 2} \ge  \breve \mu_l / (\overline{\mu}_L \breve{\theta}_{l, q}) \ge {\co} (\alpha_l^{q/2-1})^{1/(q+1)}$ for some constant ${\co} > 0$ and $x > \Co (n \log n)^{1/2}$ for sufficiently large $\Co$, we have by (\ref{eq:S30847}) that
\begin{eqnarray}
 \label{eq:J26855}
\P\big( \max_{i\le n} |\dot{S}_i-\E \dot{S}_i| \ge 5 x\big)
  \le {\co} {n \over {x^q}} K^{q-p}.
\end{eqnarray}
Let $h(x) = |x|^p g(x)$. Since $|\E \dot{S}_n | \le n K^{1-p} \E |X_i|^p = \oo(x)$, we have
\begin{align*}
\P\big(|S_n| > 6 x\big) &\le n \P\big(|X_1| \ge K\big) + \P\big(|\dot{S}_n | \ge 6 x\big) \cr
 & \le n \P\big(|X_1| > K\big) + \P\big( \max_{i\le n} |\dot{S}_i - \E \dot{S}_i | \ge 5 x\big) \cr
 &\le  n \E h(X_i) / h(K) + {\co} { n \over x^q} K^{q-p}.
 \end{align*}
Since $2p < q < 2(3p+1)/3$, by elementary manipulations, the claim follows.
\end{proof}

\subsection{Technical Auxiliary Results}


\begin{lem}\label{lem_sig_expressions_relations}
Grant Assumption \ref{ass_dependence}. Then
\begin{description}
\item[(i)] $\sum_{k = 1}^{\infty}k \bigl|\E X_0 X_k \bigr| < \infty$,
\item[(ii)] $\E S_n^2 = n \ss_m^2 - \sum_{k \in \Z} (n \wedge |k|)\E X_0 X_k$.
\end{description}
\end{lem}

\begin{proof}[Proof of Lemma \ref{lem_sig_expressions_relations}]
{\bf (i)} follows from $|\E X_0 X_k | \leq \|X_k -X_k^*\|_2 \|X_0\|_2$. {\bf (ii)} follows from {\bf (i)} and routine calculations, we omit the details.
\end{proof}

We frequently use the following lemma, which is essentially a restatement of Theorem 1 in ~\cite{sipwu}, adapted \hypertarget{thetaKPP:eq32}{to} our setting.
\begin{lem}\label{lem_wu_original}
Grant Assumption \ref{ass_dependence}, and recall $\theta_{k,p}$ defined in \eqref{defn_theta_coupling}, and put $p' = p \wedge 2$. If $\sum_{k = 1}^{\infty} \theta_{k,p} < \infty$, then
\begin{align*}
\bigl\|X_1 + \ldots + X_n\bigr\|_{p} \lesssim n^{\frac{1}{p'}}.
\end{align*}
\end{lem}

Recall that $\sqrt{2m} \overline{V}_j = U_j + R_j$, where $U_j,R_j$ are defined in \eqref{defn_U_R}, \hypertarget{BVJS:eq32}{and}
\begin{align}\label{defn_Vj_star_2}
\overline{V}_j^* = \frac{1}{\sqrt{2m}}\sum_{k = (2j-2)m + 1}^{(2j - 1)m} X_k^{(k - (2j-2)m,*)}, \quad  \text{for $1 \leq j \leq N$}.
\end{align}

\begin{lem}\label{lem_bound_R1}
Grant Assumption \ref{ass_dependence}. Then for $j = 1,\ldots,N$
\begin{description}
\item[(i)]  $\bigl\|R_j\bigr\|_p \lesssim m^{-\frac{1}{2}}$, $\bigl\|\widetilde{V}_0\bigr\|_p \lesssim m^{-1/2}$,
\item[(ii)] $\big\|\overline{V}_j - \overline{V}_j^*\big\|_p \lesssim m^{-\frac{1}{2}}$,
\item[(iii)] $\bigl\|\overline{V}_j\bigr\|_p, \bigl\|\overline{V}_j^*\bigr\|_p  < \infty$,
\item[(iv)] $\bigl\|\widetilde{V}_j\bigr\|_p, \bigl\|\widetilde{V}_j^*\bigr\|_p  < \infty$.
\end{description}
\end{lem}

\begin{proof}[Proof of Lemma \ref{lem_bound_R1}]
Without loss of generality, we can assume that $j = 1$ throughout the proof.\\
{\bf (i)}: Since for $m+1 \le k \leq 2m$
\begin{align*}
X_k - \E_{\FF_m}\bigl[X_k\bigr] \stackrel{d}{=} \E_{\sigma(\FF_m, \mathcal{E}_k^*)}\bigl[X_k^{(k-m,*)} - X_k\bigr],
\end{align*}
we have
\begin{align*}
&\Bigl\|\sum_{k = m+1}^{2m}\bigl(X_k - \E_{\FF_m} X_k \bigr)\Bigr\|_p = \Bigl\|\sum_{k = m+1}^{2m} \E_{\sigma(\FF_m, \mathcal{E}_k^*)}\bigl[X_k^{(k-m,*)} - X_k\bigr]\Bigr\|_p \\&\leq \sum_{k = m+1}^{2m}\bigl\| X_k^{(k-m,*)} - X_k\bigr\|_p \leq \sum_{k = 1}^{\infty} \lambda_{k,p} < \infty.
\end{align*}
Hence $\bigl\|R_j\bigr\|_p \lesssim m^{-\frac{1}{2}}$. Similarly, $\bigl\|\widetilde{V}_0\bigr\|_p \lesssim m^{-1/2}$ follows.\\
{\bf (ii)}: Observe that by the triangle inequality and {\bf (i)} (recall $j = 1$),
\begin{align*}
\sqrt{2 m}\bigl\|\overline{V}_{j}^{*} - V_j\bigr\|_p &\leq \sum_{k = 1}^m \bigl\|X_k - X_k^{*}\bigr\|_p + \sum_{k = 1}^m \bigl\|\E_{\FF_m} X_k \bigr\|_p + \bigl\|R_j\bigr\|_p \\&\leq \sum_{k = 1}^m \bigl\|\E_{\FF_m} X_k \bigr\|_p + \OO\big(1\big).
\end{align*}
We note $\E\bigl[X_k^{*}\bigl|\FF_m\bigr] = \E X_k= 0$ for $1 \leq k \leq m$, and thus by Jensen's and the triangle inequality
\begin{align*}
\Bigl\|\sum_{k = 1}^m\E\bigl[X_k\bigl|\FF_m\bigr]\Bigr\|_p = \Bigl\|\sum_{k = 1}^m\E\bigl[X_k - X_k^{*} \bigl|\FF_m\bigr]\Bigr\|_p \leq \sum_{k = 1}^m\bigl\|X_k - X_k^{*} \bigr\|_p \leq \sum_{k = 1}^{\infty} \lambda_{k,p} < \infty.
\end{align*}
Piecing everything together, we have {\bf (ii)}.\\
{\bf (iii)}: By {\bf (ii)}, it suffices to show $\big\|\overline{V}_{j}^{*}\big\|_p < \infty$. However, this immediately follows from Lemma \ref{lem_wu_original}.\\
{\bf (iv)}: This follows by using the same arguments as in {\bf (i)}-{\bf (iii)}.
\end{proof}

\section{Proof of Proposition \ref{prop_m_dependence}}\label{sec_proof_prop_m_dep}

{
For the proof, key estimates are provided by some lemmas. More precisely, we use
Lemma \ref{lem_taylor_expansion_smooth_function} (combined with Lemma \ref{lem_taylor_expansion_smooth_function_I_BE}) to bound the
difference between conditional and unconditional characteristic functions. We also use Lemma \ref{lem_bound_prod_cond_char} to show that conditional characteristic functions decay sufficiently fast.

\begin{proof}[Proof of Proposition \ref{prop_m_dependence}]
By properties of conditional expectations, independence, and $|e^{\ic x}| = 1$, we have
\begin{align*}
&\Bigl|\E e^{\ic \xi S_n/\sqrt{n}}  - \E e^{\ic \xi (\overline{Z} + \widetilde{Z})}\Bigr| \leq \Bigl\|\E_{\FF_m} e^{\ic \xi \overline{S}_{n|m}/\sqrt{n}} - \E_{} e^{\ic \xi \overline{Z}} \Bigr\|_1 \\& + \Bigl|\E_{}e^{\ic \xi \widetilde{S}_{n|m}/\sqrt{n}} - \E_{}e^{\ic \xi \widetilde{Z}}\Bigr| \stackrel{def}{=} {\bf A}_m(\xi) + {\bf B}_m(\xi).
\end{align*}
In the sequel, we treat these \hypertarget{phiJCM:eq33}{two} terms separately.\\

{\bf ${\bf A}_m(\xi)$:} Let $\varphi_{j|m}(x) = \E\bigl[e^{\ic x \overline{V}_j}\bigl| \FF_m\bigr]$, $\psi_{j}(x) = \E e^{\ic x \overline{Z}_j}$ and $\xi_N = \xi/\sqrt{N}$. Due to the independence of $(\overline{V}_j)_{1 \leq j \leq N}$ under $\P_{|\FF_m}$, it follows that
\begin{align}\label{eq_berry_smoothing}
{\bf A}_m(\xi) = \Bigl\|\prod_{j = 1}^N \varphi_{j|m}\bigl(\xi_N \bigr) - \prod_{j = 1}^N \psi_{j}\bigl(\xi_N\bigr)\Bigr\|_1.
\end{align}

For $a_j,b_j \in \mathds{C}$ we have
\begin{align}\label{eq_prod_formula}
\prod_{j = 1}^N a_j - \prod_{j = 1}^N b_j = \sum_{i = 1}^N \Bigl(\prod_{j = 1}^{i-1}b_j\Bigr)\big(a_i - b_i\big)\Bigl(\prod_{j = i+1}^N a_j\Bigr),
\end{align}
where we use the convention that $\prod_{j = 1}^{i-2}(\cdot) = \prod_{j = i + 2}^{N}(\cdot) = 1$ if $i-2 < 1$ or $i + 2 > N$.

Note that $(\varphi_{j|m})_{1 \leq j \leq N}$ is a one-dependent sequence. Since $|\varphi_{j|m}|\leq 1$ and $|\psi_{j}|\leq 1$, it then follows from (\ref{eq_prod_formula}), the triangle inequality, 'leave one out' (to obtain independence) and stationarity that
\begin{align}\label{eq_thm_m_dependence_1} \nonumber
{\bf A}_m(\xi) &\leq \sum_{i = 1}^N\Bigl\|\prod_{j = 1}^{i-2}\psi_{j}\bigl(\xi_N \bigr)\Bigr\|_1 \Bigl\|\varphi_{i|m}\bigl(\xi_N \bigr) -\psi_{i}\bigl(\xi_N \bigr)\Bigr\|_1\Bigl\|\prod_{j = i+2}^N \varphi_{j|m}\bigl(\xi_N \bigr)\Bigr\|_1 \\& \nonumber \leq N\Bigl\|\varphi_{1|m}\bigl(\xi_N \bigr) -\psi_{1}\bigl(\xi_N \bigr)\Bigr\|_1\Bigl\|\prod_{j = N/2}^{N} \varphi_{j|m}\bigl(\xi_N \bigr)\Bigr\|_1\\& \nonumber + N \Bigl\|\prod_{j = 1}^{N/2-3}\psi_{j}\bigl(\xi_N\bigr)\Bigr\|_1  \Bigl\|\varphi_{1|m}\bigl(\xi_N \bigr) -\psi_{1}\bigl(\xi_N \bigr)\Bigr\|_1 \\&\stackrel{def}{=} I_N(\xi) + II_N(\xi).
\end{align}
Let us first consider $I_N(\xi)$. Since $e^{\ic x} = \cos(x) + \ic \sin (x)$ for $x \in \R$, an application of Lemma \ref{lem_taylor_expansion_smooth_function} {\bf (i)}, in view of Lemma \ref{lem_taylor_expansion_smooth_function_I_BE}, yields (with $p = 3$)
\begin{align}\label{eq_thm_m_dependence_2}\nonumber
&\bigl\|\varphi_{j|m}\bigl(\xi_N \bigr) -\psi_{j}\bigl(\xi_N \bigr)\bigr\|_1 \lesssim \big(|\xi_N|^2+ |\xi_N|^3 + \tau_n^4(|\xi_N|^4 + |\xi_N|^7)\big)m^{-1}\\& \nonumber + |\xi_N|^3 \tau_n^{3-p} \log(n)^{-\bd} \oo\big(m^{1-\frac{p}{2}}\big) + m^{-\frac{1}{2}}\bigl(\tau_n^4|\xi_N|^4 + |\tau_n|^5 |\xi_N|^5\big)\\& \stackrel{def}{=} III_N(\xi),
\end{align}
uniformly for $1 \leq j \leq N$. Here, $\tau_n \geq {\co_{\tau}} \sqrt{\log n}$, ${\co_{\tau}} > 0$ sufficiently large. Further, Lemma \ref{lem_bound_prod_cond_char} implies
\begin{align*}
\Bigl\|\prod_{j = N/2}^{N} \varphi_{j|m}\bigl(\xi_N \bigr)\Bigr\|_1 \lesssim e^{-\co_{\varphi} \xi^2} + e^{-\log \frac{8}{7} \sqrt{N/32}}, \quad \xi^2 \leq \co_T^2 n,
\end{align*}
for ${\co_T} > 0$ sufficiently small. Combining both estimates and plugging in $\xi_N = \xi/\sqrt{N}$, we obtain the bound
\begin{align*}
I_N(\xi) &\lesssim III_N(\xi) \Big(e^{-\co_{\varphi} \xi^2} + e^{-\log \frac{8}{7} \sqrt{N/32}}\Big)
\\& \lesssim \Big(|\xi|^3 N^{-\frac{1}{2}} \oo\big(m^{1-\frac{p}{2}}\big) \log(\tau_n)^{-\bd} + (|\xi|^4\vee|\xi|^5) \tau_n^5 (N n)^{-\frac{1}{2}} \\& + (|\xi|^4 \vee |\xi|^7) \tau_n^4 n^{-1} + |\xi|^3 m^{-1}n^{-\frac{1}{2}} + |\xi|^2 m^{-1} \Big) \Bigl(e^{-\co_{\varphi} \xi^2} + e^{-\log \frac{8}{7} \sqrt{N/32}}\Bigr).
\end{align*}
Since $N \thicksim n^{1 - \md}$, we conclude that for $\md < 1$ close enough to one
\begin{align}\label{eq_thm_m_dependence_3}
\int_{-T_n}^{T_n} \frac{I_N(\xi)}{|\xi|} \,d\, \xi \lesssim  \oo\big(n^{-\frac{1}{2}}\big) \log(\tau_n)^{-\bd} + \tau_n^5 (N n)^{-\frac{1}{2}}.
\end{align}
Using a standard argument to bound $|\psi_j(\xi_N)|$ (cf. ~\cite{fellervolume2}, XVI.5), we obtain the same bound for $II_N(\xi)$ in an analogous manner. Piecing everything together yields
\begin{align}\label{eq_thm_m_dependence_4}
\int_{-T_n}^{T_n} \frac{{\bf A}_m(\xi)}{|\xi|} \, d\, \xi \lesssim  \oo\big(n^{-\frac{1}{2}}\big) \log(\tau_n)^{-\bd} + \tau_n^5 (N n)^{-\frac{1}{2}}.
\end{align}

{\bf ${\bf B}_m(\xi)$:}  We may proceed in the same way as for ${\bf A}_m(\xi)$. Note that the present situation is much simpler since $(\widetilde{V}_j)_{1 \leq j \leq N}$ are i.i.d., and not only conditionally independent as $(\overline{V}_j)_{1 \leq j \leq N}$. Using Lemma \ref{lem_taylor_expansion_smooth_function_II} this leads to the same bound
\begin{align}\label{eq_thm_m_dependence_12}
\int_{-T_n}^{T_n} \frac{{\bf B}_m(\xi)}{|\xi|} \, d\, \xi \lesssim  \oo\big(n^{-\frac{1}{2}}\big) \log(\tau_n)^{-\bd} + \tau_n^5 (N n)^{-\frac{1}{2}}.
\end{align}
Combining (\ref{eq_thm_m_dependence_4}) and (\ref{eq_thm_m_dependence_12}) we obtain
\begin{align*}
 \int_{-T_n}^{T_n} \big({\bf A}_m(\xi) + {\bf B}_m(\xi)\bigr) \frac{1}{|\xi|} \, d\, \xi  \lesssim \oo\big(n^{-\frac{1}{2}}\big) \log(\tau_n)^{-\bd} + \tau_n^5 (N n)^{-\frac{1}{2}}.
\end{align*}
Recall $m = n^{\md}$, $N \thicksim n^{1 - \md}$. Setting $\tau_n \thicksim n^{\tau}$ with $\tau, 1 - \md >0$ sufficiently small, we arrive at
\begin{align*}
\Ad_{T_n}  = \oo\big(n^{-\frac{1}{2}}\big)\log(n)^{-\bd}.
\end{align*}
\end{proof}

\section{Proof of Proposition \ref{prop_m_dep_it}}\label{sec_proof_dep_it}

The basic idea for the proof is to set up a recursion, subsequently improving the bound. The problem here is the characteristic $\Cd_{T_n}$, which we cannot control. To bypass $\Cd_{T_n}$, we use an additional smoothing argument that allows us to set up the recursion for the smoothed distance $\Delta_n^{\diamond}$ (defined below), where $\Cd_{T_n}$ vanishes \hypertarget{Gab:eq35}{and} thus becomes insignificant. To this end, for $a > 0$ and $b \in \N$ even, let $G_{a,b}$  be a real valued random variable with density function
\begin{align}\label{eq_g_ab}
g_{a,b}(x) = \co_b a \Big|\frac{\sin(a x)}{a x} \Big|^b, \quad x \in \R,
\end{align}
for some constant $\co_b > 0$ only depending on $b$. It is well-known (cf. ~\cite{Bhattacharya_rao_1976_reprint_2010}, Section 10) that for even $b$ the Fourier transform $\hat{g}_{a,b}$  \hypertarget{ghatab:eq35}{satisfies}
\begin{align}\label{eq_thm_smooth_fourier}
\hat{g}_{a,b}(t) =  \left\{
\begin{array}{ll}
2 \pi \co_b u^{\ast \, b}[-a,a](t) &\text{if $|t| \leq a b$},\\
0 &\text{otherwise},
\end{array}
\right.
\end{align}
where $u^{\ast \, b}[-a,a]$ denotes the $b$-fold convolution of \hypertarget{XXKdm:eq35}{the} density of the uniform distribution on $[-a,a]$, that is $u[-a,a](t) = \frac{1}{2a} \ind_{[-a,a]}(t)$. For $b \geq 6$, let $(H_k)_{k \in \Z}$ be i.i.d. with $H_k \stackrel{d}{=} G_{a,b}$ and independent of $S_n$. \hypertarget{SNDM:eq35}{Define}
\begin{align}\label{defn_diamond_mod}
{X}_k^{\diamond} = X_k + H_{k} - H_{k-1}, \quad {S}_n^{\diamond} = \sum_{k = 1}^n {X}_k^{\diamond} = S_n + H_n - H_0,
\end{align}
and in analogy ${s}_n^{\diamond}$, ${\kapp}_n^{\diamond}$ and $\Ad_{T_n}^{\diamond}$, $\Id_{T_n}^{\diamond}$, $(\Td_{x}^y)^{\diamond}$, ${\Psi}_n^{\diamond}$, $\overline{Z}_j^{\diamond}$, $\widetilde{Z}_j^{\diamond}$, $\overline{Z}^{\diamond}$, $\widetilde{Z}^{\diamond}$ and \hypertarget{deltandia:eq36}{the} difference
\begin{align}
{\Delta}_n^{\diamond}(x) = \P\big({S}_n^{\diamond} \leq x \sqrt{n} \big) - {\Psi}_n^{\diamond}(x), \quad x \in \R.
\end{align}
Note that since $b \geq 6$, exploiting also the independence of $(H_k)_{k \in \Z}$ and $S_n$, we have by \eqref{eq_g_ab} and \eqref{eq_thm_smooth_fourier}
\begin{align}\nonumber
&\E H_k = 0, \quad \E|H_k^{}|^4 < \infty,\\ \label{eq_char_S_n_diamond_zero}
&\Big|\E e^{\ic \xi S_n^{\diamond}/\sqrt{n}} \Big| = 0 \quad \text{for $|\xi| > \sqrt{n} |ab|$.}
\end{align}

Next, recall the definitions of $\overline{V}_j^*$ and $\widetilde{V}_j^*$ in \eqref{defn_Vj_star} and \eqref{defn_Vj_star_II}.
For $1 \leq j \leq N$, denote by
\begin{align}\label{defn_Vj_diamond}
&\sqrt{2m}\,\overline{V}_j^{\diamond} = \sqrt{2m}\,\overline{V}_j^* + H_{(2j - 1)m +1} - H_{(2j-2)m + 1}.
\end{align}
Observe that $\overline{V}_j^{\diamond}$ is independent of $\FF_m$ for all $1 \leq j \leq N$ by construction. Similarly, for $1 \leq j \leq N$ ($\widetilde{V}_0$ is degenerate) we put
\begin{align}\label{defn_Vj_diamond_II}
&\sqrt{2m}\widetilde{V}_j^{\diamond} = \sqrt{2m}\widetilde{V}_j^* + H_{2jm +1} - H_{(2j-1)m + 1}.
\end{align}
Note that $S_m^{\diamond}/\sqrt{2m} \stackrel{d}{=}\overline{V}_j^{\diamond} \stackrel{d}{=} \widetilde{V}_j^{\diamond}$. In analogy to \eqref{defn_Delta_star_diamond}, we
 \hypertarget{deltaDM:eq36}{put}
\begin{align}\label{defn_Delta_diamond}
{\Delta}_{j,m}^{\diamond}(x) = \overline{\Delta}_{j,m}^{\diamond}(x) &= \P_{}\bigl(\overline{V}_j^{\diamond} \geq x \bigr) - \P_{}\bigl(\overline{Z}_j^{\diamond} \geq x \bigr), \quad 1 \leq j \leq N,
\end{align}
$\widetilde{\Delta}_{j,m}^{\diamond} = {\Delta}_{j,m}^{\diamond}$ for $1 \leq j \leq N$. Observe that ${\Delta}_{j,m}^{\diamond}$ actually do not depend on $j$, but we stick to this notation to distinguish them from $\Delta_m^{\diamond}$.\\
\\
We now establish modifications of Lemmas \ref{lem_taylor_expansion_smooth_function} and \ref{lem_taylor_expansion_smooth_function_II}. The proofs are similar, so we only sketch them.

\begin{lem}\label{lem_taylor_expansion_smooth_function_diamond}
Grant Assumption \ref{ass_dependence}, and let $f$ be a smooth function such that $\sup_{x \in \R}|f^{(s)}(x)| \leq 1$ for $s = 0,\ldots,7$. Then for $\tau_n \geq {\co_{\tau}} \sqrt{\log n}$, ${\co_{\tau}} > 0$ sufficiently large, $2\bd + 2 < \ad$ and $p \geq 3$
\begin{align*}
&\text{{\bf (i)}} \,\, \max_{1 \leq j \leq N}\bigl\|\E_{\FF_m} \ff\bigl(\xi_N \overline{V}_{j}\bigr)  - \E_{}\ff\bigl(\xi_N \overline{Z}_{j}\bigr)\bigr\|_1\\& \quad \lesssim |\xi_N|^3 \tau_n^{3-p} \log(n)^{-\bd} \oo\big(m^{1-\frac{p}{2}}\big) + \tau_n^4(|\xi_N|^4 + |\xi_N|^7)m^{-1} \\&\quad +\sup_{x \in \R}\big|{\Delta}_{j,m}^{\diamond}(x)\big|\bigl(\tau_n^4|\xi_N|^4 + |\tau_n|^5 |\xi_N|^5\big) + |\xi_N|^3 m^{-1}+ |\xi_N|^2 m^{-1}.\\
&\text{{\bf(ii)}} \,\, \text{The above also gives an upper bound for}  \\& \quad \max_{1 \leq j \leq N}\bigl\|\E_{\FF_m} \ff\bigl(\xi_N \overline{V}_{j}^{\diamond}\bigr) - \E_{}\ff\bigl(\xi_N \overline{Z}_{j}^{\diamond}\bigr)\bigr\|_1.
\end{align*}
\end{lem}

\begin{proof}[Proof of Lemma \ref{lem_taylor_expansion_smooth_function_diamond}]

For showing {\bf (ii)}, we may directly apply Lemma \ref{lem_taylor_expansion_smooth_function}, setting $X_k = X_k^{\diamond}$. Next, we show {\bf (i)}. Using \eqref{eq_char_S_n_diamond_zero}, a Taylor expansion and the triangle inequality, we obtain
\begin{align*}
\bigl\|\E_{\FF_m}\ff\bigl(\xi_N \overline{V}_{j}\bigr) - \E_{}\ff\bigl(\xi_N \overline{V}_{j}^{\diamond}\bigr)\bigr\|_1 \lesssim \xi_N^2 m^{-1} \E H_0^2 \lesssim \xi_N^2 m^{-1}.
\end{align*}
The same applies to $\overline{Z}_j^{\diamond}$. The claim now follows from {\bf (ii)}. 
\end{proof}

\begin{lem}\label{lem_taylor_expansion_smooth_function_diamond_II}
Assume that Assumption \ref{ass_dependence} holds, and let $f$ be a smooth function such that $\sup_{x \in \R}|f^{(s)}(x)| \leq 1$ for $s = 0,\ldots,7$. Then for $\tau_n \geq {\co_{\tau}} \sqrt{\log n}$, ${\co_{\tau}} > 0$ sufficiently large, $2\bd +2 < \ad$ and $p \geq 3$
\begin{align*}
&\text{{\bf (i)}} \,\, \max_{1 \leq j \leq N}\bigl\|\E_{}\ff\bigl(\xi_N \widetilde{V}_{j}\bigr)  - \E_{}\ff\bigl(\xi_N \widetilde{Z}_{j}\bigr)\bigr\|_1 \\&\quad \lesssim |\xi_N|^3 \tau_n^{3-p} \log(n)^{-\bd} \oo\big(m^{1-\frac{p}{2}}\big) + \tau_n^4(|\xi_N|^4 + |\xi_N|^7)m^{-1}\\& \quad +\sup_{x \in \R}\big|{\Delta}_{j,m}^{\diamond}(x)\big|\bigl(\tau_n^4|\xi_N|^4 + |\tau_n|^5 |\xi_N|^5\big) + \big(\xi_N^2 + |\xi_N|^3\big) m^{-1}.\\
&\text{{\bf (ii)}} \,\, \text{The above estimate also bounds}  \\&\quad \max_{1 \leq j \leq N}\bigl\|\E_{} \ff\bigl(\xi_N \widetilde{V}_{j}^{\diamond}\bigr) - \E_{} \ff\bigl(\xi_N \widetilde{Z}_{j}^{\diamond}\bigr) \bigr\|_1.\\
&\text{{\bf (iii)}} \,\, \bigl\|\E_{} \ff\bigl(\xi_N \widetilde{V}_0^{\diamond}\bigr) - \E_{} \ff\bigl(0\bigr) \bigr\|_1 \lesssim \xi_N^2 m^{-1}.
\end{align*}
\end{lem}

\begin{proof}[Proof of Lemma \ref{lem_taylor_expansion_smooth_function_diamond_II}]
For {\bf (i)} and {\bf (ii)} we may argue almost in the same manner as in the proof of Lemma \ref{lem_taylor_expansion_smooth_function_diamond} (resp. Lemma \ref{lem_taylor_expansion_smooth_function}). 
Similarly, for {\bf (iii)}, it suffices to note that $\|\widetilde{V}_{0}^{\diamond}\|_p \lesssim m^{-\frac{1}{2}}$, which follows from Lemma \ref{lem_bound_R1} {\bf (i)} and the triangle inequality. Hence
$\widetilde{\sigma}_1^2 \lesssim m^{-1}$, 
and the claim follows from a second order Taylor expansion at $x = 0$, using the independence of $(H_k)_{k \in \Z}$ and $\widetilde{V}_0$.
\end{proof}

Having established the preliminary modifications, we are \hypertarget{deltandia:eq37}{now} ready to establish the desired recursion for $\Delta_n^{\diamond}$.
\begin{lem}\label{lem_diamond_Sn_expansions_I}
Grant Assumption \ref{ass_dependence} for $p \in [3,4)$. Then for $N \gtrsim (\log n)^{\md}$, $\md > 2$, $2\bd + 2 < \ad$,
\begin{align*}
\sup_{x \in \R}\big|{\Delta}_n^{\diamond}(x)\big| \lesssim \oo\big(n^{1-\frac{p}{2}}\big) N^{\frac{p-3}{2}} (\log n)^{-\bd} + (\log n)^{2} N^{-1}\sup_{x \in \R}\big|{\Delta}_m^{\diamond}(x)\big|.
\end{align*}
\end{lem}

\begin{proof}[Proof of Lemma \ref{lem_diamond_Sn_expansions_I}]
Routine computations reveal
\begin{align*}
\E S_n^2 = \E(S_n^{\diamond})^2 + \OO\big(1\big),
\end{align*}
and hence $s_n^{\diamond} \thicksim \ss_m$ by Lemma \ref{lem_sig_expressions_relations}. Consequently, the variance $(s_n^{\diamond})^2$ is uniformly bounded away from zero due to \hyperref[B3]{\Bthree} for $n \geq n_0$, $n_0$ large enough. We may thus apply Berry's smoothing inequality (cf. ~\cite[Lemma 2, XVI.3]{fellervolume2}). For $T_n^+ \geq T_n$ we then obtain
\begin{align*}
\sup_{x \in \R}\big|{\Delta}_n^{\diamond}(x)\big| \lesssim \Ad_{T_n}^{\diamond} + \sup_{x \in \R}(\Td_{T_n}^{T_n^+})^{\diamond}(x) + (T_n^+)^{-1}.
\end{align*}
However, selecting $a > 0$ such that $\co_T > a b$, we get from \eqref{eq_char_S_n_diamond_zero}
\begin{align*}
\sup_{x \in \R}\bigl|\Td_{T_n}^{T_n^+}(x)\bigr| \leq \int_{\co_T \sqrt{n} \leq |\xi| < T_n^+} \big|\E e^{\ic \xi H_0/\sqrt{n}}\big]\big|^2 \frac{1}{|\xi|} d \xi = 0.
\end{align*}
Setting $T_n^+ = \infty$ and $\tau_n = \co_{\tau} \sqrt{\log n}$, ($\co _{\tau} >0$ large enough), we may repeat the proof of Proposition \ref{prop_m_dependence}. The key difference is that we bound the difference of the (conditional) characteristic functions with the help of $\Delta_m^{\diamond}(x)$, setting up a recursion. More precisely, we use Lemma \ref{lem_taylor_expansion_smooth_function_diamond} {\bf (ii)} instead of Lemma \ref{lem_taylor_expansion_smooth_function} and Lemma \ref{lem_taylor_expansion_smooth_function_diamond_II} {\bf (ii)}, {\bf (iii)} instead of Lemma \ref{lem_taylor_expansion_smooth_function_II} {\bf (i)}, {\bf (ii)}. This gives the estimate
\begin{align*}
\sup_{x \in \R}\big|{\Delta}_n^{\diamond}(x)\big| &\lesssim \oo\big(n^{1-\frac{p}{2}}\big) N^{\frac{p-3}{2}} (\log n)^{-\bd}+(\log n)^{2}N^{-2}\sum_{j = 1}^N \sup_{x \in \R}\big|{\Delta}_{j,m}^{\diamond}(x)\big|,
\end{align*}
where ${\Delta}_{j,m}^{\diamond}(x)$ is defined in \eqref{defn_Delta_star_diamond}. However, $\sup_{x \in \R}|{\Delta}_{j,m}^{\diamond}(x)|$ is very close to $\sup_{x \in \R}|{\Delta}_{m}^{\diamond}(x)|$. Indeed, using Lemma \ref{lem_edge_for_FF} (an Edgeworth expansion for $\overline{Z}_j^{\diamond}$ and $\widetilde{Z}_j^{\diamond}$), it follows that
\begin{align*}
&\sup_{x \in \R}|{\Delta}_{j,m}^{\diamond}(x)| \lesssim m^{-1} + \sup_{x \in \R}|{\Delta}_{m}^{\diamond}(x)|, \quad 1 \leq j \leq N.
\end{align*}
Hence
\begin{align*}
(\log n)^{2}N^{-2}\sum_{j = 1}^N \sup_{x \in \R}\big|{\Delta}_{j,m}^{\diamond}(x)\big| \lesssim (\log n)^{2}n^{-1} + (\log n)^{2} N^{-2} \sup_{x \in \R}|{\Delta}_{m}^{\diamond}(x)|,
\end{align*}
and the claim follows.
\end{proof}

\begin{lem}\label{lem_diamond_Sn_expansions_II}
Grant Assumption \ref{ass_dependence} for $p \in [3,4)$. Then for $\delta > 0$ arbitrarily small
\begin{align*}
\sup_{x \in \R}\big|{\Delta}_n^{\diamond}(x)\big| \lesssim n^{1-\frac{p}{2} + \delta}.
\end{align*}
\end{lem}

\begin{proof}[Proof of Lemma \ref{lem_diamond_Sn_expansions_II}]
Let $N \thicksim n^{1-\md}$, $\md < 1$ and $\co_{\md} = \lfloor (1-\md)^{-1} \rfloor$. Since $p < 4$, iterating Lemma \ref{lem_diamond_Sn_expansions_I}  \hypertarget{deltandia:eq38}{yields}
\begin{align*}
\sup_{x \in \R}\big|{\Delta}_n^{\diamond}(x)\big| &\lesssim \Big(\oo\big(n^{1-\frac{p}{2}}\big) N^{\frac{p-3}{2}} \sum_{k = 0}^{\co_{\md} + 1} \big((2 \co^{-1})^3 N^{\frac{p}{2}-2} (\log n)^2\big)^{k} \Big) + \big(N^{-1} (\log n)^2\big)^{\co_{\md}}\\&\lesssim n^{1 -\frac{p}{2}} N^{1/2} + \big(N^{-1} (\log n)^2\big)^{\co_{\md}} \lesssim n^{1-\frac{p}{2} + \delta},
\end{align*}
selecting $\md$ close enough to one.
\end{proof}

\begin{proof}[Proof of Proposition \ref{prop_m_dep_it}]
We repeat the proof of Proposition \ref{prop_m_dependence} verbatim with $\tau_n = \co_{\tau} \sqrt{\log n}$ ($\co_{\tau} >0$ large enough) and $N \thicksim n^{1-\md}$, $\md < 1$. As in the proof of Lemma \ref{lem_diamond_Sn_expansions_I}, the key difference is that we express bounds for (conditional) characteristic functions with the help of $\Delta_m^{\diamond}(x)$. We may then apply Lemma \ref{lem_diamond_Sn_expansions_II} to conclude the result. In more detail, compared to Proposition \ref{prop_m_dependence}, we use Lemma \ref{lem_taylor_expansion_smooth_function_diamond} {\bf (i)} instead of Lemma \ref{lem_taylor_expansion_smooth_function} and Lemma \ref{lem_taylor_expansion_smooth_function_diamond_II} {\bf (i)} instead of Lemma \ref{lem_taylor_expansion_smooth_function_II} {\bf (i)}. We still require Lemma \ref{lem_taylor_expansion_smooth_function_II} {\bf (ii)}. This yields the inequality
\begin{align}\label{eq_prop_m_dep_it_1}
\Ad_{T_n} &\lesssim \oo\big(n^{1-\frac{p}{2}}\big) (\log n)^{-\bd }N^{\frac{p-3}{2}} + N^{-1}(\log n)^{2} \sup_{x \in \R}\big|{\Delta}_m^{\diamond}(x)\big|.
\end{align}
An application of Lemma \ref{lem_diamond_Sn_expansions_II} now yields
\begin{align}\label{eq_prop_m_dep_it_2}
\sup_{x \in \R}\big|{\Delta}_m^{\diamond}(x)\big| \lesssim m^{1-\frac{p}{2} + \delta},
\end{align}
where $\delta > 0$ is arbitrarily small. Since $m = n^{\md}$ for $\md<1$ arbitrarily close to one, the claim follows by combining \eqref{eq_prop_m_dep_it_1} and \eqref{eq_prop_m_dep_it_2}.
\end{proof}

\section{Proof of Proposition \ref{prop_m_dep_it_quad}}\label{sec_proof_prop_quad}

The basic idea is again the iterative argument, previously used in the proof of Proposition \ref{prop_m_dep_it}. However, exploiting the fact that $p \geq 4$, we can refine it. As a first step, we show Lemma \ref{lem_taylor_expansion_smooth_function_quad} and Lemma \ref{lem_taylor_expansion_smooth_function_quad_II} below.

\begin{lem}\label{lem_taylor_expansion_smooth_function_quad}
Grant Assumption \ref{ass_dependence} for $p \geq 4$, and let $f$ be a smooth function such that $\sup_{x \in \R}|f^{(s)}(x)| \leq 1$ for $s = 0,\ldots,8$. Then for $\tau_n \geq {\co_{\tau}} \sqrt{\log n}$, ${\co_{\tau}} > 0$ sufficiently large 
\begin{align*}
&\text{{\bf (i)}} \,\, \max_{1 \leq j \leq N}\bigl\|\E_{\FF_m} \ff\bigl(\xi_N \overline{V}_{j}\bigr) - \E_{}\ff\bigl(\xi_N \overline{Z}_{j}\bigr)\bigr\|_1 \lesssim  \big(|\xi_N|^2 + |\xi_N|^3+|\xi_N|^4\big) m^{-1} \\& \quad + \tau_n^5(|\xi_N|^5 + |\xi_N|^8)m^{-1} + \sup_{x \in \R}\big|{\Delta}_{j,m}(x)\big|\bigl(\tau_n^5|\xi_N|^5 + |\tau_N|^6 |\xi_N|^6).\\
&\text{{\bf (ii)}} \,\, \text{We may replace ${\Delta}_{j,m}(x)$ with ${\Delta}_{j,m}^{\diamond}(x)$ in the above estimate.}\\
&\text{{\bf (iii)}} \,\, \text{If we replace ${\Delta}_{j,m}(x)$ with ${\Delta}_{j,m}^{\diamond}(x)$, this also bounds}  \\& \quad \quad \max_{1 \leq j \leq N}\bigl\|\E_{\FF_m} \ff\bigl(\xi_N \overline{V}_{j}^{\diamond}\bigr) - \E_{} \ff\bigl(\xi_N \overline{Z}_{j}^{\diamond}\bigr)\bigr\|_1.
\end{align*}
\end{lem}


\begin{proof}[Proof of Lemma \ref{lem_taylor_expansion_smooth_function_quad}]
The proof goes along the same lines as for Lemma \ref{lem_taylor_expansion_smooth_function}. The essential difference is a slightly different handling of function $g$, defined in \eqref{eq_lem_taylor_exp_g}. We first show {\bf (i)}. Using a Taylor expansion, we have
\begin{align*}
g(x) = h_n(x)\Big(x^4 \ff^{(4)}(0) &+  s t\xi_n x^5 \ff^{(5)}(0) \\&+ s t\xi_n x^5 \int_0^1(1-u)\big(\ff^{(5)}(s t u\xi_n x)-f^{(5)}(0)\big)d u \Big).
\end{align*}
Let
\begin{align}
\check{g}(x) = x^5 \int_0^1(1-u)\big(\ff^{(5)}(s t u\xi_n x)-\ff^{(5)}(0)\big)d u \, h_n(x).
\end{align}
Replacing $g$ with $\check{g}$, we now proceed as in \eqref{eq_lem_taylor_exp_10} plus the remaining steps in Lemma \ref{lem_taylor_expansion_smooth_function}. This yields the estimate
\begin{align}\label{eq_lem_taylor_4_eq_4}\nonumber
&\Bigl\|\E_{\FF_m} s t \xi_N^5  \check{g}\big(\overline{V}_j\big) - \E_{} s t \xi_N^5  \check{g}\big(\overline{Z}_j\big) \Big\|_1
\\&\lesssim m^{-1} \tau_n^5(|\xi_N|^5 + |\xi_N|^8) + \sup_{x \in \R}\big|{\Delta}_{j,m}(x)\big|\bigl(\tau_n^5|\xi_N|^5 + |\tau_N|^6 |\xi_N|^6).
\end{align}

Note the additional factor $\xi_N$ in this estimate, which is crucial. Next, using \eqref{eq_lem_taylor_exp_6} and arguing as in Step 2 of Lemma \ref{lem_taylor_expansion_smooth_function}, we obtain
\begin{align}
\Big\|\E_{\FF_m}(\xi_N \overline{V}_j)^4 \big(1-h_n(\overline{V}_j)\big) \Big\|_1\lesssim \xi_N^4 \oo\big(m^{-1}\big) (\log n)^{-{\bd}},
\end{align}
where $2\bd + 2 < \ad$. By Lemma \ref{lem_quad_expansion}\footnote{Recall that $\overline{Z}_j$ is independent of $\FF_m$}
\begin{align*}
\big\|\E_{\FF_m}\overline{V}_j^4 - \E_{\FF_m} \overline{Z}_j^4 \big\|_1 \lesssim m^{-1}.
\end{align*}
Hence piecing everything together, we arrive at
\begin{align}\label{eq_lem_taylor_4_exp_12} \nonumber
\Bigl\|\E_{\FF_m} (\xi_N\overline{V}_j)^{3} \ff^{(3)}\bigl(t \xi_N \overline{V}_j \bigr) - \E_{\FF_m} (\xi_N\overline{V}_j)^{3}\ff^{(3)}\bigl(0\bigr) \\ \nonumber \indent -\int_0^1 \E_{}\Bigl[(\xi_N \overline{Z}_j)^{4} \ff^{(4)}\big(st \xi_N \overline{Z}_j\big) \hh_n(\overline{Z}_j)\Bigr] d\,s\Bigr\|_1 \\ \nonumber \lesssim  |\xi_N|^3 \tau_n^{-1}\log(n)^{-\bd} \oo\big(m^{-1}\big) + \tau_n^5(|\xi_N|^5 + |\xi_N|^8)m^{-1} \\ +|\xi_N|^4 m^{-1} + \sup_{x \in \R}\big|{\Delta}_{j,m}(x)\big|\bigl(\tau_n^5|\xi_N|^5 + |\tau_N|^6 |\xi_N|^6).
\end{align}
We may now continue as in the proof of Lemma \ref{lem_taylor_expansion_smooth_function} to complete the proof. For {\bf (ii)} and {\bf (iii)}, we may argue as in Lemma \ref{lem_taylor_expansion_smooth_function_diamond} {\bf (i)} and {\bf (ii)}.
\end{proof}

\begin{lem}\label{lem_taylor_expansion_smooth_function_quad_II}
Grant Assumption \ref{ass_dependence} for $p \geq 4$, and let $f$ be a smooth function such that $\sup_{x \in \R}|f^{(s)}(x)| \leq 1$ for $s = 0,\ldots,8$. Then for $\tau_n \geq {\co_{\tau}} \sqrt{\log n}$, ${\co_{\tau}} > 0$ sufficiently large 
\begin{align*}
&\text{{\bf (i)}} \,\, \max_{1 \leq j \leq N}\bigl\|\E_{}\ff\bigl(\xi_N \widetilde{V}_{j}\bigr)  - \E_{}\ff\bigl(\xi_N \widetilde{Z}_{j}\bigr)\bigr\|_1 \lesssim \big(|\xi_N|^2 + |\xi_N|^3+|\xi_N|^4\big) m^{-1} \\& \quad \quad + \tau_n^5(|\xi_N|^5 + |\xi_N|^8)m^{-1} + \sup_{x \in \R}\big|{\Delta}_{j,m}(x)\big|\bigl(\tau_n^5|\xi_N|^5 + |\tau_N|^6 |\xi_N|^6).\\
&\text{{\bf (ii)}} \,\, \text{We may replace ${\Delta}_{j,m}(x)$ with ${\Delta}_{j,m}^{\diamond}(x)$ in the estimate above.}\\
&\text{{\bf (iii)}} \,\, \text{If we replace ${\Delta}_{j,m}(x)$ with ${\Delta}_{j,m}^{\diamond}(x)$, the same bound applies to}  \\&\quad \quad \max_{1 \leq j \leq N}\bigl\|\E_{}\ff\bigl(\xi_N \widetilde{V}_{j}^{\diamond}\bigr)  - \E_{}\ff\bigl(\xi_N \widetilde{Z}_{j}^{\diamond}\bigr)\bigr\|_1.
\end{align*}
\end{lem}

\begin{proof}[Proof of Lemma \ref{lem_taylor_expansion_smooth_function_quad_II}]
We may proceed as in the proof of Lemma \ref{lem_taylor_expansion_smooth_function_quad}.
\end{proof}

Lemmas \ref{lem_taylor_expansion_smooth_function_quad} and \ref{lem_taylor_expansion_smooth_function_quad_II} allow us to establish the key recursion, given in the next lemma.

\begin{lem}\label{lem_diamond_Sn_expansions_quad}
Grant Assumption \ref{ass_dependence} for $p \geq 4$. \hypertarget{deltandia:eq40}{Then} for $N \thicksim (\log n)^{\md}$, $\md \geq 5$,
\begin{description}
  \item[(i)] $\sup_{x \in \R}\big|{\Delta}_n^{\diamond}(x)\big| \lesssim n^{-1}N + (\log n)^{\frac{5}{2}}N^{-\frac{3}{2}}\sup_{x \in \R}\big|{\Delta}_m^{\diamond}(x)\big|$.
  \item[(ii)] $\sup_{x \in \R}\big|{\Delta}_n^{\diamond}(x)\big| \lesssim n^{-1} N$.
\end{description}
\end{lem}

\begin{proof}[Proof of Lemma \ref{lem_diamond_Sn_expansions_quad}]
To obtain {\bf (i)}, we argue as in the proof of Lemma \ref{lem_diamond_Sn_expansions_I}. The key difference is that we use Lemma \ref{lem_taylor_expansion_smooth_function_quad} {\bf (iii)} and Lemma \ref{lem_taylor_expansion_smooth_function_quad_II} {\bf (iii)} instead of Lemma \ref{lem_taylor_expansion_smooth_function_diamond} {\bf (ii)} and Lemma \ref{lem_taylor_expansion_smooth_function_diamond_II} {\bf (ii)} (we still require Lemma \ref{lem_taylor_expansion_smooth_function_diamond_II} {\bf (iii)}) to establish the recursion.
For {\bf (ii)}, we may follow Lemma \ref{lem_diamond_Sn_expansions_II}. Using {\bf (i)} repeatedly, we get, with $\tau_n \thicksim \sqrt{\log n}$, $\co_n = \lfloor \frac{\log n}{\log 2 N}\rfloor$, $N \geq \co \tau_n^{10}$, $\co > 0$ large enough,
\begin{align*}
\sup_{x \in \R}\big|{\Delta}_n^{\diamond}(x)\big| \lesssim n^{-1} N \sum_{k = 0}^{\co_n + 1} \big(\tau_n^5 N^{-\frac{3}{2}}\big)^k (2N)^k + \big(\tau_n^5 N^{-\frac{3}{2}}\big)^{\co_n} \lesssim n^{-1} N.
\end{align*}
\end{proof}

\begin{proof}[Proof of Proposition \ref{prop_m_dep_it_quad}]
We argue as in the proof of Proposition \ref{prop_m_dep_it}, and repeat the proof of Proposition \ref{prop_m_dependence} verbatim with $\tau_n = \co_{\tau} \sqrt{\log n}$ ($\co_{\tau} >0$ large enough) and $N = \co (\log n)^{\md}$, $\md = 5$, $\co >0$ large enough. The difference is that we use Lemma \ref{lem_taylor_expansion_smooth_function_quad} {\bf (ii)}, Lemma \ref{lem_taylor_expansion_smooth_function_quad_II} {\bf (ii)} instead of Lemma \ref{lem_taylor_expansion_smooth_function_diamond} {\bf (i)}, Lemma \ref{lem_taylor_expansion_smooth_function_diamond_II} {\bf (i)}. We still use Lemma \ref{lem_taylor_expansion_smooth_function_II} {\bf (ii)}. This yields the inequality
\begin{align}\nonumber
\Ad_{T_n} &\lesssim n^{-1} N + (\log n)^5 N^{-\frac{3}{2}} \sup_{x \in \R}\big|{\Delta}_m^{\diamond}(x)\big|.
\end{align}
An application of Lemma \ref{lem_diamond_Sn_expansions_quad} {\bf (ii)} then yields the claim.
\end{proof}

\section{Proofs of Section \ref{sec_main}}\label{sec_proof_main}

\begin{proof}[Proof of Theorem \ref{thm_edge}]
We carry out most of the arguments for general $p \geq 3$ so we can recycle the arguments for subsequent proofs. We first show
\begin{align}\label{eq_thm_edge_1}
\sup_{x \in \R}\big|\Delta_n(x)\big| \lesssim  \oo\big(n^{-1/2}\log(n)^{-\bd}\big) + \Cd_{T_n}.
\end{align}
Recall that for $i \geq j$, we have $\mathcal{E}_{i,j} = \sigma\big(\varepsilon_i, \varepsilon_{i-1}, \ldots, \varepsilon_j \big)$ and \hypertarget{XKM:eq41}{let}
\begin{align}\label{defn_X_km:2}
X_{km} = \E\big[X_k\big|\mathcal{E}_{k,k-m}\big], \quad k \in \Z.
\end{align}
Obviously, \hyperref[A1]{\Aone} now implies \hyperref[B1]{\Bone} for $(X_{km})_{k \in \Z}$ due to Jensen's inequality. Next, we show that \hyperref[A2]{\Atwo} implies \hyperref[B2]{\Btwo}. It suffices to consider the case $1 \leq k \leq m$ since $X_{km} = X_{km}^*$ for $k > m$. Observe that (almost surely)
\begin{align*}
X_{km}^* = \E\big[X_{k}^{*}\big| \sigma(\mathcal{E}_{k,1}, \mathcal{E}_{0,-m}^*)\big],
\end{align*}
and
\begin{align*}
 \E\big[X_{k}^{**}\big| \sigma(\mathcal{E}_{k,1}, \mathcal{E}_{0,-m}^*)\big] = \E\big[X_{k}^{}\big| \mathcal{E}_{k,1}\big] = \E\big[X_{k}^{*}\big| \mathcal{E}_{k,k-m}\big].
\end{align*}
Then by the triangle, Jensen's inequality and the above
\begin{align*}
\bigl\|X_{km}-X_{km}^*\big\|_p &= \bigl\|\E\big[X_{k} - X_{k}^{*} + X_{k}^{*} \big| \mathcal{E}_{k,k-m}\big]\\ & \indent -\E\big[X_{k}^{*} - X_{k}^{**} + X_{k}^{**} \big| \sigma(\mathcal{E}_{k,1}, \mathcal{E}_{0,-m}^*)\big]\big\|_p\\& \leq 2\big\|X_{k} - X_{k}^* \big\|_p.
\end{align*}
Next, using the same argument as in the proof of Theorem 1 in ~\cite{Wu_fuk_nagaev}, it follows that (with $S_{nm} = \sum_{k = 1}^n X_{km}$)
\begin{align}\label{eq_thm_edge_2}
\big\|S_n - S_{nm} \big\|_p \lesssim \sqrt{n} m^{-2} \sum_{k = m}^{\infty}k^2 \big\|X_k-X_k^*\big\|_p \lesssim \sqrt{n} m^{-2}.
\end{align}
Using Cauchy-Schwarz inequality, \eqref{eq_thm_edge_2} and Lemma \ref{lem_wu_original}, we obtain
\begin{align}\label{eq_thm_edge_2.1}
n^{-1}\big|\E S_{n}^2 - \E S_{nm}^2 \big| &\lesssim n^{-1}\|S_n- S_{nm}\|_2\|S_n + S_{nm}\|_2 & \lesssim m^{-2}.
\end{align}
Similarly, by H\"{o}lder's inequality and Lemma \ref{lem_wu_original}, we have
\begin{align}\nonumber
\big|\E S_n^3 - \E S_{nm}^3\big| &\lesssim \big\|S_n - S_{nm}\big\|_3\big(\big\|S_{n}\big\|_3^2 +\big\|S_{n}\big\|_3 \big\|S_{nm}\big\|_3 + \big\|S_{nm}\big\|_3^2 \big) \\&\label{eq_thm_edge_2.2} \lesssim \sqrt{n} m^{-2} n \lesssim \sqrt{n}
\end{align}
for $m$ large enough (e.g. $m = \sqrt{n}$). Relation \eqref{eq_thm_edge_2.1} together with Lemma \ref{lem_sig_expressions_relations} and \hyperref[A3]{\Athree} shows that $0 < \co_0 \leq \ss_m^2 \leq \co_1 < \infty$ for all $m \geq m_0$, $m_0$ large enough, and hence \hyperref[B3]{\Bthree} holds. We have thus established the validity of Assumption \ref{ass_dependence} for $(X_{k,m})_{k \in \Z}$, $m$ sufficiently large. To continue with the proof, we note
\begin{align}\nonumber \label{eq_thm_edge_2.3}
\sup_{x \in \R}\big|\P\bigl(S_n \leq x \sqrt{n}\bigr) - {\Psi}_{n}\big(x\big) \big| &\leq \sup_{x \in \R}\big|\P\bigl(S_n \leq x \sqrt{n}\bigr) - \P\big(\overline{Z} + \widetilde{Z} \leq x\big) \big| \\&+ \sup_{x \in \R}\big|\P\big(\overline{Z} + \widetilde{Z} \leq x\big) - {\Psi}_{n}\big(x\big) \big|,
\end{align}
where $\overline{Z}$, $\widetilde{Z}$ are defined with respect to $(X_{k,m})_{k \in \Z}$. We first show
\begin{align}\label{eq_thm_edge_3.1}
\sup_{x \in \R}\bigl|\P\bigl(S_n \leq x \sqrt{n}\bigr) - \P\big(\overline{Z} + \widetilde{Z} \leq x\big) \bigr| \lesssim \oo\big(n^{-1/2}\log(n)^{-\bd}\big) + \Cd_{T_n}.
\end{align}
To this end, we observe that by Lipschitz continuity and \eqref{eq_thm_edge_2}, we have
\begin{align}\label{eq_thm_edge_3}
\bigl|\E e^{\ic \xi S_n/\sqrt{n}} - \E e^{\ic \xi S_{nm}/\sqrt{n}} \bigr| \lesssim |\xi|m^{-2}.
\end{align}
From Berry's classical smoothing inequality (cf. ~\cite[Lemma 1, XVI.3]{fellervolume2}) we derive, using \eqref{eq_thm_edge_3},
\begin{align}\label{thm_edge_smoothing} \nonumber
\sup_{x \in \R}\bigl|\P\bigl(S_n \leq x \sqrt{n}\bigr) &- \P\big(\overline{Z} + \widetilde{Z} \leq x\big) \bigr| \\& \nonumber \lesssim \int_{-T_n}^{T_n}\Bigl|\E e^{\ic \xi S_n/\sqrt{n}}  - \E e^{\ic \xi (\overline{Z}+ \widetilde{Z})}\Bigr| \frac{1}{|\xi|}\, d\xi + \Cd_{T_n}\\& \lesssim \Ad_{T_n} + \frac{T_n}{m^{2}} + \Cd_{T_n}.
\end{align}
Here, $\Ad_{T_n}, \overline{Z}$, $\widetilde{Z}$ are defined with respect to $(X_{k,m})_{k \in \Z}$, while $\Cd_{T_n}$ is defined with respect to $(X_{k})_{k \in \Z}$. Proposition \ref{prop_m_dependence} then yields \eqref{eq_thm_edge_3.1}, selecting $m$ large enough (e.g. $m = n^{\frac{7}{8}}$). Having in mind \eqref{eq_thm_edge_2.3}, it remains \hypertarget{Psinmx:eq42}{to} show
\begin{align}\label{eq_thm_edge_3.2}
\sup_{x \in \R}\big|\Psi_n(x) - \P\big(\overline{Z} + \widetilde{Z} \leq x\big)\big| = \oo\big(n^{-1/2}\log(n)^{-\bd}\big).
\end{align}
To this end, we define $\Psi_{n,m}(x)$ in analogy to \eqref{defn_PSI_m} with respect to $S_{n,m}$. Then
\begin{align*}
\sup_{x \in \R}\big|\P\big(\overline{Z} + \widetilde{Z} \leq x\big) - {\Psi}_{n}\big(x\big) \big| &\leq \sup_{x \in \R}\big|\P\big(\overline{Z} + \widetilde{Z} \leq x\big) - {\Psi}_{n,m}\big(x\big) \big| \\&+ \sup_{x \in \R}\big|\Psi_{n,m}\big(x\big) - {\Psi}_{n}\big(x\big) \big|.
\end{align*}
By Lemma \ref{lem_edge_for_FF}\footnote{Note: $\Psi_n$ in Lemma \ref{lem_edge_for_FF} becomes $\Psi_{n,m}$}, the first term on the RHS is bounded by $m^{-1}$. For the second term, it suffices to show that the second and third moments differ by at most $m^{-1}$ due to Lemma \ref{lem_psi_compare}. However, we have already shown this in \eqref{eq_thm_edge_2.1} and \eqref{eq_thm_edge_2.2}, and obtained the even stronger result
\begin{align}\label{eq_thm_edge_ss_bound}
\sup_{x \in \R}\big|{\Psi}_{n,m}\big(x\big) - {\Psi}_n\big(x\big)\big| \lesssim n^{-1},
\end{align}
and hence \eqref{eq_thm_edge_3.2} holds. But \eqref{eq_thm_edge_3.2} with \eqref{eq_thm_edge_3.1} yields \eqref{eq_thm_edge_1} via \eqref{eq_thm_edge_2.3}, which completes the first part of the proof. Since {\bf (ii)} $\Rightarrow$ {\bf (i)} is obvious, it remains to show {\bf (i)} $\Rightarrow$ {\bf (ii)}. To this end, denote by $v_T$ the smoothing density
\begin{align}\label{defn_vT}
v_T(t) = \frac{1}{\pi}\frac{1 - \cos(T t)}{1 + t^2 T}, \quad x \in \R, T \in \N.
\end{align}
Recall $\Delta_n(x) = \P\bigl(S_n \leq x \sqrt{n} \bigr) - \Psi_n(x)$, and \hypertarget{deltaNT:eq43}{put}
\begin{align}\label{defn_DeltanT}
\Delta_n^T(x) = \int_{\R} \Delta_n(x-t) v_T(t) dt.
\end{align}
Since $v_T(x)$ is the density of a probability measure, we obtain 
\begin{align}\label{eq_thm_optimal_1}
\sup_{x \in \R}|\Delta_n^T(x)| \leq \sup_{x \in \R}|\Delta_n(x)|.
\end{align}
Moreover, by Fourier inversion
\begin{align}\label{eq_thm_optimal_1.5}
\Delta_n^T(x) = \frac{1}{2 \pi} \int_{-T}^T e^{-\ic \xi x } \frac{\E e^{\ic \xi S_n/\sqrt{n}} - \E e^{\ic \xi (\overline{Z} + \widetilde{Z})}}{- \ic \xi}\Big(1 - \frac{|\xi|}{T} \Big) d \xi.
\end{align}
Setting $T=n$, $T_n \thicksim \sqrt{n}$ and using \eqref{eq_thm_edge_3} together with Proposition \ref{prop_m_dependence}, the reverse triangle inequality gives
\begin{align}\label{eq_thm_optimal_1.7} \nonumber
|\Delta_n^T(x)| &\gtrsim \Bigl|\int_{T_n \leq |\xi| \leq n} e^{-\ic \xi x}\E\bigl[e^{\ic \xi S_n/\sqrt{n}} \bigr] \Bigl(1 - \frac{|\xi|}{n}\Bigr) \frac{1}{\xi} d \, \xi \Bigr| - \oo\big(n^{-\frac{1}{2}}\log(n)^{-\bd}\big)\\&= |\Td_{T_n}^n(x)| - \oo\big(n^{-\frac{1}{2}}\log(n)^{-\bd}\big).
\end{align}
Combining this with inequality (\ref{eq_thm_optimal_1}), we arrive at
$$
\oo\bigl(n^{-\frac{1}{2}}\bigr) = \sup_{x \in \R}|\Delta_n(x)| \geq \sup_{x \in \R}|\Delta_n^T(x)|  \gtrsim \sup_{x \in \R}|\Td_{T_n}^n(x)| - \oo\big(n^{-\frac{1}{2}}\log(n)^{-\bd}\big).
$$
This implies
\begin{align*}
\Cd_{T_n} \lesssim \oo\bigl(n^{-\frac{1}{2}}\bigr) + n^{-1} = \oo\bigl(n^{-\frac{1}{2}}\bigr),
\end{align*}
which completes the proof.
\end{proof}

\begin{proof}[Proof of Theorem \ref{thm_edge_frac}]
We can argue as in the proof of Theorem \ref{thm_edge}. The main difference is that we use Proposition \ref{prop_m_dep_it} instead of Proposition \ref{prop_m_dependence}. 
\end{proof}

\section{Proofs of Section \ref{sec:appl:I}}\label{sec:proof:applI}

The method of proof for the results of Section \ref{sec:appl:I} is of relevance for those of Section \ref{sec_wasserstein}, hence we supply the corresponding arguments first.

\begin{lem}\label{lem_nonunif}
Assume that Assumption \ref{ass_dependence_main} holds for $0 \leq q < \ad-2$. Then there exists a constant $\co > 0$ such that for any $x \in \R$
\begin{align*}
\bigl|\Delta_n(x) \bigr| \leq \frac{\co}{1 + |x|^{q}}\Big(\oo(n^{-\frac{1}{2}}) + (\log n)^{\frac{q}{2}}\Cd_{T_n}\Big),
\end{align*}
where $T_n = \co_T n^{1/2}$ with $\co_T > 0$ sufficiently small. 
If $3 < p  < 4$, then for any $q>0$ there exists $\delta > 0$ (arbitrarily small) such that
\begin{align*}
\bigl|\Delta_n(x) \bigr| \leq \frac{\co}{1 + |x|^{q}}\Big(n^{-\frac{p}{2}+1+\delta} + (\log n)^{\frac{q}{2}}\Cd_{T_n}\Big).
\end{align*}
\end{lem}

\begin{proof}[Proof of Lemma \ref{lem_nonunif}]
By Lemma \ref{th:Moritz}
\begin{align*}
\P\big(|S_n| \geq x \sqrt{\ss^2 n} \big) \lesssim n^{-\frac{p}{2} + 1}x^{-\frac{p}{2}}
\end{align*}
for $x \geq \co \sqrt{\log n}$ and $\co > 0$ sufficiently large. One readily verifies an analogous bound for $\Psi_n(x)$ using standard Gaussian tail bounds. On the other hand, Theorem \ref{thm_edge} implies
\begin{align*}
\bigl|\Delta_n(x) \bigr| \lesssim \frac{(\log n)^{\frac{q}{2}}}{1 + |x|^{q}}\Big(\oo\big(n^{-\frac{1}{2}}\log(n)^{-\bd}\big) + \Cd_{T_n}\Big), \quad |x| \leq \co \sqrt{\log n},
\end{align*}
hence the first claim. The second follows in the same manner, using Theorem \ref{thm_edge_frac} instead of Theorem \ref{thm_edge}.
\end{proof}

\begin{proof}[Proof of Theorem \ref{thm_smooth}]
To lighten the notation, we assume without loss of generality that $s_n^2 = 1$ in the sequel. Note first that due to \eqref{eq_derivative_condition_1} and $|f(0)| < \infty$ by continuity
\begin{align}\label{eq_thm_smooth_1}
\bigl|f(x)\bigr| \leq \big|f(0)\big| + |x|\sup_{y \in \R}\bigl|f^{(1)}(y)\bigr|.
\end{align}
It follows that $\E\bigl[|f(S_n/\sqrt{n})|\bigr] < \infty$ is well-defined. For $a > 0$, $b \geq 6$ even, let $H_k \stackrel{d}{=} G_{a,b}$ be a sequence of i.i.d. random variables, independent of $(X_k)_{k \in \Z}$, where $G_{a,b}$ is defined in \eqref{eq_g_ab}. Put $f_n\big(x\big) = f\big(x/\sqrt{n}\big)$. A Taylor expansion yields
\begin{align*}
&f_n\bigl(S_n + H_{n} - H_0\big) - f_n\big(S_n\big) = (H_{n} - H_0)f_n^{(1)}\bigl(S_n\bigr)
\\& + (H_{n} - H_0) \int_0^1 \Big(f_n^{(1)}\bigl(S_n + t (H_{n} - H_0)\bigr) - f_n^{(1)}\bigl(S_n\bigr) \Big)\,dt.
\end{align*}
Fix $K > 1$ and consider the interval $\mathrm{I}_K = [-K,K]$. Since $f^{(1)}$ restricted to $\mathrm{I}_K$ is uniformly continuous, for any $\eta > 0$ there exists $\delta_K > 0$ such that $|f^{(1)}(x)-f^{(1)}(y)| < \eta$ for $|x-y| < \delta_K$ and $x,y \in \mathrm{I}_K$. Let
\begin{align*}
\mathcal{A}_{K} = \big\{S_n \in \sqrt{n} \mathrm{I}_{K-1}, \big\}, \quad \mathcal{B}_{\delta_K} = \big\{|H_n - H_0| < \sqrt{n}\delta_K \big\}.
\end{align*}
By Lemma \ref{lem_wu_original}, it follows that
\begin{align*}
\P\big((\mathcal{A}_{K} \cup \mathcal{B}_{\delta_K})^c\big) &\leq \P\big(|S_n| \geq (K-1) \sqrt{n} \big) + \P\big(|H_n - H_0| \geq \sqrt{n}\delta_K \big) \\&\lesssim K^{-p} + n^{\frac{-b+2}{2}} \delta_K^{-b+2}.
\end{align*}

Since $\sup_{x \in \R}|f_n^{(1)}(x)| \leq n^{-1/2}\mathrm{D}_f$, the above yields
\begin{align*}
&\sqrt{n}\Big\|\int_0^1 \Big(f_n^{(1)}\bigl(S_n + t (H_{n} - H_0)\bigr) - f_n^{(1)}\bigl(S_n\bigr) \Big)\,dt\Big\|_2 \\&\lesssim \sqrt{ \P\big((\mathcal{A}_{K} \cup \mathcal{B}_{\delta_K})^c\big)} + \eta \lesssim \sqrt{K^{-p} + n^{\frac{-b+2}{2}} \delta_K^{-b+2}} + \eta.
\end{align*}
In particular, by appropriate choices of $\eta, K,\delta_K$, we get
\begin{align*}
\Big\|\int_0^1 \Big(f_n^{(1)}\bigl(S_n + t (H_{n} - H_0)\bigr) - f_n^{(1)}\bigl(S_n\bigr) \Big)\,dt\Big\|_2 = \oo\big(n^{-1/2}\big)
\end{align*}
as $n \to \infty$. Since $\E[H_0] = 0$ and $\|H_n - H_0\|_2 < \infty$, the above and Cauchy-Schwarz inequality give
\begin{align}
\bigl|\E\bigl[f_n(S_n + H_{n} - H_0) - f_n(S_n)\bigr] \bigr| & = \oo\big(n^{-\frac{1}{2}}\big).
\end{align}
Recall ${X}_k^{\diamond} = X_k + H_{k} - H_{k-1}$, ${S}_n^{\diamond} = \sum_{k = 1}^n {X}_k^{\diamond} = S_n + H_n - H_0$, and that for a real valued random variable $Y$
\begin{align}\label{eq_thm_smooth_4}
\E\bigl[f(Y) - f(0)\bigr] = \int_0^{\infty}f^{(1)}(y) \P\bigl(Y \geq y \bigr)dy - \int_{-\infty}^{0}f^{(1)}(y)\P\bigl(Y \leq y \bigr)dy.
\end{align}
Next, we note that straightforward computations give
\begin{align*}
(s_n^{\diamond})^2 = s_n^2 + \OO\big(n^{-1}\big), \quad (\kapp_n^{\diamond})^3 = \kapp_n^3 + \OO\big(n^{-1}\big).
\end{align*}
Hence Lemma \ref{lem_psi_compare} yields
\begin{align}\label{eq:thm:smooth:4.2}
\sup_{x \in \R}(x^2+1)\big|\Psi_n^{\diamond}(x) - \Psi_n(x) \big| \lesssim n^{-1}.
\end{align}
Since $\sup_{x \in \R}|f^{(1)}(x)| \leq \mathrm{D}_f$, \eqref{eq:thm:smooth:4.2} and Lemma \ref{lem_nonunif} yield
\begin{align*}
&\Big|\int_{0}^{\infty}f^{(1)}(y)\Big(\P\bigl({S}_n^{\diamond} \geq y \sqrt{n} \bigr) - \big(1- \Psi_n^{}(y)\big)\Big)dy \Big|\\& \lesssim \Big|\int_{0}^{\infty}f^{(1)}(y)\Big(\P\bigl({S}_n^{\diamond} \geq y \sqrt{n} \bigr) - \big(1- \Psi_n^{\diamond}(y)\big)\Big)dy \Big| + \frac{1}{n}
\\&\lesssim \oo\big(n^{-\frac{1}{2}}\big) + \big(\log n\big)^{\frac{q}{2}} \Cd_{T_n}, \quad q > 1,
\end{align*}
where $T_n = \co_T \sqrt{n}$. Selecting $a > 0$ such that $\co_T > a b$, we get from \eqref{eq_thm_smooth_fourier}
\begin{align}\label{eq_thm_smooth_5.5}
\sup_{x \in \R}\bigl|\Td_{T_n}^{n}(x)\bigr| \leq \int_{\co_T \sqrt{n} \leq |\xi| \leq n} \big|\E e^{\ic \xi n^{-\frac{1}{2}} H_0} \big|^2 \frac{1}{|\xi|} d \xi = 0.
\end{align}
Hence
\begin{align*}
&\Big|\int_{0}^{\infty}f^{(1)}(y)\Big(\P\bigl({S}_n^{\diamond} \geq y \sqrt{n} \bigr) - \big(1- \Psi_n^{}(y)\big)\Big)dy \Big| = \oo\big(n^{-\frac{1}{2}}\big).
\end{align*}
In an analogous manner, one establishes
\begin{align*}
\Big|\int_{-\infty}^{0}f^{(1)}(y)\Big(\P\bigl({S}_n^{\diamond} \leq y \sqrt{n} \bigr) &- \big(1- \Psi_n^{}(y)\big)\Big)dy \Big|  = \oo\big(n^{-\frac{1}{2}}\big).
\end{align*}
Piecing everything together, the claim follows.
\end{proof}

\begin{proof}[Proof of Theorem \ref{thm_hoelder}]
We proceed as in the proof of Theorem \ref{thm_smooth}. A Taylor expansion yields
\begin{align*}
&\big|\E\big[f_n(S_n + H_{n} - H_0) - f_n(S_n)\big]\big| \lesssim \mathrm{L} n^{-s/2} \E|H_{n} - H_0|^{s}.
\end{align*}
Using Lemma \ref{lem_nonunif}, the claim now follows as in Theorem \ref{thm_smooth}.
\end{proof}

\section{Proofs of Section \ref{sec_wasserstein}}\label{sec_proof_wasser}

\begin{proof}[Proof of Theorem \ref{thm_wasserstein}]
In broad brushes, the proof is very similar to the one of Theorem \ref{thm_edge}. The main difference is that we replace '$\sup$'s with integrals. Recall the well-known representation
\begin{align}\label{wasserstein_rep_3}
W_1\big(F,G\big) = \int_{\R}\big|F(x) - G(x)\big| d x
\end{align}
for distribution functions $F$, $G$ and let $\tau_n = \co_{\tau} \sqrt{\log n}$, $\co_{\tau} > 0$ large enough. Then by Lemma \ref{th:Moritz}
\begin{align}\label{eq_thm_wasserstein_3}\nonumber
&\int_{|x| > \tau_n} \bigl|\P\bigl(S_n \leq x \sqrt{n}\bigr) - \P\bigl(L_n \leq x \sqrt{n}\bigr) \bigr| dx \\&\leq \int_{x > \tau_n}\Big(\P\big(|S_n| > x \sqrt{n}\big) + \P\big(|L_n| > x \sqrt{n}\big)\Big) d x \lesssim n^{-\frac{p}{2}+1}.
\end{align}
On the other hand, proceeding similar as in \eqref{thm_edge_smoothing} (we integrate, instead of taking the supremum), we get
\begin{align}\label{thm_wasser_smoothing} \nonumber
\int_{-\tau_n}^{\tau_n} \bigl|\P\bigl(S_n \leq x \sqrt{n}\bigr) - {\Psi}_n\big(x\big) \bigr| dx &\lesssim \tau_n \Ad_{T_n} + \tau_n T_n m^{-2} + \Id_{T_n} \\&\lesssim \oo\big(n^{-\frac{1}{2}}\big) (\log n)^{-\bd + \frac{1}{2}} + \Id_{T_n}
\end{align}

for $m$ large enough (e.g. $m = n^{\frac{7}{8}}$). In addition, since $\E L_n^2 = s_n^2$, $\E L_n^3 = \kapp_n^3$, $L_n$ admits the Edgeworth expansion (c.f. ~\cite{petrov_book_sums})
\begin{align}\label{eq_thm_wasserstein_edge_L}
\sup_{x \in\R} \big(|x|^3+1\big)\big|\P\big(L_n \leq x \sqrt{n}\big) - \Psi_n\big(x\big) \big| \lesssim n^{-1}.
\end{align}
It follows that
\begin{align}\label{eq_thm_wasserstein_4}
\int_{-\tau_n}^{\tau_n} \bigl|\Psi_n(x) - \P\big(L_n \leq x \sqrt{n}\big) \bigr| dx \lesssim  n^{-1}.
\end{align}
Combining \eqref{eq_thm_wasserstein_3}, \eqref{thm_wasser_smoothing} and \eqref{eq_thm_wasserstein_4}, the triangle inequality gives
\begin{align}\label{eq_thm_wasserstein_5}
W_1\big(\P_{S_n/\sqrt{n}}, \MM_n\big) \leq \oo\big(n^{-\frac{1}{2}}\big) (\log n)^{\frac{1}{2} - \bd} + \Id_{T_n}.
\end{align}
Next, we establish the bound
\begin{align}\label{eq_thm_wasserstein_1}
W_1\big(\P_{S_n/\sqrt{ n}}, \MM_n\big) \lesssim n^{-\frac{1}{2}}.
\end{align}
To this end, we use Representation \ref{wasserstein_2}. As in the proof of Theorem \ref{thm_smooth}, let $(H_k)_{k \in \Z}$ be an i.i.d. sequence with $H_k \stackrel{d}{=} G_{a,b}$, $b = 6$, $a > 0$, ${S}_n^{\diamond} = S_n + H_n - H_0$. Applying \eqref{eq_thm_wasserstein_5}, we get
\begin{align}
W_1\big(\P_{S_n^{\diamond}/\sqrt{n}}, \MM_n^{\diamond}\big) \lesssim \oo\big(n^{-\frac{1}{2}}\big) (\log n)^{\frac{1}{2} - \bd} + \Id_{T_n}^{\diamond},
\end{align}
where $\MM_n^{\diamond}$ is defined in the obvious way. Using \eqref{eq_thm_smooth_5.5}, we deduce $\Id_{T_n}^{\diamond} = 0$ for $a > 0$ small enough.
On the other hand, by Representation \ref{wasserstein_2} and Lipschitz-continuity
\begin{align*}
&W_1\big(\P_{S_n^{\diamond}/\sqrt{n}},\P_{S_n^{}/\sqrt{n}} \big) = \sup_{f \in \mathcal{H}_1^1}\big|\E f\big({S}_n^{\diamond}/\sqrt{n}\big)\big] - \E f\big({S}_n^{}/\sqrt{n}\big)\big]\big| \\&\lesssim n^{-\frac{1}{2}}\big\|H_n - H_0\big\|_1 \lesssim n^{-\frac{1}{2}},
\end{align*}
and the same argument also gives $W_1\big(\MM_n, \MM_n^{\diamond}\big)\lesssim n^{-1/2}$. The triangle inequality, in view of the above estimates, then yields
\begin{align*}
W_1\big(\P_{S_n^{}/\sqrt{n}},\MM_n \big) &\leq W_1\big(\P_{S_n^{\diamond}/\sqrt{n}},\P_{S_n^{}/\sqrt{n}} \big) + W_1\big(\P_{S_n^{\diamond}/\sqrt{n}},\MM_n^{\diamond} \big) + W_1\big(\MM_n, \MM_n^{\diamond}\big)\\&\lesssim n^{-1/2}.
\end{align*}
Together with \eqref{eq_thm_wasserstein_5}, we have thus established
\begin{align*}
W_1\big(\P_{S_n/\sqrt{\ss^2 n}}, \Psi_n\big) \lesssim \oo\big(n^{-\frac{1}{2}}\big)(\log n)^{\frac{1}{2} - \bd} + n^{-\frac{1}{2}} \wedge \Id_{T_n}.
\end{align*}
Since {\bf (ii)} $\Rightarrow$ {\bf (i)} is obvious, it remains to show {\bf (i)} $\Rightarrow$ {\bf (ii)}. To this end, recall the smoothing density $v_T$ and $\Delta_n^T$, given in \eqref{defn_vT} and \eqref{defn_DeltanT}. Note that particularly due to \eqref{eq_thm_wasserstein_edge_L}, we have
\begin{align*}
W_1\big(\P_{S_n/\sqrt{n}} \ast v_T, \MM_n \ast v_T\big) \geq \int_{\R}\big|\Delta_n^T(x)\big|dx - \OO\big(n^{-1}\big).
\end{align*}
Since $W_1$ is an ideal metric (or by Representation \eqref{wasserstein_rep_3} and Tonelli), we have 
\begin{align}
W_1\big(\P_{S_n/\sqrt{ n}} \ast v_T, \MM_n \ast v_T\big) \leq W_1\big(\P_{S_n/\sqrt{ n}}, \MM_n\big).
\end{align}
Together with relation \eqref{eq_thm_optimal_1.5} and $T=n$ this implies
\begin{align*}
\oo\big(n^{-\frac{1}{2}}\big) &= W_1\big(\P_{S_n/\sqrt{n}}, \MM_n\big) \geq W_1\big(\P_{S_n/\sqrt{n}} \ast v_n, \MM_n \ast v_n\big) \\&\geq \int_{-\tau_n}^{\tau_n}\big|\Delta_n^n(x)\big|dx - \OO\big(n^{-1}\big)  \gtrsim \int_{-\tau_n}^{\tau_n}\big|\Td_{T_n}^n(x)\big|dx - \oo\big(n^{-\frac{1}{2}}\big)\log(n)^{\frac{1}{2}-\bd}.
\end{align*}
Since $\ad > 3$ by assumption we may select $\bd > 1/2$ and conclude
\begin{align*}
\Id_{T_n} \leq \oo\big(n^{-\frac{1}{2}}\big) + 2\tau_n n^{-1} = \oo\big(n^{-\frac{1}{2}}\big).
\end{align*}
\end{proof}

\begin{proof}[Proof of Theorem \ref{thm_wasserstein_frac}]
We can argue as in the proof of Theorem \ref{thm_wasserstein}, based on Propositions \ref{prop_m_dependence} and \ref{prop_m_dep_it_quad}.
\end{proof}

\begin{proof}[Proof of Corollary \ref{cor_clt_wasser}]
Lemma \ref{lem_S_n_third_expectation} implies $|\kapp_n^3| \lesssim n^{-1/2}$. Using the Edgeworth expansion \eqref{eq_thm_wasserstein_edge_L}, we obtain the bound
\begin{align*}
\int_{\R}\big|\Phi(x/s_n) - \P(L_n \leq x \sqrt{n})\big| d x \lesssim n^{-1/2},
\end{align*}
and hence from Representation \eqref{wasserstein_rep_3}, it follows that $W_1(\Phi, \MM_n)\lesssim n^{-1/2}$. The triangle inequality and Theorem \ref{thm_wasserstein} then yield the claim.
\end{proof}

\begin{proof}[Proof of Corollary \ref{cor_clt_Lp}]
Due to Theorem \ref{thm_edge} ($|\kappa_n^3| \lesssim n^{-1/2}$ by Lemma \ref{lem_S_n_third_expectation}) we have
\begin{align*}
\sup_{x \in \R}\big|\P\big(S_n \leq x \sqrt{n}\big) - \Phi\big(x/s_n\big) \big| \lesssim n^{-1/2}.
\end{align*}
Using Representation \eqref{wasserstein_rep_3}, we thus obtain
\begin{align*}
&\int_{\R}\big|\P(S_n \leq x \sqrt{n}) - \Phi(x/s_n)\big|^q d x \\&\leq \sup_{x \in \R}\big|\P(S_n \leq x \sqrt{n}) - \Phi(x)\big|^{q-1}  \int_{\R}\big|\P(S_n \leq x \sqrt{n}) - \Phi(x)\big| d x \\&\lesssim n^{-\frac{q-1}{2}} W_1\big(\P_{S_n/\sqrt{n}}, \Phi \big).
\end{align*}
Corollary \ref{cor_clt_wasser} now yields the result.
\end{proof}

\newpage

\begin{center}
\begin{table}
\begin{tabular}{l | l}
 Notation  \hypertarget{table:fin}{Table} XI & Page References \\
   \quad & \quad \\
\hline
  $ S_n $  ~\hyperlink{sn:eq2}{2}, \hyperlink{sn:eq15}{22}, \hyperlink{sn:eq16}{23} \qquad  \qquad  \qquad \qquad \qquad \qquad   &   $\Psi_n(x)$ ~\hyperlink{psin:eq2}{2}, \hyperlink{psin:eq3}{3},  \hyperlink{psin:eq5}{6},  \hyperlink{psin:eq9}{14}        \qquad  \qquad  \qquad \qquad  \\
    $\Cd_{a}$  ~\hyperlink{ca:eq5}{7}   & $\Lambda_{q,p}$  ~ \hyperlink{lambdaqp:eq5}{6}, \hyperlink{lamdaqp:eq12}{19} \\
  $ \mathcal{G}_m $  ~\hyperlink{gm:eq2}{3}  &  $\lambda_{k,p}$ ~ \hyperlink{littlelamkp:eq3}{4}, \hyperlink{littlelamkp:eq5}{6}, \hyperlink{littlelamkp:eq16}{23} \\
   $\Id_{T_n}$  ~ \hyperlink{scttn:eq4}{4}, \hyperlink{scttn:eq10}{10} &  $\|\cdot\|_p$ ~ \hyperlink{normp:eq4}{6} \\
    $\mathcal{A}^{c}$ ~ \hyperlink{scAc:eq14}{22}  &  $\tilde{\kappa}_j$ ~ \hyperlink{tilkappaj:eq13}{20}, \hyperlink{tilkappaj:eq23}{30}    \\
    $s_n^2$ ~ \hyperlink{snsquare:eq5}{6}, \hyperlink{snsquare:eq15}{22}  &  $\kappa_n^3$ ~ \hyperlink{kappan3:eq5}{6}, \hyperlink{kappan3:eq10}{10} \\
    $\Td_{a}^{b}(x)$ ~ \hyperlink{TTab:eq5}{7} & $T_n$ ~\hyperlink{bigtn:eq6}{7} \\
     $\mathcal{E}_k$ ~ \hyperlink{EEk:eq15}{22} & $\tilde{\nu}_j$ ~ \hyperlink{tildeNuj:eq13}{20}, \hyperlink{tildeNuj:eq23}{30} \\
    $\bar{\sigma}_j $ ~ \hyperlink{barsigmajj:eq13}{20}, \hyperlink{barsigmajj:eq23}{30} &    $\ss^2 $ ~ \hyperlink{sss:eq5}{6}  \\
    $\Delta_n(x)$ ~ \hyperlink{Deltasubn:eq6}{7}   & $\kappa^3$ ~\hyperlink{kappa3:eq6}{7}  \\
    $H_k$ ~ \hyperlink{bigHk:eq7}{8}   &  $\bd$ ~ \hyperlink{BBD:eq6}{7} \\
      $\mathcal{E}_k^{*}$ ~ \hyperlink{scEkstar:eq15}{22}    & $\bar{\nu}_j$ ~ \hyperlink{barNuj:eq13}{20}, \hyperlink{barNuj:eq23}{30} \\
      $\widetilde{V}_j^{*}$ ~ \hyperlink{VVj:eq18}{25}  & $R_1^{(l,**)}$ ~ \hyperlink{RR1:eq18}{25} \\
    $D_f$ ~ \hyperlink{bigdf:eq9}{14} & $\phi(\cdot)$ ~ \hyperlink{littlephi:eq5}{7}  \\
     $W_1(\cdot,\cdot)$ ~ \hyperlink{Wone:eq10}{9}  & $\Id_a$ ~ \hyperlink{TTa:eq10}{10} \\
      $\mathcal{F}_m^{*}$ ~ \hyperlink{scFmstar:eq22}{29}   & $\mathcal{F}_m$ ~ \hyperlink{FFm:eq12}{19}  \\
     $\E_{\mathcal{H}}[\cdot]$ ~ \hyperlink{EEH:eq12}{20} &  $\tilde{\sigma}_j$ ~ \hyperlink{tilsigmajj:eq13}{20}, \hyperlink{tilsigmajj:eq23}{30} \\
     $ \Ad_{T_n} $ ~ \hyperlink{UUTn:eq14}{21}  &  $X_k^{*}$ ~ \hyperlink{bigXkstar:eq5}{6}, \hyperlink{bigXkstar:eq15}{22}      \\
   $\mathcal{S}_{kj}$ ~ \hyperlink{scSkj:eq17}{24}& $\overline{V}_j^{*}$ ~\hyperlink{BVJS:eq17}{25}, \hyperlink{BVJS:eq22}{29}, \hyperlink{BVJS:eq23}{30}, \hyperlink{BVJS:eq32}{39} \\
     $\overline{V}_1^{(*,>l)}$ ~ \hyperlink{barV1:eq18}{25} & $ \overline{V}_1^{(**,>l)}$ ~ \hyperlink{barvv1:eq18}{25} \\
     $\xi_k^{(l,\prime)}$ ~ \hyperlink{xikell:eq29}{36}  &  $X_k^{\prime}$ ~ \hyperlink{Xkprime:eq29}{36} \\
      $M_k$ ~ \hyperlink{bigMk:eq10}{10} &  $\mathcal{E}_k^{**}$ ~ \hyperlink{EEKK:eq15}{22} \\
      $ \Delta_{j,m}(x)$ ~ \hyperlink{DeltaJMX:eq24}{31}, \hyperlink{DeltaJMX:eq26}{33}, \hyperlink{DeltaJMX:eq27}{34} & $\mathcal{E}_{i,j}$ ~ \hyperlink{sceij:eq29}{36} \\
      $X_{k,m}$ ~ \hyperlink{bigXkm:eq29}{36} & $\theta_{k,p}$ ~ \hyperlink{thetaKPP:eq29}{36}, \hyperlink{thetaKPP:eq32}{39} \\
      $\dot{S}_n$ ~ \hyperlink{dotSn:eq30}{37} & $\breve M_{n,l}$ ~ \hyperlink{breMNL:eq31}{38} \\
      $ \Theta_{j,p} $ ~ \hyperlink{Thetajp:eq29}{36}  &  $\dot{M}_{i,l}$ ~ \hyperlink{dotMIL:eq30}{37} \\
      $\breve \mu_{l}$ ~ \hyperlink{bremuL:eq31}{38} &  ${X}_k^{\diamond}$ ~ \hyperlink{XXKdm:eq35}{42} \\
      $S_n^{\diamond}$ ~ \hyperlink{SNDM:eq35}{42} &   $R_j$ ~ \hyperlink{bigrj:eq13}{20}   \\
       $\Delta_{j,m}^{\diamond}$ ~ \hyperlink{deltaDM:eq36}{43} & $X_{k,m}^{*}$ ~ \hyperlink{XKM:eq41}{48} \\
       $\Psi_{n,m}(x)$ ~ \hyperlink{Psinmx:eq42}{49} & $\Delta_n^{T}(x)$ ~ \hyperlink{deltaNT:eq43}{50} \\
        $\varphi_{j|m}(x)$ ~ \hyperlink{phiJCM:eq33}{40} & $\bar{\sigma}_{j\vert m}$ ~ \hyperlink{sigmaJCM:eq13}{20} \\
       $\Phi(\cdot)$ ~ \hyperlink{Phi:eq5}{7} &  $\mathcal{H}_{L}^s$ ~ \hyperlink{schls:eq9}{14}  \\
       $\ss_m^2 $ ~ \hyperlink{ssm2:eq12}{19}  &  $\MM_n$ ~ \hyperlink{MMn:eq10}{10} \\
       $\mathds{G}_n$ ~ \hyperlink{GGn:eq11}{10} & $e_j$ ~ \hyperlink{littleek:eq12}{19} \\
       $\overline{S}_{n\vert m}$ ~ \hyperlink{barSNM:eq12}{20} & $\widetilde{S}_{n\vert m}$ ~ \hyperlink{tilSNM:eq12}{20} \\
       $U_j$ ~ \hyperlink{biguj:eq13}{20} &   $\Delta_n^{\diamond}(x)$ ~ \hyperlink{deltandia:eq36}{43}, \hyperlink{deltandia:eq37}{44}, \hyperlink{deltandia:eq38}{45}, \hyperlink{deltandia:eq40}{47} \\
       $\bar{\kappa}_{j\vert m}$ ~ \hyperlink{kappaJM:eq13}{20}  & $\bar{\nu}_{j\vert m}$ ~ \hyperlink{nuJM:eq13}{20} \\
       $\overline{Z}$ ~\hyperlink{barZ:eq14}{21} & $\widetilde{Z}$ ~\hyperlink{tildeZ:eq14}{21} \\
           \hline
\end{tabular}
\end{table}
\end{center}

\newpage

\begin{center}
\begin{table}\label{table:notation}
\begin{tabular}{l | l}

\hline
        $\alpha$ ~ \hyperlink{ALPHA:eq11}{10}  \quad \qquad \qquad \qquad \qquad  \qquad  \qquad  \qquad  &  $\beta$ ~ \hyperlink{BETA:eq11}{10}  \qquad \qquad \qquad   \qquad \qquad \qquad \qquad  \\
        $\lambda_{l,p}$ ~ \hyperlink{lambdaLPP:eq13}{11}    & $\md$ ~ \hyperlink{darkm:eq15}{21}  \\
        $v_n$ ~ \hyperlink{VN:eq18}{24}  &  $g(x)$ ~ \hyperlink{GX:eq26}{32} \\
        $\ad$ ~ \hyperlink{BAD:eq5}{6}, \hyperlink{BAD:eq6}{7} &  $U_k$ ~ \hyperlink{USUBK:eq7}{8}  \\
        $\overline{W}_j$ ~ \hyperlink{OWJ:eq28}{34} &  $F(x)$ ~ \hyperlink{FXX:eq29}{35}  \\
        $\overline{Z}_j$ ~ \hyperlink{OZJ:eq29}{35} & $X_k^{(l,\prime)}$ ~ \hyperlink{XKLP:eq30}{36} \\
        $L_n$ ~ \hyperlink{LSUBN:eq11}{10}   &  $\tau_n$ ~ \hyperlink{TAUSUBN:eq11}{10}  \\
         $\dot{X}_i$ ~ \hyperlink{dotXi:eq30}{37} &  $\hat{g}_{a,b}(t)$~\hyperlink{ghatab:eq35}{42}    \\
       $G_{a,b}$ ~ \hyperlink{Gab:eq35}{42}  &  $\breve v_{l}$ ~\hyperlink{brevenul:eq31}{38} \\
        $\dot{\Theta}_{n,q}$ ~ \hyperlink{Thetanq:eq30}{37}  &    $\dot{S}_{n,m}$ ~ \hyperlink{dotSnm:eq30}{37} \\
        $[s]$ ~ \hyperlink{smid:eq10}{14}   &  $\varphi_H(t)$ ~ \hyperlink{phiHt:eq7}{8} \\
                $X_k$ ~ \hyperlink{XSUBK:eq5}{6}  &   $\mathds{V}_n$ ~ \hyperlink{VVN:eq11}{9}  \\
                $\xi_{k}^{(l,*)}$ ~ \hyperlink{xiklstar:eq13}{11}, \hyperlink{xiklstar:eq15}{22} & $X_k^{(l,*)}$ ~ \hyperlink{xkell:eq15}{22} \\
       $h_n(x)$ ~ \hyperlink{hnx:eq24}{31} &  $A_k$ ~ \hyperlink{bigAk:eq28}{35} \\
  \hline
\end{tabular}
\end{table}
\end{center}

\section{Proofs of Section \ref{sec:applicationII}}\label{sec:proof:applicationsII}

\begin{proof}[Proof of Theorem \ref{thm:high:linear:stat}]
Due to Theorem \ref{thm_edge_four}, it suffices to establish
\begin{align}\label{eq:thm:high:lin}
\Cd_{T_n} \lesssim n^{-(\alpha +1)/2}.
\end{align}

Observe that as in the proof of Theorem \ref{thm_edge_four}, we may assume that $(X_k)_{k \in \Z}$ is $m$-dependent with $m = n^{\mathfrak{m}}$, $\mathfrak{m} < 1$. By independence and $|e^{\mathrm{i}x}|= 1$, we have
\begin{align*}
\Big|\E e^{\mathrm{i} \xi S_n/\sqrt{n}} \Big| \leq \Big|\E e^{\mathrm{i} \xi \sum_{i \in \mathcal{I}} S_{ni}/\sqrt{n}} \Big|. 
\end{align*}
By Burkholder's inequality and Assumption \ref{ass:tail}, we have
\begin{align*}
\big\|\sum_{i \in \mathcal{I}} \theta_i (X_{ki} - X_{ki}^{\ast})\big\|_p^2 \lesssim \sum_{i \in \mathcal{I}} \theta_i^2 \big\|X_{ki} - X_{ki}^{\ast} \big\|_p^2.
\end{align*}
Consider now the sum $S_{n\mathcal{I}} = \sum_{k = 1}^n \sum_{i \in \mathcal{I}} \theta_i X_{ki}/\Theta_{\mathcal{I}}$, where $\Theta_{\mathcal{I}}^2 = \sum_{i \in \mathcal{I}} \theta_i^2$. By the above, we may apply Lemma \ref{lem_bound_prod_cond_char}, which yields
\begin{align*}
&\Big|\E e^{\mathrm{i} \xi \sum_{i \in \mathcal{I}} S_{ni}/\sqrt{n}} \Big| = \Big|\E e^{\mathrm{i} \xi \Theta_{\mathcal{I}}  S_{n\mathcal{I}}/\sqrt{n}} \Big| \\&\lesssim e^{-\co_{\varphi} \xi^2} + e^{-\sqrt{N/32}\log 8/7}, \quad \xi^2 \Theta_{\mathcal{I}}^2 \leq {\co_T} n
\end{align*}
for constants $\co_{\varphi}, \co_T > 0$. Due to  \hyperref[T3]{\Tthree}, this in turn yields that for some constant $\co > 0$
\begin{align*}
\Cd_{T_n} \lesssim e^{-\co \log^{\beta} n}  + e^{-\sqrt{N/32}\log 8/7} \log n +  n^{-(\alpha +1)/2} \lesssim n^{-(\alpha + 1)/2},
\end{align*}
and hence the validity of \eqref{eq:thm:high:lin}.
\end{proof}

\section*{Acknowledgements}

We would like to thank the reviewer for a careful reading
of the manuscript and valuable remarks. Moreover, we thank
Kasun Fernando Akurugodage, Christophe Cuny and Florence Merlev\`{e}de for
insightful comments.

\end{document}